\documentclass[oneside, 10pt]{amsart}   	

\usepackage{fullpage}
\usepackage{graphicx}				
\usepackage{color}
\usepackage[dvipsnames,svgnames,table]{xcolor}

\usepackage[utf8]{inputenc}
\usepackage[parfill]{parskip}
\usepackage[colorlinks=true, pdfstartview=FitV,linkcolor=ForestGreen,citecolor=ForestGreen, urlcolor=black]{hyperref}

\usepackage{amssymb,amsmath,amsfonts,amsthm,amscd}
\usepackage{esint}
\usepackage{mathrsfs}
\usepackage{enumerate}
\usepackage{cite}
\usepackage{bbm}
\usepackage{comment}
\usepackage{url} 
\usepackage{tikz}

\newcommand{\abs}[1]{\left\vert#1\right\vert}
\newcommand{\norm}[1]{\left\Vert#1\right\Vert}  

\newcommand{\R}{\ensuremath{{\mathbb R}}}

\newcommand{\eps}{\varepsilon}

\newcommand{\dd}{\, {\rm d}}
\newcommand{\dv}{\, \mathrm{d} v}
\newcommand{\dx}{\, \mathrm{d} x}
\newcommand{\dt}{\, \mathrm{d} t}
\newcommand{\dsigma}{\, \mathrm{d} \sigma}

\newcommand{\un}{\mathbf{1}}

\newcommand{\fin}{f_{\mathrm{in}}}
\newcommand{\fb}{f_b}
\newcommand{\phizk}{{\varphi_{0,\kappa}}}
\newcommand{\dphizk}{{\dot{\phizk}}}
\newcommand{\ddphizk}{{\ddot{\phizk}}}
\newcommand{\Phizk}{{\Phi_{0,\kappa}}}

\newcommand{\Etrois}{{\mathcal{E}_4}}
\newcommand{\Equatre}{\mathcal{E}_3}

\newcommand{\bast}{\mathfrak{b}_\ast}
\newcommand{\ma}{\mathfrak{a}}
\newcommand{\fraka}{\mathfrak{a}}

\newcommand{\Cb}{C_b}
\newcommand{\Ccor}{C_{\mathrm{cor\,}}}
\newcommand{\Ccorbis}{C_{\mathrm{cor'\,}}}
\newcommand{\Cget}{C_{\mathrm{get\,}}}

\newcommand{\Cun}{C_1}
\newcommand{\Chard}{C_{\mathrm{hard\,}}}
\newcommand{\Cconv}{C_{\mathrm{conv\,}}}

\newcommand{\aast}{a_\ast}
\newcommand{\aget}{a_{\mathrm{get\,}}}

\newcommand{\qnsl}{q_{\mathrm{nsl\,}}}

\newcommand{\Cqu}{C_{q,1}}
\newcommand{\Cqd}{C_{q,2}}
\newcommand{\Cqt}{C_{q,3}}

\newcounter{num} \numberwithin{num}{section}
\numberwithin{equation}{section}

\newtheorem{theorem}[num]{Theorem}
\newtheorem{proposition}[num]{Proposition}
\newtheorem{lemma}[num]{Lemma}

\theoremstyle{definition}
\newtheorem{definition}[num]{Definition}

\theoremstyle{remark}
\newtheorem{remark}[num]{Remark}

\title[Conditional appearance of decay for non-cutoff Boltzmann]{Conditional appearance of decay for the non-cutoff Boltzmann equation in a domain}
\author{Cyril Imbert \& Amélie Loher}
\date{\today}

\address{Département de Mathématiques et Applications, ENS Rue d'Ulm, Paris (FRANCE)}
\email{cyril.imbert@ens.fr}
\address{Department of Pure Mathematics and Mathematical Statistics, University of Cambridge (UK)} 
\email{ajl221@cam.ac.uk}

\thanks{The authors would like to thank C.~Mouhot for several useful comments on a preliminary version of this work. They also thank all referees for the numerous comments and the careful reading of a previous version of this article.}
\thanks{AL acknowledges funding from the Cambridge Trust.}

\begin{document}

\begin{abstract}
This work is concerned with the generation of decay estimates in the velocity variable for solutions of the space-inhomogeneous Boltzmann equation without cutoff on a bounded spatial domain for hard and moderately soft potentials. We work with  suitable weak solutions, provided that mass, energy and entropy density functions are under control. The following boundary conditions are treated: in-flow, bounce-back, specular reflection, diffuse reflection and Maxwell reflection. The notion of weak solutions relies on a family of Truncated Convex Inequalities that is inspired by the one recently introduced through F.~Golse, L.~Silvestre and the first author (2023) in the spatially homogeneous case. 
We show that the solutions generate some amount (up to $d+1$) of pointwise polynomial velocity decay. In case of moderately soft potentials, we show that it is not possible to generate a decay higher than $d+2$ if the energy is bounded.  
\end{abstract}



\maketitle

\tableofcontents

\section{Introduction}

\subsection{The Boltzmann equation}

We consider the Boltzmann equation posed in a bounded $C^{1}$ domain $\Omega \subset \R^d$ (with $d \ge 2$),
\begin{equation}
  \label{eq:boltzmann}
	(\partial_t  + v \cdot \nabla_x) f = Q(f, f), \qquad t \in (0,T), ~ x \in \Omega, ~ v \in \R^d
\end{equation}
where the unknown function $f=f(t,x,v)$ is non-negative; $Q(f,f)$ denotes the Boltzmann collision operator, 
\begin{equation}
  \label{eq:Q}
  Q(f, f)(v) = \iint_{{\mathbb S}^{d-1} \times \R^d}   \big( f(v_*')f(v') - f(v_*)f(v)\big) B(|v - v_*|, \sigma) \dd \sigma \dd v_*
\end{equation}
where $v'$ and $v'_*$ are given by
\begin{equation*}
	v' = \frac{v + v_*}{2} + \frac{\abs{v- v_*}}{2}\sigma, \qquad v'_* =  \frac{v + v_*}{2} - \frac{\abs{v- v_*}}{2}\sigma
\end{equation*} 
and the collision kernel $B$ satisfies
\begin{equation*}
	B\big(\abs{v-v_*}, \sigma \big) = \abs{v-v_*}^\gamma b(\cos \theta) \quad \text{ with }\quad \cos\theta = \frac{v-v_*}{\abs{v-v_*}} \cdot \sigma
\end{equation*} 
for $\gamma \in (-d, 1]$, and 
\[
	b(\cos \theta) \approx \abs{\sin \theta}^{-d+1 -2s} 
\]
for $s \in (0, 1)$.
This means that there exist  constants $C_\pm$ such that $0 < C_- \abs{\sin \theta}^{-d+1-2s} \le b(\cos \theta) \le C_+ \abs{\sin\theta}^{-d+1-2s}$. Without loss of generality,
we can also assume that for $\cos \theta <0$, we have
\[
	b(\cos \theta)= \abs{v-v'}^{-(d-1) -2s} \abs{v - v_*}^{d-2 - \gamma} \abs{v-v_*' }^{\gamma + 2s + 1} \tilde b(\cos \theta),
\]
where $\tilde b(\cos \theta)$ is such that $0 < \tilde b_- < \tilde b(\cos \theta) < \tilde b_+$ for constants $\tilde b_\pm > 0$, see \cite{MR3551261}.
We are concerned with the case when $\gamma+2s > 0$ corresponding to so-called moderately soft potentials.

The equation is supplemented with
an initial condition $f(0,x,v) = \fin (x,v)$ for $(x,v) \in \Omega \times \R^d$ and a boundary condition. In order to describe the latter, we let $n$ denote the outward unit normal vector on the boundary and we denote by $\Gamma$ the domain that considers boundary points in space, that is $\Gamma :=  \partial \Omega \times \R^d$, by $\Gamma_-$ the incoming part of the boundary, and by $\Gamma_+$ the outgoing part of the boundary, that is
\[\Gamma_- := \{( x, v) \in \Gamma :  v \cdot n(x) < 0\} \quad \text{ and }  \Gamma_+ := \{( x, v) \in \Gamma :  v \cdot n(x) > 0\}.\]
We then supplement the Boltzmann equation \eqref{eq:boltzmann} with boundary conditions that are commonly considered in the literature,
\begin{enumerate}[(i).]
\item {\sc In-flow}: $f(t, x, v) \vert_{\Gamma_-} = \fb(t, x, v)$ for a given function $\fb$.
\item {\sc Bounce-back}: $f(t, x, v) \vert_{\Gamma_-} = f(t, x, -v)$. 
\item {\sc Specular reflection}: $f(t, x, v) \vert_{\Gamma_-} = f(t, x, \mathcal R_x v)$  with $\mathcal R_x v = v - 2 (v \cdot n) n$.
\item {\sc Diffuse reflection}: $$f(t, x, v)\vert_{\Gamma_-} = c_\mu (x) \mu(v) \int_{\Gamma_+} f(t, x, v')(v' \cdot n (x))\dd v',$$ where $\mu(v) = e^{-\abs{v}^2}$  and $c_\mu=c_\mu (x)$ is a normalisation constant, such that $c_\mu \int_{\Gamma_+} \mu(v')   (v' \cdot n) \dd v'=1$.
\item {\sc Maxwell reflection}: $$f(t, x, v)\vert_{\Gamma_-} = (1-\iota(x)) f(t, x, \mathcal R_x v) + \iota (x) c_\mu \mu(v) \int_{\Gamma_+} f(t, x, v')(v' \cdot n )\dd v',$$ where $\iota: \partial \Omega \to [0,1]$ is the accomodation coefficient. 
\end{enumerate}
The Maxwell reflection is a convex combination of specular and diffuse reflections. It is particularly relevant from the physical point of view: a part of the particles is reflected and another part is ``absorbed by the wall and re-emitted according to a Gaussian distribution [\dots], corresponding to a thermodynamical equilibrium between particles and the wall'', \cite{MR1942465}.

\subsection{Hydrodynamical quantities}

L.~Silvestre \cite{MR3551261} showed that when some hydrodynamical quantities are under control, an a-priori solution of the Boltzmann equation is $L^\infty((0, T)\times \Omega\times\R^d)$, and thus the non-linear equation can be viewed as a linear equation with \textit{rough} coefficients. The roughness of the coefficients can be quantified, and therefore, there is a notion of ellipticity in \eqref{eq:boltzmann}. Throughout this work, we assume that there exist positive constants $m_0,M_0,E_0,H_0$ such that for almost every $(t,x) \in (0,T) \times \Omega$, the function $f(v) = f(t,x,v)$ satisfies,
\begin{equation}\label{e:hydro}
  m_0 \le \int_{\R^d} f(v) \dv \le M_0, \quad \int_{\R^d} f(v) |v|^2 \dv \le E_0, \quad \int_{\R^d} f(v) \ln f (v) \dv \le H_0.
\end{equation}
In particular, the entropy production estimate yields some integrability in $(t,x,v)$ with a negative weight in velocity \cite[Theorem~0.1]{chaker2022entropy}. In case of the spatially homogeneous equation \eqref{eq:boltzmann}, when $f = f(t, v)$ does not depend on $x$, the conditions in \eqref{e:hydro} reduce to assumptions on the initial datum, since the quantities in \eqref{e:hydro} are conserved by the equation \eqref{eq:boltzmann}. 

\subsection{The collision operator}

In order to present the Truncated Convex Inequalities satisfied by our weak solutions, it is necessary to recall some facts about the collision operator and its kernel representation. 

Using Carleman coordinates, the collision operator can be written as follows \cite{MR1942465} -- see also \cite{MR3551261}, 
\[
	Q(f, f) (v) = \int_{\R^d} \left[ f(v') K_f(v, v') - f(v) K_f(v', v)\right] \dd v',
\]
where the kernel $K_f(v, v')$ is given by 
\begin{equation*}
\begin{aligned}
	K_f(v, v') &= 2^{d-1} \abs{v'-v}^{-1} \int_{w \perp v' -v} f(v+w) B(r, \cos \theta) r^{-d+2}\dd w \\
	&=2^{d-1} \abs{v'-v}^{-(d+2s)} \int_{w \perp v' -v} f(v+w)  \abs{w}^{\gamma +2s +1}\tilde b(\cos \theta) \dd w,
\end{aligned}
\end{equation*} 
for $r^2 = \abs{v'-v}^2 + \abs{w}^2$ and $\cos \theta = \frac{|w|^2 - |v-v'|^2}{|w|^2+|v-v'|^2}$. The function $\tilde b$ is bounded from above and below by positive constants. 

\subsection{Truncated Convex Inequalities}

We derive decay estimates by studying the evolution along time of some $L^{q_0}$-Lebesgue norm of the function $(f-A)_+$. To this end, we consider a general function $\varphi = \varphi (t, v,r)$ that is convex in $r$ instead of $\varphi_0((f-A)_+) = (f-A)_+^{q_0}$. We emphasise that $\varphi$ does not depend on $x$: our result is global. 
For such a $\varphi$, a formal computation yields,
\begin{align*}
  \frac{\dd}{\dd t} \iint_{\R^d \times \Omega} \varphi(t,v,f) \dv \dx
  &= \iint_{\R^d \times \Omega} \partial_r \varphi(t, v,f) Q(f, f) \dv \dx + \iint_{\R^d \times \Omega} \partial_t \varphi(t, v,f) \dv \dx\\
  &\quad  +  \iint_{\Gamma} \varphi(t, x, v, f) (v \cdot \nabla_x f) (t,x,v,f) \dv \dx \\
  &= \iint_{\R^d \times \Omega} \partial_r \varphi(t, v,f) \left\{ \int_{\R^d} \left[ f(v') K_f(v, v') - f(v) K_f(v', v)\right] \dd v'\right\} \dv \dx \\
  &\quad+ \iint_{\R^d \times \Omega} \partial_t \varphi(t, v,f) \dv \dx  -  \iint_{\Gamma} \varphi(t, x, v, f) (v \cdot n) \dv \dd S(x),
\end{align*} 
We can then add and subtract 
\[
	\mathcal D_\varphi(f, f') := \varphi(t, v', f') -\varphi(t, v,f) - \partial_r \varphi(t, v,f) \left(f' - f\right), 
\]
where we denote here and in the sequel $f' = f(v')$ and $f = f(v)$,
\begin{align*}
  \frac{\dd}{\dd t} \iint_{\R^d \times \Omega} &\varphi(t, v,f) \dv \dx \\
  =& - \iiint_{\R^d \times \R^d \times \Omega}\mathcal D_\varphi(f, f') K_f(v, v') \dd v' \dd v \dd x \\
	&+  \iiint_{\R^d \times \R^d \times \Omega}\left[ \partial_r \varphi(t, v,f) f - \varphi(t, v,f) - \partial_r \varphi(v', f') f' + \varphi(v', f')\right] K_f(v, v') \dd v' \dd v \dd x \\
	&+ \iint_{\R^d \times \Omega} \partial_t \varphi(t, v,f) \dv \dx -  \iint_{\Gamma} \varphi(f) (v \cdot n) \dv \dd S(x).
\end{align*}
If we now denote by
\begin{equation*}
  \Phi(t,v,r) := r \partial_r \varphi(t,v,r) - \varphi(t, v,r),
\end{equation*}
we get,
\begin{align*}
   \frac{\dd}{\dd t} \iint_{\R^d \times \Omega} &\varphi(t, v,f) \dv \dx \\
  =& - \iiint_{\R^d \times \R^d \times \Omega}\mathcal D_\varphi(f, f') K_f(v, v') \dd v' \dd v \dd x \\
                                               &+  \iiint_{\R^d \times \R^d \times \Omega}\left[ \Phi(f) - \Phi(f')\right] K_f(v, v') \dd v' \dd v \dd x \\
  &+ \iint_{\R^d \times \Omega} \partial_t \varphi(t, v,f) \dv \dx -  \iint_{\Gamma} \varphi( f) (v \cdot n) \dv \dd S(x).
 \end{align*}
Due to the cancellation lemma \cite{ADVW} (recalled below -- see \eqref{eq:cancellation}), we are thus led to,
\begin{align*}
  \frac{\dd}{\dd t} \iint_{\R^d \times \Omega} \varphi(t, v,f) \dv \dx
  =& - \iiint_{\R^d \times \R^d \times \Omega}\mathcal D_\varphi(f, f') K_f(v, v') \dd v' \dd v \dd x \\
  & + c_b  \iint_{\R^d \times \Omega} \Phi(f) \left(f \ast \abs{\cdot}^\gamma\right) \dd v \dd x \\
	&+ \iint_{\R^d \times \Omega} [\partial_t \varphi](t, v, f) \dv \dx -  \iint_{\Gamma} \varphi (f) (v \cdot n) \dv \dd S(x).
\end{align*}
We can now use the special structure of the function $\varphi$ that we are going to use. More precisely,
\[ \varphi (t,v,r) = \bar \varphi ((r-A(t,v))_+) \]
for some convex function $\bar \varphi$ vanishing at $0$ together with its derivative $\dot {\bar \varphi}$. 
Then we compute,
\[ \Phi (f) = \bar \Phi ((f-A)_+) \]
with $\bar \Phi (r) = r \dot {\bar \varphi} - \bar \varphi$ and
\begin{equation}\label{eq:D}
\mathcal{D}_\varphi (f,f') = d_{\bar \varphi} ((f-A)_+,(f'-A')_+) + \dot {\bar \varphi} ((f-A)_+) (A'-f')_+ - \dot {\bar \varphi} ((f-A)_+) (A'-A),
\end{equation}
where
\begin{multline}\label{eq:d}
	d_{\bar \varphi} ((f-A)_+,(f'-A')_+) = \bar\varphi((f'-A')_+) - \bar \varphi((f-A)_+) - \dot{ \bar \varphi}((f-A)_+) \big((f'-A')_+-(f-A)_+\big).
\end{multline}
We remark that the first two terms in $\mathcal D_{\varphi}$ are non-negative, while the third one is an error that we will have to handle. 

\subsection{Suitable weak sub-solutions}

In order to define suitable weak solutions, we consider a family of convex functions $ \varphi_a (r) = (r-a)_+$ associated with $a >0$. These functions are commonly used in the theory of
entropy solutions for scalar conservation laws and are known as \emph{Kruzhkov's semi-entropies}, see for instance \cite{MR1425004,MR2083859}.
In this case,
\begin{equation}\label{e:dphia}
  d_{\varphi_a} (r,s) = \begin{cases} (s-a)_+ & \text{ if } r \le a \\ (a-s)_+ & \text{ if } r > a \end{cases}
\end{equation}
and $\Phi_a (r) = a \un_{\{ f > a \}}$.

The upper bound on the mass density -- see condition~\eqref{e:hydro} -- and the bound on the time-integrated entropy production (see Theorem~\ref{t:entropy})
suggest to consider solutions that are $L^{q_0}$ in all variables with $q_0 = 1 +\frac{2s}d$, see Lemma~\ref{l:integrability}. 

In accordance with the formal computation that we performed above, we introduce the following notion of weak sub-solutions for the spatially inhomogeneous Boltzmann equation.
\begin{definition}[Suitable weak sub-solutions] \label{defi:suitable}
  Let $T \in (0,+\infty]$ and let $\Omega$ be a $C^{1}$ domain of $\R^d$.
  A non-negative function $f \in L^1 ((0,T) \times \Omega \times \R^d)$ is a \emph{suitable weak sub-solution} of \eqref{eq:boltzmann} if 
\begin{enumerate}[(i).]
\item {\sc in-flow}: given $f_b \colon (0,T) \times \Gamma_- \to [0,+\infty)$,
  for any real number $a>0$ and any  function $A \colon (0,T)  \times \R^d \to (0,+\infty)$ such that $\partial_t A$ and $\partial^2_{v_i,v_j} A$ exist and are bounded continuous in $(0,T) \times \R^d$,
  there holds in $\mathcal{D}'((0,T))$, 
  \begin{equation}
\label{e:convex-family}
\left\{
  \begin{aligned}
  \frac{\dd}{\dd t}\iint_{\R^d \times \Omega} &\varphi_a (f -A)  \dv \dx  \\
  &+ \iiint_{\R^d \times \R^d \times \Omega} d_{\varphi_a} ((f-A)_+,(f'-A')_+) K_f (v,v') \dv \dv' \dx \\
  &+ \iiint_{\R^d \times \R^d \times \Omega} \dot \varphi_a (f-A) (A'-f')_+ K_f (v,v') \dv \dv' \dx \\
  &\le  \iint_{\R^d \times \Omega} \bigg\{ c_b \bigg( \Phi_a (f-A) + A \dot \varphi_a (f-A)\bigg)(f \ast_v |\cdot|^\gamma) - \dot \varphi_a (f-A) \partial_t A  \bigg\} \dv \dx \\
  & + \iiint_{\R^d \times \R^d \times \Omega}\dot \varphi_a (f-A) (A'-A) K_f (v,v') \dv \dv' \dx -  \iint_{\Gamma_-} \varphi_a (\fb-A) (v \cdot n) \dv \dd S(x) 
\end{aligned}
\right.
\end{equation}
with $\varphi_a (r) = (r-a)_+$, $\Phi_a (r) = a \un_{\{ f > a \}}$ and $d_{\varphi_a}$ given by \eqref{e:dphia}.
\item  {\sc bounce-back}: it satisfies  \eqref{e:convex-family} with $\fb=0$.  
\item  {\sc specular reflection}: it satisfies \eqref{e:convex-family} with $\fb=0$. 
\item  {\sc diffuse reflection}: it satisfies \eqref{e:convex-family} with $\fb(t, x, v) := c_\mu (x)\mu(v) \int_{\Gamma_+} f(t, x, v') (v'\cdot n) \dd v'$.
\item {\sc Maxwell reflection}:  it satisfies \eqref{e:convex-family} with $\fb(t, x, v) := \iota (x)c_\mu(x) \mu(v) \int_{\Gamma_+} f(t, x, v') (v'\cdot n) \dd v'$. 
     \end{enumerate}
\end{definition}
\begin{remark}[General convex functions]\label{r:more-convex}
  The family of inequalities \eqref{e:convex-family} is only imposed for elementary (non-decreasing non-negative) convex functions $\varphi_a (r) = (r-a)_+$. But this implies that such inequalities hold true for general Lipschitz convex functions $\varphi$ such that $\varphi (0) = \dot \varphi (0) = 0$  by simply integrating in the parameter $a$. As already observed in \cite{gis}, a general $C^2$ convex function $\varphi$ satisfies,
  \[ \varphi (r) = \varphi (0) + \dot \varphi(0)r + \int_0^{+\infty} \ddot\varphi(a) (r-a)_+ \dd a, \]
  and we can easily check that
  \[ \Phi (r) = \int_0^{+\infty} \ddot \varphi(a) \Phi_a (r) \dd a \quad \text{   and } \quad 
    d_\varphi (r,\rho) = \int_0^{+\infty} \ddot \varphi (a) d_{\varphi_a} (r,\rho) \dd a.\]
  We will follow this idea through a truncation procedure (see Section~\ref{s:propagation}). 
\end{remark}
\begin{remark}[Link with renormalised solutions]\label{r:renormalized}
  R.~J.~DiPerna and P.-L.~Lions \cite{MR1014927} constructed weak solutions of the Boltzmann equation in the cutoff case
  by considering the (Lipschitz) concave non-linear change of variables $\varphi (f)$ with $\varphi (r) = \ln (1+r)$.
  We consider here a different type of generalized solutions, in particular we compose solutions with (Lipschitz) convex functions $\varphi$.  
\end{remark}
\begin{remark}[Positive terms]
  The first positive term in the left hand side of \eqref{e:convex-family} corresponds to the classical entropy production term.
  The second one is reminiscent of the ``good extra term'' first exhibited by L.~Caffarelli, C.~H.~Chan and A.~F.~Vasseur in \cite{MR2784330} for parabolic equations, and by the second author in \cite{L2024} for kinetic equations. It also plays a crucial role in the recent work by Z.~Ouyang and L.~Silvestre \cite{ouyang-silvestre}.
\end{remark}
\begin{remark}[Error terms]
  \label{r:integrability-rhs}
  Right-hand sides of time differential inequalities are well-defined and integrable in time.
  Indeed, $\dot \varphi_a$ and $\Phi_a$ are bounded and $f \ast_v |\cdot|^\gamma$ is integrable in $(t,x,v)$. The term involving the difference $A'-A$ comes from the dependence of the barrier function $A$ on the velocity variable. It did not appear in the work of Q.~Ouyang and L.~Silvestre \cite{ouyang-silvestre}. It is reminiscent of the ``bad terms'' that were treated by C.~Mouhot, L.~Silvestre and the first author in \cite{MR4033752}. 
\end{remark}

\subsection{Main result}

\begin{theorem}[Conditional decay estimates]\label{thm:decroissance}
Let the parameters $\gamma \in (-2s, 1]$ and $s \in (0, 1)$ of the non-cutoff collision kernel $B$ satisfy $\gamma + 2s \in (0,1]$.

Let $T \in (0,+\infty]$ and $f$ be a suitable weak sub-solution of the Boltzmann equation \eqref{eq:boltzmann} in $(0, T) \times \Omega \times \R^{d}$ with either in-flow, bounce-back, specular / diffuse / Maxwell reflection boundary conditions. We assume that  the entropy production $D(f)$ lies in $L^1 ((0,T) \times \Omega)$ and its mass, energy and entropy density functions satisfy Condition~\eqref{e:hydro}.
Let
  \[
    \qnsl = \begin{cases} d +1 +  \frac{ d }{2s}\min (\gamma,0),\qquad &\textrm{ if } -\frac{2sd}{d+2s} \leq \gamma \leq 1,\\
      d+1 -\frac{d}{d+2s}(2-\gamma) ,\qquad\qquad &\textrm{ if } -2s < \gamma < -\frac{2sd}{d+2s}.
    \end{cases}
  \]
Then, for any $q \in [0, \qnsl]$ there exist $C_q >0$, depending on $m_0,M_0, E_0, s, d, \gamma$ and $q$,  such that
\[
  f(t, x, v) \leq C_q \left(1 + t^{-\frac{d}{2s}}\right) \langle v \rangle^{-q} \quad  \text{a.e. in } (0, T) \times \Omega \times \R^{d}.
\]
In case of the in-flow boundary condition, we need to additionally assume that there exists $C_b > 0$ such that the boundary data $f_b$ satisfies
\[ \fb (t,x,v) \le \Cb \left(1 + t^{-\frac{d}{2s}}\right) \langle v \rangle^{-q} \quad \text{a.e. in } (0,T) \times \Gamma_-.\]
In this case, $C_q$ depends also on $C_b$. 
\end{theorem}
\begin{remark}[Boundary condition]
  We draw the attention of the reader towards the fact that a condition on the boundary data $f_b$ is necessary only in the in-flow case. For  bounced-back and specular reflection, we just have $f_b =0$. In the case of diffuse and Maxwell reflections, the function $f_b$ decays exponentially
  fast in the velocity variable and the constant in front of the Maxwellian is bounded due to the bounds on mass, energy and entropy density functions. 
\end{remark}
\begin{remark}[Assumptions]
    The assumptions on the entropy production (whose definition is recalled in \eqref{eq:entropy-prod} below) is natural in view of Boltzmann's H-theorem. It will imply some integrability on solutions (see next remark).
  The assumption on density functions was first considered in \cite{MR3551261}, it ensures that Boltzmann's collision operator is (weakly) elliptic. As far as this work is concerned, it will ensure that some coercivity estimates hold. 
\end{remark}
\begin{remark}[Integrability of solutions]\label{rmk:integrability}
The integrability of these suitable weak sub-solutions to \eqref{eq:boltzmann} stems from the entropy production estimate {\cite[Theorem~0.1]{chaker2022entropy}}: it implies that $f\in L^1((0, T)\times \Omega; L^{p_0}_{k_0}(\R^d))$ for $\frac{1}{p_0} = 1 -\frac{2s}{d}$ and $k_0 = \gamma + 2s - \frac{2s}{d}$. The definition of weighted Lebesgue spaces is recalled below in the paragraph dedicated  to notation. Consequently, when combined with the hydrodynamical bounds \eqref{e:hydro}, we obtain that the functions $f$ we work with are such that  $f\in L^{q_0}_{k_0^*}(0, T)\times \Omega\times\R^d)$ with $q_0 = \frac{d+2s}{d}$ and some $k_0^* \in (0, 2)$. 
\end{remark}
\begin{remark}[Large times]
  We emphasise that the estimates hold true uniformly in time, even in the case  $T=+\infty$. This is important in the conditional regularity
  programme. Obtaining bounds that are uniform in time is expected to imply that the large time behaviour can be studied in the class of regular solutions,
  see in particular the work by L.~Desvillettes and C.~Villani \cite{MR2116276} for bounce-back and specular reflection boundary conditions. 
\end{remark}

Whereas Theorem \ref{thm:decroissance} shows that independent of the initial data (or the boundary data), solutions to the Boltzmann equation generate some amount of decay, we can also show that there is no hope of generating an arbitrary amount of polynomial velocity decay in case of soft potentials ($\gamma < 0$).
\begin{theorem}[Propagation of a polynomial lower bound]\label{thm:no-gen}
Let the parameters $\gamma \in (-d, 0]$ and $s \in (0, 1)$ of the non-cutoff collision kernel $B$ satisfy $\gamma + 2s \in (0,1]$.

Let $T \in (0,+\infty)$ and $f$ be a non-negative classical  solution of the Boltzmann equation \eqref{eq:boltzmann} in $[0, T] \times \Omega \times \R^{d}$ with in-flow boundary condition. Let $f$ satisfy condition~\eqref{e:hydro} about mass, energy and entropy density functions.

Let $\hat q > d+2$. Let $\fin \geq 0$ and $f_b \geq 0$ denote the initial data and the boundary data. Assume there exist $a_0,a_b > 0$ such that 
\begin{equation}\label{eq:f0-lb}
	\fin(x, v) \geq a_0\langle v \rangle^{- \hat q}, \qquad \forall (x, v) \in \Omega \times \R^d, 
\end{equation}
and such that
\begin{equation}\label{eq:fb-lb}
	f_b(t, x, v) \geq a_b\langle v \rangle^{- \hat q}, \qquad \forall (t, x, v) \in [0, T]\times \partial \Omega \times \R^d.
\end{equation}

Then there exists $a_1 > 0$ and $\beta > 0$ such that
\[
	f(t, x, v) \geq a_1 e^{-\beta t} \langle v \rangle^{-\hat q}, \qquad \forall (t, x, v) \in [0, T]\times  \Omega \times \R^d.
\]
\end{theorem}
\begin{remark}[No generation for soft potentials]
  In particular, Theorem \ref{thm:no-gen} excludes the possibility of generating moments in case of soft potentials, that is $\gamma < 0$. The fact that $\hat q > d+2$ comes from the conditional regime \eqref{e:hydro}: finite energy cannot be guaranteed if $\fin$ satisfies \eqref{eq:f0-lb} for $\hat q \leq d+2$. If $f$ has $\ell \geq 2$ bounded higher velocity moments for almost every $(t, x) \in [0, T]\times \Omega$, then $\hat q$ in Theorem \ref{thm:no-gen} would need to be larger than $d + \ell$, that is $\hat q > d + \ell$. In fact, our work shows that if more velocity moments are bounded, then we can also generate more pointwise velocity decay. This is particularly interesting in the spatially homogeneous case where $L^1$ moments are propagated, so that higher moments is merely an assumption on the initial datum. For the sake of brevity, we do not elaborate further on higher moments in this work.
\end{remark}

\subsection{Comments}

This work forms part of a large body of literature dealing with velocity decay of solutions to kinetic equations. Before reviewing the literature, we make several comments about our main theorem and its proof. 
Deriving decay estimates, such as pointwise or moment bounds, is a classical theme in kinetic theory; it is key if the conditional regularity programme by L.~Silvestre and the first author (see \cite{MR4433077} and also \cite{MR4195746}) is to be extended to the case of domains. 

\medskip

\paragraph{\sc Conditional decay estimates.}
The conditional decay estimates
that we obtain for the space-inhomo\-geneous Boltzmann equation in the non-cutoff case are to be compared with the ones obtained by C.~Mouhot, L.~Silvestre and the first author \cite{MR4033752}. These results apply to classical solutions: using a barrier argument and a maximum principle they show that $f$ decays at a certain algebraic rate in velocity.  

We extend this result in several directions. First, the Boltzmann equation is posed in a domain supplemented with physically meaningful boundary conditions. Second, we consider weak (sub-)solutions instead of classical ones. We generate the same decay as they do, apart for very negative $\gamma < -\frac{2sd}{d+2s}$: in \cite{MR4033752}, their $\qnsl = d+1+\frac{d}{2s} \min(\gamma, 0)$ for $\gamma \in (-2s, 1)$, whereas we have a slightly weaker $\qnsl$ in case that $\gamma < -\frac{2sd}{d+2s}$.

\medskip

\paragraph{\sc Continuous De Giorgi's method.}
Our methods of proofs are developed after the article  by Z.~Ouyang and L.~Silvestre \cite{ouyang-silvestre} about conditional pointwise bounds of weak solutions. They  show conditional boundedness of the solutions by studying the evolution of the $L^2$-norm of the positive part of $(f-a)$ for some time-dependent function $a = a(t)$ along the flow of the equation. We point out that, in contrast with seminal works by De Giorgi and Moser, their proof does not involve any iteration procedure on the truncation parameter $a$. Instead, the truncation $a$ depends on time and the iteration is replaced with a differential inequality. In view of the article by B.~Perthame and A.~F.~Vasseur  \cite{MR2901315}, we can still look at this procedure as a legacy of De Giorgi's  ideas and call it a continuous De Giorgi's method. 

\medskip

\paragraph{\sc Truncated Convex Inequalities.}
In this work, we consider solutions with an integrability that stems from the control of the mass density and the entropy production (see Remark \ref{rmk:integrability}). We also choose the framework of weak solutions satisfying a family of inequalities \eqref{e:convex-family} associated with convex functions. A similar family of inequalities were introduced by F.~Golse, L.~Silvestre and the first author in the homogeneous case for very soft potentials \cite{gis}. It turns out that the classical notion of weak solutions constructed in \cite{MR1942465} termed $H$-solutions fit into our framework. We thus call them suitable weak sub-solutions. We emphasise that we do not use the equation but only the family of inequalities described above. It is reminiscent of the notion of De Giorgi classes from classical elliptic regularity; we refer the interested reader to \cite{zbMATH03184239}, or also to \cite{MR1465184} for a modern presentation.

\medskip

\paragraph{\sc Coercive terms.}
In the inequalities \eqref{e:convex-family}, there are coercive terms, that is to say positive terms on the left hand side; and error terms, corresponding to all terms appearing on the right hand side (some of which are positive, others have \textit{a priori} no sign). Amongst the coercive terms, we distinguish two types. 

Some are counterparts to the  $\dot{H}^1$-norm in De Giorgi's original article. They ensure that the sub-solution enjoys higher order integrability. These terms were exploited in \cite{gis} with techniques developed in \cite{MR4049224,chaker2022entropy}. They are non-local in nature, but also linear. They  rely on properties of the kernel $K$ that are derived from the hydrodynamical bounds -- see Subsection~\ref{sub:coll}. 

The second coercive term is what we call the ``good extra term'', in reference to \cite{MR2784330}, where an additional ``coercive'' term was exhibited in the non-local setting. This term was exploited for kinetic equations for the first time in \cite{L2024}. This additional term is the most important in the proof. It is a purely non-local effect.

\medskip

\paragraph{\sc Error terms.}
There are various error terms, appearing for different reasons.

There is the error coming from the dependence of the barrier $A$ on the velocity variable. Even if it resembles certain error terms in \cite{MR4033752}, we use here a truncation argument in contrast to the pointwise contact barrier argument employed by \cite{MR4033752}, and we split the collision operator differently.

Several other terms appear after using the cancellation lemma. They are non-linear in nature: they make appear a product with the convolution of $f$ with $|\cdot|^\gamma$.

Then there are error terms coming from the time dependence of the barrier and the boundary condition.
\medskip

\paragraph{\sc Generation of pointwise decay.}
Our main result asserts that, as long as mass, energy and entropy densities are under control, suitable weak sub-solutions satisfy a polynomial decay in the velocity variable up to order $q \leq \qnsl \leq d+1$. This result holds for hard ($\gamma > 0$) and moderately soft ($-2s<\gamma < 0$) potentials. 

In case of soft potentials ($\gamma < 0$), we show that we cannot generate decay higher than $d+2$. A similar result is true for the Landau equation \cite[Theorem 4.3]{cameron-snelson-silvestre}.

\medskip
\paragraph{\sc Main ideas.}
The strategy is as follows. We start with exploiting the coercivity of the signed terms that stems from the boundedness of the macroscopic quantities in the conditional regime \eqref{e:hydro}. As explained above, this yields two different terms: the first one is linear and can be associated to the energy from the equation; the second one is the ``good extra term'', which occurs from the non-locality of the collision operator \eqref{eq:Q}. We then bound the error terms, that is all the terms on the right hand side of \eqref{e:convex-family}, by these two  coercive terms. This is done mainly through interpolation, weighted Hölder estimates, and a careful study of the non-linearity of \eqref{eq:Q}. This results in the monotonicity of the $L^{q_0}$-norm of $(f-A)_+$, where $q_0$ is the integrability of our weak sub-solutions, and where $A = A(t, v) = a(t) \langle v \rangle^{-q}$ encodes the temporal and velocity decay. We have to impose the restriction $q \in [0, \qnsl]$ with $\qnsl$ given in \S\ref{constants} - \eqref{itm:qnsl}, and therefore obtain Theorem \ref{thm:decroissance}. 

We then construct a sub-solution to \eqref{eq:boltzmann} when $\gamma \leq 0$: we show that $\Psi(t, v) := e^{-\beta t} \langle v \rangle^{\hat q}$ is a sub-solution to \eqref{eq:boltzmann} for $\beta > 0$ sufficiently large, provided that $\gamma \leq 0$. Because of the maximum principle, we see that if $f$ initially does not decay, then $f$ will not decay at later times. The argument works whenever $\hat q > d+2$, since we assume the energy to be bounded at all times. If we initially have $m$ bounded moments, then the argument works for $\hat q > d+m$.

\subsection{Review of literature}

The study of velocity moments of solutions to kinetic equations plays a central role since Boltzmann's (and Landau's) collision operator(s) integrates the velocity variable in the open space $\R^d$. It has a long history, mostly in the space-homogeneous setting, that goes back to the article by T.~Carleman \cite{zbMATH02543771}. It is precisely described in the introduction of \cite{MR4033752}.
Let us give references, summarise the review of literature from \cite{MR4033752} and review the literature written since then. The reader is also referred to  references contained in the articles that are quoted in this subsection. 
\medskip

\paragraph{\sc Velocity moments.}
We first review contributions to the study of moments of solutions in the velocity variable. 
Maxwell potentials are treated in \cite{MR0075725,MR0075726}. In the case of hard potentials and
angular cutoff, Povzner's inequalities \cite{MR0142362} are commonly used: see works by T.~Elmroth \cite{MR684411}, L.~Desvillettes \cite{zbMATH00474497}, S.~Mischler and B.~Wennberg \cite{MR1697562}, X.~Lu \cite{MR1716814} and B.~Wennberg (in the non-cutoff case) \cite{MR1461113}. A.~Bobylev considered exponential moments in \cite{MR1478067}, see also
\cite{gamba-panferov-villani-2009}, and in particular the work \cite{MR3005550} in this direction, where they establish the creation and the propagation of exponential moments to the spatially homogeneous Boltzmann equation for hard potentials.   The case of moderately soft potentials and angular cutoff is addressed by L.~Desvillettes \cite{zbMATH00474497}, see also \cite{MR1942465,MR2359877} in this direction. We finally mention the work by M.~Gualdani, S.~Mischler and C.~Mouhot \cite{zbMATH06889665} that focusses on hard spheres and makes
assumptions on hydrodynamical quantities that are similar to what is assumed in this work. More recently, \cite{MR3759871} studied the generation and propagation of Mittag-Leffler moments  for the space hommogeneous Boltzmann equation with hard potentials without cutoff. N.~Fournier continued this study (hard potentials, non-cutoff) in \cite{MR4315665} and showed that exponential moments are generated and discussed their optimality. C.~Cao, L.~B.~He and J.~Ji \cite{MR4704643} studied the propagation of exponential moments in $L^2$ for (very) soft potentials in a perturbative regime.

\medskip

\paragraph{\sc Pointwise bounds and decay.}
There are fewer results about pointwise decay. It starts with works by T.~Carleman \cite{zbMATH02543771,zbMATH03126493}, later extended by L.~Arkeryd \cite{MR711482}. Exponential pointwise upper bounds were obtained in \cite{gamba-panferov-villani-2009}, see also \cite{gamba2017pointwise}. More recently, the work by C.~Mouhot, L.~Silvestre and the first author \cite{MR4033752} addressed the question of appearance and propagation of polynomial decay in the velocity variable under condition~\eqref{e:hydro} on hydrodynamical quantities. They generate a fixed number of moments for classical solutions on the torus. In this vein, S.~Cameron and S.~Snelson \cite{MR4105382} established similar results in the case $\gamma +2s >2$. The study of polynomial decay is also central in \cite{MR4112183} in which the authors are able to deal with very soft potentials ($\gamma+2s<0$). See also \cite{MR4168919} for results dealing with the Landau equation.
In \cite{MR4416998}, C.~Henderson and W.~Wang are interested in very soft potentials and short time existence. They work in the class of solutions with polynomial decay.

We would finally like to mention some results making use of iterative De Giorgi techniques: \cite{cao, deng, AMSY}. These papers derive pointwise bounds or decay of classical weak solutions in the perturbative regime (close to Maxwellians).

\subsection{Open questions}

Solutions of the spatially inhomogeneous Boltzmann equation converge to Maxwel\-lians for large times \cite{MR2116276}. We recall that under the condition~\eqref{e:hydro}, C.~Mouhot, L.~Silvestre and the first author \cite{MR4112729} proved that solutions stay above a Maxwellian. 
It is thus natural to ask ourselves if they can be bounded from above by another Maxwellian. Unfortunately, our proof does not yield neither such an (optimal) upper bound nor any exponential decay (in $v$). The first open question is thus to show the propagation of pointwise Gaussian bounds. 

Another natural open question is to show or disprove the generation of arbitrary decay for hard potentials. We emphasize that in \cite{MR4033752}  solutions  are
assumed to decay at any polynomial rate $q$ and an a priori estimate is given on the related constant $C_q$. Without this assumption on solutions, methods from \cite{MR4033752}
yield the decay rate $d+1$ for hard potentials. 

The critical case $\gamma + 2s=0$ is left open. The case of very soft potentials is also another natural open question.

\subsection{Notation}
For $p \in [1, \infty]$ and $k \in \R$,  the weighted Lebesgue space  $L^p_k$ is given by
\[
	L^p_k(\R^d):= \left\{f  \colon\R^d \to \R \;\text{ s.t. }  \int_{\R^d} f^p(v) \langle v \rangle^{kp} \dd v < +\infty\right\},
\]
where $\langle v \rangle := (1 + \abs{v}^2)^{\frac{1}{2}}$. When $k=0$, we simply write $L^p$. 

For $a \in \R$ we define $a_+ := \max(a, 0)$.

The volume of the unit sphere of $\R^m$ is denoted by $\omega_{m-1}$. 

\subsection{Constants}
\label{constants}

We gather here parameters and constants that are used repeatedly in statements and proofs. 
\begin{enumerate}[(i).]
\item The dimension of the $x$ and $v$ variables is denoted by $d$. It is always larger than or equal to $2$. 
\item
  The parameters $\gamma$ and $s$ from the kernel satisfy: $\gamma \in \left(-\frac{2sd}{d+2s},1\right]$ and $s \in (0,1)$ and $0 < \gamma +2s \leq 1$. 
\item\label{itm:p0}
The Lebesgue exponent $p_0>1$,  $$\frac1{p_0}  = 1 - \frac{2s}d$$  comes from the entropy production estimate. 
\item\label{itm:q0}
The Lebesgue exponent $q_0>1$, $$q_0=1+\frac{2s}d$$ is related to integrability of solutions. 
\item\label{itm:k0}
The moment exponent $$k_0 = (\gamma+2s) - \frac{2s}d$$  appears in the entropy production estimate. 
\item\label{itm:lq}
  To any decay exponent $q \ge 1$ is associated a  moment exponent $$l_q = (\gamma+2s) +  \frac{2s}d  - q \left(1 + \frac{2s}d\right).$$
\item\label{itm:qnsl}
  The decay exponent $\qnsl$  is given by
  \[
    \qnsl = \begin{cases} d +1 +  \frac{ d }{2s}\min (\gamma,0),\qquad &\textrm{ if } -\frac{2sd}{d+2s} \leq \gamma \leq 1,\\
    d+1 -\frac{d}{d+2s}(2-\gamma) ,\qquad\qquad &\textrm{ if } -2s < \gamma < -\frac{2sd}{d+2s}.
    \end{cases}
  \]
\end{enumerate}

\section{Preliminaries}
\label{s:prelim}

We gather in this section known results and technical lemmas that will be used afterwards. 

\subsection{The collision operator}
\label{sub:coll}

We recall that the use of Carleman coordinates allows us to write the collision operator $Q(f,f)$ and  the kernel $K_f$ as follows,
\[
	Q(f, f) (v) = \int_{\R^d} \left[ f(v') K_f(v, v') - f(v) K_f(v', v)\right] \dd v'
\]
and
\begin{equation}\label{eq:K-f}
	K_f(v, v') = 2^{d-1} \abs{v'-v}^{-1} \int_{w \perp v' -v} f(v+w) B(r, \cos \theta)r^{-d+2} \dd w,
\end{equation}
where $r^2 = \abs{v'-v}^2 + \abs{w}^2$ and $\cos \theta = \frac{|w|^2 - |v-v'|^2}{|w|^2 + |v-v'|^2}$.

\subsubsection{Coercivity}

Under the condition~\eqref{e:hydro}, we know that for any fixed  $(t, x) \in [0,T] \times \Omega$ there exists a set $\Xi(v) \subset \R^d$ for every $v \in \R^d$ such that $\Xi(v)$ is a symmetric cone, and such that for $v' - v\in \Xi(v)$ there holds
\begin{equation}\label{eq:coercivity}
	K_f(v, v') \geq  \lambda_0 \langle v \rangle^{\gamma + 2s + 1}\abs{v-v'}^{-d -2s},
\end{equation}
with $\lambda_0 = \lambda_0(d, m_0, M_0, E_0, H_0)$. Moreover, there holds
\begin{equation}\label{eq:cone-meas}	
	\abs{\Xi(v) \cap  \mathbb S^{d-1}} \geq c_0 \langle v \rangle^{-1},
\end{equation}
for some constant $c_0= c_0(d, m_0, M_0, E_0, H_0)$. The set $\Xi(v)$ is the \textit{cone of non-degeneracy} of $K_f(v, v')$.
We can also ensure that the following estimate holds,
\begin{equation}\label{eq:cone-meas-bis}	
	\abs{\Xi(v) \cap  B_3 \setminus B_2} \geq c_0 \langle v \rangle^{-1}.
\end{equation}
We refer to \cite[Lemma 4.8]{MR3551261} and \cite[Lemma 2.2]{MR4033752}.

\subsubsection{Upper Bound}
For $(t, x) \in [0,T] \times \Omega$ there exists a postive constant $\Lambda_0=\Lambda_0(m_0, M_0, E_0, H_0, d)$ such that
\begin{equation}\label{eq:upperbound}
	\forall v \in \R^d, ~ \forall r > 0, \qquad \int_{B_r(v)} K_f(v, v') \abs{v' - v}^2 \dd v' \leq \Lambda_0 \langle v \rangle^{\gamma +2s} r^{2-2s}. 
\end{equation}
We refer to \cite[Corollary 4.4]{MR3551261}.

\subsubsection{Cancellation}
For $(t, x) \in [0,T] \times \Omega$  there holds
\begin{equation}\label{eq:cancellation}
	\forall v \in \R^d: \int_{\R^d} \left(K_f(v, v') - K_f(v', v)\right) \dd v' = c_b \int_{\R^d} f(v_*)  \abs{v-v_*}^\gamma \dd v_*, 
\end{equation}
with 
\begin{equation*}
	c_b = \int_{\mathbb S^{d-1}} \left\{ \frac{2^{\frac{d+\gamma}{2}}}{(1+\sigma \cdot \textbf e)^{\frac{d+\gamma}{2}}} - 1\right\} b(\sigma \cdot\textbf e) \dd \sigma > 0,
\end{equation*}
for any $\textbf e \in \mathbb S^{d-1}$. See \cite{ADVW}.

\subsubsection{Symmetry}
For any $(t, x) \in [0,T]\times \Omega$ and $v,w \in \R^d$ there holds
\begin{equation}\label{eq:symmetry}
	K_f(v, v+w) = K_f(v, v-w).
\end{equation}

\subsection{Integrability of solutions}

The following lemma is a consequence of Hölder's inequality and a proof is given in \cite{gis}.
\begin{lemma}[H\"older's inequality with weights]\label{l:holder-weight}
  Let $p,q,r \in [1,+\infty]$ and $p \le r \le q$. Then,
  \[ \| f  \|_{L^r_{k_r}} \le \| f \|_{L^p_{k_p}}^\alpha
    \| f  \|_{L^q_{k_q}}^{1-\alpha} \]
  with $\alpha \in (0,1)$ such that
  \( \frac1{r} = \frac{\alpha}p + \frac{1-\alpha}q\)  and \( k_r = \alpha k_p + (1-\alpha)k_q.\)
\end{lemma}
We can use the previous lemma to derive the following one. 
\begin{lemma}[Integrability of solutions] \label{l:integrability}
  If $f \in L^\infty((0, T) \times \bar \Omega; L^1_2(\R^d))$ and $f\in L^1((0, T) \times \bar \Omega; L^{p_0}_{k_0}(\R^d))$ with $k_0$ given by \eqref{itm:k0}, then
\(f \in L^{q_0}_{k_0^\ast}((0, T) \times \bar \Omega \times \R^d),\) where $q_0$ and $k_0^*$ are given by
\[
	q_0 = 2 - \frac{1}{p_0} = 1 + \frac{2s}{d} \in (1, 2), \qquad k_0^\ast = \frac{2s + d(\gamma + 2s)}{d+2s} > 0.
\]
\end{lemma}

\subsubsection{Entropy production estimate}
The entropy production of a function $f$ is given by the following formula,
\begin{equation}\label{eq:entropy-prod}
\begin{aligned}
  D (f) &= - \langle Q(f,f) , \ln f \rangle \\
        &= \frac12 \iiint_{\R^d \times \R^d \times \mathbb{S}^{d-1}} [f(v_*)f(v)-f(v'_*)f(v')] \ln \frac{f(v_*) f(v)}{f(v'_*)f(v')} B(|v-v_*|,\cos \theta) \dsigma \dv \dv'.
\end{aligned}
\end{equation} 
\begin{theorem}[Entropy production estimate -- {\cite[Theorem~0.1]{chaker2022entropy}}] \label{t:entropy}
  Let $\gamma \in (-d,1]$ and $s \in (0,1)$. Then, for any non-negative function $f$  verifying the bounds \eqref{e:hydro}, the entropy production $D(f)$ satisfies
  \[ \|f \|_{L^{p_0}_{k_0} (\R^d)} - c_0 M_0^2 \le C_0 D (f),  \]
  where $p_0$ is such that $\frac{1}{p_0} = 1 - \frac{2s}d$ and $k_0 = \gamma +2s - \frac{2s}{d}$
  and $C_0 = C_0 (d,b,\gamma,s,m_0,M_0,E_0,H_0)$ and $c_0 = c_0 (d,b,\gamma,s)$.
\end{theorem}
In particular, as a consequence of the entropy production estimate and Lemma \ref{l:integrability}, we note that for any solution of \eqref{eq:boltzmann} satisfying the hydrodynamical condition \eqref{e:hydro} there holds $f \in L^{q_0}_{k_0^\ast}((0, T) \times \bar \Omega \times \R^d)$ (Lemma~\ref{l:integrability}).

\subsection{Interpolation estimates}

This subsection is devoted to the proofs of interpolation estimates that will be used in the proof of the main theorem.

H\"older's inequality with weights (Lemma~\ref{l:holder-weight}) applied to $p=r=q=1$ yields the following result.
\begin{lemma}[First weighted $L^1$ estimate]\label{lem:interpolation-q0-0}
Let $f : \R^d \to [0, \infty)$ have a finite $2$-moment. Then there holds,
\[
\forall k \in \R, ~\forall \alpha_0 \in (0,1), \qquad  \int_{\R^d}  f (v) \langle v \rangle^{k}  \dd v
  \leq E_0^{\alpha_0} \left( \int_{\R^d}  f(v) \langle v \rangle^{\frac{k - 2 \alpha_0}{1-\alpha_0}} \dd v \right)^{1-\alpha_0}.
\]
\end{lemma}

\begin{lemma}[Another weighted $L^1$ estimate]\label{lem:interpolation-q0-1-2}
Let $f : \R^d \to [0, \infty)$ satisfy
\[
	f(v) \leq \bast \langle v \rangle^{-q_\ast}
\]	 
for some $q_\ast > 0$, and some $\bast \geq 0$. Then for all $k \in \R$ such that $k  < \alpha_0q_\ast-d$ and all $\alpha_0 \in (0,1)$,
\[
 \int_{\R^d}  f (v) \langle v \rangle^{k}  \dd v
  \leq \bast^{\alpha_0}\left(\int_{\R^d}  \langle v \rangle^{k-\alpha_0q_\ast}  \dd v\right)^{\alpha_0}\left( \int_{\R^d}  f(v) \langle v \rangle^{k-\alpha_0q_\ast} \dd v\right)^{1-\alpha_0}.
\]
\end{lemma}
\begin{proof}
We consider the probability measure
\[
	\dd\mu(v) := \frac{ \langle v \rangle^{k-\alpha_0q_\ast}  \dd v}{\int_{\R^d}  \langle v \rangle^{k-\alpha_0q_\ast}  \dd v} \,.
\]
This is well-defined since $k-\alpha_0q_\ast < -d$. Because $x \to x^{\frac{1}{1-\alpha_0}}$ is convex,  we find with Jensen's inequality 
\begin{align*}
	 \int_{\R^d}  f (v) \langle v \rangle^{k}  \dd v &\leq  \bast^{\alpha_0}\int_{\R^d}  f^{1-\alpha_0} (v) \langle v \rangle^{k-\alpha_0q_\ast}  \dd v \\
	 &= \bast^{\alpha_0} \left(\int_{\R^d}  \langle v \rangle^{k-\alpha_0q_\ast}  \dd v\right) \left( \int_{\R^d}  f^{1-\alpha_0} (v) \dd \mu (v) \right) \\
	 &\leq \bast^{\alpha_0}\left(\int_{\R^d}  \langle v \rangle^{k-\alpha_0q_\ast}  \dd v\right) \left( \int_{\R^d}  f(v) \dd \mu\right)^{1-\alpha_0} \\
	 &= \bast^{\alpha_0}\left(\int_{\R^d}  \langle v \rangle^{k-\alpha_0q_\ast}  \dd v\right)^{\alpha_0}\left( \int_{\R^d}  f(v) \langle v \rangle^{k-\alpha_0q_\ast} \dd v\right)^{1-\alpha_0}. \qedhere
\end{align*}
\end{proof}


Second, we  use that $q_0 \in (1,2)$ to get $1 \in [q_0-1,q_0]$ and  interpolate $f \in L^1$ between $f \in L^{q_0}$ and $f \in L^{q_0-1}$ (with weights). Since $q_0-1 <1$, the interpolation is applied to $f^{(q_0-1)^{-1}}$. 
\begin{lemma}[Second weighted $L^1$ estimate]\label{lem:interpolation-q0-2}
  Let $f \in L^1_2 (\R^d)$ with $f \ge 0$. Let $q_0 = \frac{d+2s}{d} \in (1,2)$, $\frac{1}{p_0} =1 - \frac{2s}{d}$, $k_0 = \gamma + 2s - \frac{2s}{d}$. Then there holds,
\[
\forall k_2 \in \R, \qquad	 \int_{\R^d} f(v) \langle v \rangle^{k_2}  \dd v \leq  \norm{f^{q_0}}_{L^{p_0}_{k_0}}^{\frac{d(d - 2s)}{4s^2 + d^2}}\left( \int_{\R^d}  f^{q_0 -1}(v) \langle v \rangle^{m_2}  \dd v \right)^{ \frac{4sd}{4s^2 + d^2}},
\]
with \[m_2 =  \frac{4s^2 + d^2}{4sd} \left[k_2 - \frac{d(d-2s)}{4s^2 + d^2} k_0\right].\]
\end{lemma}
\begin{proof}
  Because $q_0-1 < 1 < p_0q_0$, we would like interpolate $L^1$ between $L^{q_0-1}$ and $L^{q_0 p_0}$. Because $q_0-1 < 1$, we interpolate $L^{\frac1{q_0-1}}$ between $L^1$ and $L^{\frac{q_0 p_0}{q_0-1}}$ and get,
\[
 \int_{\R^d} f(v) \langle v \rangle^{k_2}  \dd v \leq  \left(\int_{\R^d}  f^{q_0 p_0} (v) \langle v \rangle^{k_0 p_0}  \dd v \right)^{\frac{1-\alpha_4}{p_0q_0}}\left( \int_{\R^d}  f^{q_0-1}(v)  \langle v \rangle^{m_2} \dv \right)^{\frac{\alpha_4}{q_0-1}},
\]
where we need $\alpha_4$ and $m_2$ to satisfy,
\[
\begin{aligned}
	\frac{1-\alpha_4}{p_0 q_0} + \frac{\alpha_4}{q_0-1} = 1, \qquad k_2 = (1-\alpha_4) \frac{k_0}{q_0} + \alpha_4 \frac{m_2}{q_0-1} .
\end{aligned}
\]
This yields
\[
  \begin{cases}
	\alpha_4 = \frac{(q_0 p_0 -  1)(q_0-1)}{q_0p_0 +1 - q_0} = \frac{8s^2}{4s^2 + d^2},\\
	m_2 = \frac{q_0-1}{\alpha_4} \left( k_2 - (1-\alpha_4) \frac{k_0}{q_0} \right).
   \end{cases}
\]
\end{proof}

We can then apply Lemma~\ref{l:holder-weight} with $p = 1$ and $q = p_0 q_0$. 
\begin{lemma}[General weighted $L^{p_1}$ estimate]\label{lem:interpolation-q0-1}
Let $f : \R^d \to [0, \infty)$. Then for all $p_1 \in [1,p_0 q_0]$ and all $k_1 \in \R$,
\[
\| f \|_{L^{p_1}_{k_1/p_1}}\leq  \norm{ f^{q_0} }_{L^{p_0}_{k_0}}^{\frac{d(p_1 -1)}{p_1 4s}}\left( \int_{\R^d}  f(v) \langle v \rangle^{m_1}  \dd v \right)^{ \frac{d+2s - p_1(d-2s)}{p_1 4s}},
\]
with
\[
  m_1 =  \frac{4s}{d+2s - p_1(d-2s)} \left[k_1 - \frac{d(p_1 - 1)}{4s} k_0\right].
\]
\end{lemma}
\begin{proof}
We interpolate
\[
	 \int_{\R^d} f^{p_1} (v) \langle v \rangle^{k_1}  \dd v \leq  \left(\int_{\R^d}  f^{q_0 p_0} (v) \langle v \rangle^{k_0p_0}  \dd v \right)^{\frac{\alpha_1}{p_0}}\left( \int_{\R^d}  f(v) \langle v \rangle^{m_1} \dd v \right)^{\alpha_2},
\]
where we need $\alpha_1, \alpha_2$ and $m_1$ to satisfy,
\[
	\frac{\alpha_1}{p_0} + \alpha_2 = 1, \qquad \alpha_1 q_0 + \alpha_2 = p_1, \qquad\alpha_1 k_0 + \alpha_2 m_1 = k_1.
\]
This yields
\begin{align*}
	&\alpha_1 = \frac{(p_1 -1)p_0}{q_0p_0 -1} = \frac{d(p_1 - 1)}{4s}, \qquad \alpha_2 = \frac{q_0 p_0 -  p_1}{q_0p_0 -1} = \frac{d+2s - p_1(d-2s)}{4s},\\
	&m_1 = \frac{1}{\alpha_2}\left(k_1 - \alpha_1 k_0\right) =  \frac{4s}{d+2s - p_1(d-2s)} \left[k_1 - \frac{d(p_1 - 1)}{4s} k_0\right]. \qedhere
\end{align*}
\end{proof}

Finally, as a consequence of Lemmas~\ref{lem:interpolation-q0-0}, \ref{lem:interpolation-q0-2} and \ref{lem:interpolation-q0-1}, there holds the following estimate.
\begin{lemma}[Weighted $L^{q_0}$ estimate]\label{lem:interpolation-q0-3}
Let $f : \R^d \to [0, \infty)$ have a finite $2$-moment. Then there holds for any $k_3 \in \R$ and any $\varepsilon \in (0, 1)$
\[
   \int_{\R^d}   f^{q_0} (v) \langle v \rangle^{k_3} \dd v \leq \varepsilon \norm{   f^{q_0} }_{L^{p_0}_{k_0}} + C(\varepsilon, E_0) \left( \int_{\R^d}  f^{q_0 -1} (v) \langle v \rangle^{m_3}  \dd v \right),
\]
with \[m_3 =  \left(1+\frac{d}{2s} \right)k_3 - \frac{d}{2s} k_0 - 2,\]
and $C(\eps,E_0)$ only depends on $E_0$, $\eps$, $s$ and $d$.
\end{lemma}
\begin{proof}
We first interpolate with Lemma \ref{lem:interpolation-q0-1} for $p_1 = q_0$ and $k_1 = k_3$,
\[
	\int_{\R^d}  f^{q_0}(v)  \langle v \rangle^{k_3} \dd v \leq \norm{   f^{q_0} }_{L^{p_0}_{k_0}}^{\alpha_1} \left( \int_{\R^d}   f(v) \langle v \rangle^{m_1} \dd v \right)^{\alpha_2},
\]
where $\alpha_1,\alpha_2, m_1$ are given by
\[
	\alpha_1 = \frac{1}{2},\qquad \alpha_2 = \frac{d+2s}{2d}, \qquad m_1 = \frac{2d}{d+2s} \left[k_3 - \frac{1}{2} \left(\gamma + 2s - \frac{2s}{d}\right)\right].
\]
Then apply Lemma \ref{lem:interpolation-q0-0} for any $\alpha_0 \in (0, 1)$ and with $k = m_1$.
\[
	\int_{\R^d}  f(v) \langle v \rangle^{m_1} \dd v  \leq E_0^{\alpha_0} \left(\int_{\R^d}   f(v) \langle v \rangle^{\frac{ m_1 - 2 \alpha_0 }{1-\alpha_0}} \dd v\right)^{1-\alpha_0}.
\]
Next we interpolate with Lemma \ref{lem:interpolation-q0-2} for $k_2 = \frac{m_1 - 2\alpha_0}{1-\alpha_0}$,
\[
  \int_{\R^d}   f(v) \langle v \rangle^{\frac{ m_1 - 2 \alpha_0 }{1-\alpha_0}} \dd v \leq \norm{f^{q_0}}_{L^{p_0}_{k_0}}^{\alpha_3} \left( \int_{\R^d}  f^{q_0-1}(v) \langle v \rangle^{m_2}  \dd v \right)^{\alpha_4},
\]
where 
\[
  \alpha_3 = \frac{d(d - 2s)}{4s^2 + d^2}, \qquad \alpha_4 = \frac{4sd}{4s^2 + d^2},\qquad m_2 = \frac{4s^2 + d^2}{4sd} \left[\frac{m_1 - 2\alpha_0 }{1-\alpha_0} - \frac{d(d-2s)}{4s^2 + d^2} \left(\gamma + 2s - \frac{2s}{d}\right)\right].
\]

Finally we combine the three previous inequalities 
\begin{align*}
  \int   f^{q_0}(v) \langle v \rangle^{k_3} \dd v &\leq  E_0^{\alpha_2 \alpha_0} \norm{ f^{q_0 }}_{L^{p_0}_{k_0}}^{\alpha_1 +(1-\alpha_0)\alpha_2\alpha_3 }\left(\int  f^{q_0-1}\langle v \rangle^{m_2}\dd v \right)^{(1-\alpha_0)\alpha_2\alpha_4}\\
  \intertext{we use Young's inequality for some $\theta >1$ and for any $\varepsilon \in (0, 1)$,}
	&\leq \varepsilon \norm{f^{q_0 }}_{L^{p_0}_{k_0}}^{\theta(\alpha_1 +(1-\alpha_0)\alpha_2\alpha_3 )} + C(\varepsilon, E_0) \left( \int_{\R^d}   f^{q_0 -1} (v) \langle v \rangle^{m_2} \dd v \right)^{\frac{\theta(1-\alpha_0)\alpha_2\alpha_4}{\theta-1}}
\end{align*}
with $ C(\varepsilon, E_0) = E_0^{\frac{\alpha_2 \alpha_0\theta}{\theta-1}} \eps^{-\frac{1}{\theta-1}}$.
We choose $\theta$ and $\alpha_0$ such that
\begin{equation*}
\begin{aligned}
	\theta(\alpha_1 + (1-\alpha_0)\alpha_2\alpha_3) = 1, \qquad \frac{\theta}{\theta-1}(1-\alpha_0)\alpha_2\alpha_4 = 1.
\end{aligned}
\end{equation*}
Then we showed
\[
  \int_{\R^d} f^{q_0} (v) \langle v \rangle^{k_3} \dd v \leq \varepsilon \norm{f^{q_0 }}_{L^{p_0}_{k_0}} + C(\varepsilon, E_0) \left( \int f^{q_0 -1} \langle v \rangle^{m_2}  \dd v \right).
\]
Note that we find
\[
	\theta = \frac{\alpha_3 + \alpha_4}{\alpha_3 +\alpha_1 \alpha_4} =\frac{d+2s}{d}, \qquad \alpha_0= \frac{\alpha_1 - 1 + \alpha_2(\alpha_3 + \alpha_4)}{\alpha_2(\alpha_3 + \alpha_4)}  = \frac{4sd}{(d+2s)^2},
\]
(in particular $\theta >1$), so that
\[
	m_3 = m_2 =  \frac{(d+2s)k_0 - d\left(\gamma + 2s - \frac{2s}{d}\right)}{2s} - 2.
\]
\end{proof}

\section{Truncated Lebesgue norms}
\label{s:propagation}

In order to get the decay estimates from Theorem~\ref{thm:decroissance}, we aim at proving that the following quantity equals zero,
\[  \iint_{\R^d \times \Omega} (f(t,x,v) - A (t,v))_+^{q_0} \dv \dx  \]
with \[A (t,v) = \ma (t) \langle v \rangle^{-q} \] for some well chosen $C^1$ function $\ma >0$ and some exponent $q>0$.
We recall that $q_0 =1 + \frac{2s}d$ (see Lemma~\ref{l:integrability}) and that $r_+$ denotes $\max(0,r)$ for any real number $r$. 
In order to prove that this functional vanishes when applied to a solution of the space-inhomogeneous Boltzmann equation, we investigate how it evolves with time.

\subsection{Truncated Convex Inequalities}
Keeping in mind that $q_0 >1$, we thus consider the convex function $\varphi_0(r) = r_+^{q_0}$ and we would like to apply
the definition of suitable weak sub-solutions  -- see \eqref{e:convex-family} -- and integrate against $\ddot \varphi_0$
(see Remark~\ref{r:more-convex}). Unfortunately, we are not sure that the right hand side of \eqref{e:convex-family}
is integrable with respect to $a$ against $\ddot \varphi_0(a)$. For this reason, we follow \cite{gis} and approximate
$\varphi_0(r)$ with $\phizk(r)$  defined for $\kappa  >0$ by 
\begin{equation}
  \label{eq:varphi-k}
  \phizk(r) = \big(r_+ \wedge \kappa \big)^{q_0} + q_0 \kappa^{q_0-1} (r-\kappa)_+.
\end{equation}
In particular \(\phizk(0) = \dot{(\phizk)}(0) = 0\).
We can now apply the definition of suitable weak sub-solutions and integrate with respect to $\ddphizk(a)$ for $a \in [0,\kappa]$.
What we obtain is expressed in the following statement.       
\begin{lemma}[Truncated Convex Inequalities] \label{l:weak-sols-q0}
Let $f$ be a suitable weak sub-solution of the Boltzmann equation with either in-flow, bounce-back, specular/ diffuse / Maxwell reflection boundary condition. Then, 
\begin{equation}\label{eq:truncated-convex}
\begin{aligned}
\frac{\dd}{\dd t}&\iint_{\Omega \times \R^d}\phizk\big(f-A\big)\dv \dx 	+\iiint_{\Omega \times \R^{2d}} d_{\phizk} ((f-A)_+,(f'-A')_+) K_f(v, v') \dd v' \dd v \dd x \\
                 & + \iiiint_{\Omega \times \R^{2d}} \dphizk(f-A) \left( A' - f'\right)_+ K_f(v, v')  \dd v' \dd v \dd x \\
	&\le  c_b \iint_{\Omega \times \R^d} \Phi\big(f\big) \left(f \ast \abs{\cdot}^\gamma\right) \dd v \dd x +\iiint_{\Omega \times \R^{2d}}  \dphizk(f-A)\left(A' -A\right) K_f(v, v') \dd v' \dd v \dd x \\
	&- \iint_{\Omega \times\R^d} \dphizk(f-A)\partial_t A   \dd v \dd x  - \iint_{\partial \Omega \times \R^d}\phizk(f_b-A)(v \cdot n(x)) \dd v \dd S(x) ,
\end{aligned}
\end{equation}
with  $\Phi$ given by
\begin{equation}\label{eq:Phi}
	\Phi(f) := \dphizk(f-A) f - \phizk(f-A)=\Phizk(f-A) + A \dphizk(f-A).
\end{equation}
\end{lemma}

\subsection{Useful properties associated with $\phizk$}
The derivatives of the approximate convex function $\phizk$ are given by the following formulas,
\[
  \dot{\phizk}(r) = q_0\big(r_+ \wedge \kappa\big)^{q_0-1} \quad \text{and} \quad \ddphizk(r) = q_0(q_0-1) r^{q_0-2}\un_{\{0 \le r \leq \kappa\}}.
\]	
Moreover, we compute for $r \ge 0$,
\[
  \Phizk(r) = \dot{(\phizk )}(r) r - \phizk(r) = (q_0-1) \big(r \wedge \kappa\big)^{q_0},
\]
and for $r,s \ge 0$,
\begin{equation}
  \label{eq:d-phi-k}
	d_{\phizk}(r, s) = \phizk(s) - \phizk(r) - \dphizk(r)(s-r),
\end{equation}
We first remark that 
  \begin{equation}\label{e:monotone-d}
    \frac{\partial d_\phizk}{\partial r} \ge 0 \text{ in } \{ r \ge s \} \qquad \text{ and } \quad \frac{\partial d_\phizk}{\partial r} \le 0 \text{ in } \{ r \ge s \}.
  \end{equation}
If $0 \le r, s \leq \kappa$,
\[
	d_{\phizk}(r, s) = s^{q_0} - r^{q_0} - q_0 r^{q_0-1}(s-r). 
\]
Note that if $r \leq \kappa$, then for $q_0 > 1$
\[
	d_{\phizk}\Big(r, \frac{r}{2} \Big) = c_{ q_0} r^{q_0} \quad \text{ with } \quad  c_{q_0} = 2^{-q_0} -1 + \frac{q_0}{2} > 0.
\]
In particular,
\begin{equation}
  \label{e:breg}
d_{\varphi_0} (r,s) \ge c_{q_0} r^{q_0} \quad \text{ in } \quad \{  0 \le s \le r/2 \} .
\end{equation}
We also have for $r \ge 0$,
\begin{equation}\label{eq:d-phi-k-special}
	d_{\phizk}\left(r, \frac{r \wedge \kappa}{2} \right) = c_{q_0} \Big(r \wedge \kappa\Big)^{q_0}.
\end{equation}

We conclude with
\begin{equation}\label{e:c2p}
d_{\phizk} (r,s) \ge C_{2,q_0} (\phizk (s) + \kappa^{q_0}) \quad \text{ for } \quad 2r < \kappa < s 
\end{equation}
where $C_{2,q_0} = \min (1-2^{1-q_0}, (q_0-1) 2^{-q_0} - q_0 2^{1-q_0} +1)>0$, see \cite[Eq.~(15)]{gis}. 

\section{Coercivity estimates}
\label{s:coercivity}

This section is devoted to the study of positive terms appearing in the left hand side of the family of convex inequalities associated with the approximate $L^{q_0}$-norm, see \eqref{eq:truncated-convex}.
The term in which $(A'-f')_+$ appears is referred to as the ``good extra term'' (see the introduction).

\subsection{First coercivity estimate}

In order to get our first coercivity estimate, we follow the method introduced in \cite{chaker2022entropy} and used in \cite{ouyang-silvestre} and \cite{gis}.
Precisely, the proof of the following proposition is adapted from the proof of \cite[Lemma~3.1]{gis} in which $k_0 <0$. We recall that $k_0 = (\gamma+2s) -\frac{2s}d$. 
\begin{proposition}[Coercivity estimate for $k_0 \ge 0$]\label{p:coercivity-1}
Assume $f$ satisfies \eqref{e:hydro}. If $k_0 \ge 0$, then there exists a constant $\Ccor > 0$, depending on $d, \gamma, s$ and constants in \eqref{e:hydro}, such that
for all $f,g \colon \R^{d} \to [0,+\infty)$,
\begin{equation}\label{eq:i-large-q-gamma-pos}
 \iint_{\R^{2d}} d_{\phizk}  (g,g') K_f(v, v') \dv \dv' \ge \Ccor \norm{\big(g\wedge \kappa\big)^{q_0}}_{L^{p_0}_{k_0}(\R^d)}.
\end{equation}
\end{proposition}
\begin{proof}
  For $R > 0$ (to be chosen), we define
  \[ T_R = \begin{cases} B_{R/2} & \text{ if } |v| \ge R, \\ B_{3R}  \setminus B_{2R} & \text{ if } |v| \le R. \end{cases}\]
In particular, for $v' -v \in T_R$, we have $|v| \le 2 |v'|$, which implies $\langle v \rangle \le 2 \langle v' \rangle$.  

We consider $g' < \frac{g \wedge \kappa}2$ and use the monotonicity properties of $d_\phizk$ (see \eqref{e:monotone-d}), we have
\[ g' < \frac{g\wedge \kappa}2 \qquad \Rightarrow \qquad d_{\phizk} (g,g') \ge d_{\phizk} (g, \frac{g \wedge \kappa}2 )
  \ge d_{\phizk} (g \wedge \kappa, \frac{g \wedge \kappa}2 ) \ge c_{q_0} \big(g\wedge \kappa\big)^{q_0}\]
where we used \eqref{e:breg} to get the last line. 

Then we use the non-degeneracy cone $\Xi (v)$ associated with the kernel $K_f$ and get from \eqref{eq:coercivity},
\begin{align*}
 \iint_{\R^{2d}} &d_{\phizk} (g,g') K_f(v, v') \dv \dv' \\
	 &\geq c_{q_0} \int_{\R^d} \int_{\left\{ g' < \frac{g\wedge \kappa}{2} \right\} \cap \left\{v'-v \in \Xi(v) \cap T_R\right\}} \big(g\wedge \kappa\big)^{q_0}(v) K_f(v, v') \dd v' \dd v\\
	 &\geq c_{q_0}\lambda_0  \int_{\R^d} \int_{\left\{ g' < g/2 \right\} \cap \left\{v'-v \in \Xi(v) \cap T_R\right\}} \big(g\wedge \kappa\big)^{q_0}(v) \langle v \rangle^{\gamma + 2s + 1}\abs{v-v'}^{-d -2s} \dd v' \dd v\\
	  &\geq c\int_{\R^d}  \big(g\wedge \kappa\big)^{q_0}(v)  R^{-d-2s}\langle v \rangle^{\gamma + 2s + 1}\abs{\left\{ g' < \frac{g\wedge \kappa}{2} \right\} \cap \left\{v + \Xi(v) \cap T_R\right\} }\dd v
\end{align*}
where  $c = 3^{-d - 2s}c_{q_0} \lambda_0$. 
Moreover, for any $v \in \R^d$, we use either \eqref{eq:cone-meas-bis} or  \eqref{eq:cone-meas} and Chebyshev's inequality  to estimate the Lebesgue norm of the sub-level set as follows,
\begin{align*}
	&\abs{\left\{ g' <  \frac{g\wedge \kappa}{2} \right\} \cap \left\{v + \Xi(v) \cap T_R\right\} } \\
	&\qquad= \abs{\left\{v + \Xi(v) \cap T_R\right\} } - \abs{\left\{ g' >  \frac{g\wedge \kappa}{2} \right\} \cap \left\{v + \Xi(v) \cap T_R\right\} }\\
	&\qquad\geq c_0 R^{d} \langle v\rangle^{-1} - \frac{2^{p_0(q_0+k_0)}}{\langle v \rangle^{k_0p_0} \big(g\wedge \kappa\big)^{q_0p_0}(v)} \int_{\R^d}  (g' \wedge \kappa)^{q_0p_0} \langle v'\rangle^{k_0p_0}\dd v'.
\end{align*}
The last line uses that $k_0 \ge 0$ and $\langle v \rangle \le 2 \langle v' \rangle$. 
Then,  we choose $R$ such that 
\begin{equation}\label{eq:R-q0}
	R^d = c_1 \norm{g^{q_0}}_{L^{p_0}_{k_0}}^{p_0} \langle v \rangle^{-k_0 p_0 +1} \big(g\wedge \kappa\big)^{-q_0p_0}(v), 
\end{equation}
with $c_1 = \frac{2^{p_0 (q_0+k_0)+1}}{c_0}$ and get
\begin{align*}
  \abs{\left\{ g' <  \frac{g\wedge \kappa}{2}  \right\} \cap \left\{v + \Xi(v) \cap T_{R}\right\}} & \geq c_0 R^{d} \langle v\rangle^{-1} - 
  \frac{c_0}{2}R^d \langle v \rangle^{-1}  = \frac{c_0}{2} R^{d} \langle v\rangle^{-1}.
\end{align*}
Therefore, with such a choice of $R$, 
\begin{align*}
   \iint_{\R^{2d}} d_{\phizk} (g,g') K_f(v, v') \dv \dv' 
  &\geq c \frac{c_0}2 \int_{\R^d} \big(g\wedge \kappa\big)^{q_0}(v) R^{-2s} \langle v\rangle^{\gamma +2s} \dd v\\
  &= c \frac{c_0}2 c_1^{-\frac{2s}d} \norm{g^{q_0}}_{L^{p_0}_{k_0}}^{-\frac{2sp_0}{d}} \int_{\R^d} \langle v\rangle^{\gamma +2s + \frac{2s(k_0 p_0 -1)}{d}} \big(g\wedge \kappa\big)^{q_0\left(1 + \frac{2sp_0}{d}\right)}(v)\dd v,
    \intertext{and remarking that $p_0$ and $k_0$ are such that $p_0 = 1 + \frac{2sp_0}{d}$ and $k_0 p_0 = \gamma +2s + \frac{2s(k_0 p_0 -1)}{d}$,}
  &\geq c \frac{c_0}2 c_1^{-\frac{2s}d}  \norm{g^{q_0}}_{L^{p_0}_{k_0}(\R^d)}^{1-p_0}\norm{\big(g\wedge \kappa\big)^{q_0}}_{L^{p_0}_{k_0}(\R^d)}^{p_0}.
\end{align*}
We reach the desired conclusion with $\Ccor = c \frac{c_0}2 c_1^{-\frac{2s}d}$.
\end{proof}

\begin{proposition}[Coercivity estimate for $k_0<0$] \label{p:cor-gamma-neg}
Assume $k_0<0$. 
Then there exists $\Ccor > 0$, depending on $s, d, M_0, E_0$, such that
for all $f,g \colon \R^{d} \to [0,+\infty)$,
\begin{equation}
  \label{eq:i-q0-large-q}
 \iint_{\R^{2d}} d_{\phizk} (g,g') K_f(v, v') \dv \dv' 
 \ge \Ccor \norm{\big(g\wedge \kappa\big)^{q_0}}_{L^{p_0}_{k_0}(\R^d)}- \Ccor \int_{\R^d} (g \wedge \kappa)^{q_0} \langle v \rangle^\gamma \dd v.
\end{equation}
 In particular,
\begin{equation}
  \label{eq:i-q0}
 \iint_{\R^{2d}} d_{\phizk} (g,g') K_f(v, v') \dv \dv' 
 \ge \Ccor \norm{\big(g\wedge \kappa\big)^{q_0}}_{L^{p_0}_{k_0}(\R^d)}- \Ccor\norm{\big(g\wedge \kappa\big)^{q_0-1}}_{L^1_{\gamma -d - 1}(\R^d)}.
\end{equation}
\end{proposition}
\begin{proof}
The estimate~\eqref{eq:i-q0-large-q} follows the lines of  the proof of Proposition \ref{p:coercivity-1} and is proved in \cite[Lemma~3.1]{gis}. 

We can obtain \eqref{eq:i-q0} by applying Lemma \ref{lem:interpolation-q0-3} to $g\wedge \kappa$ and $\varphi = 1$ and $k_3=\gamma$ and get for all $\varepsilon \in (0, 1)$,
\[\int_{\R^d} \big(g\wedge \kappa\big)^{q_0}\langle v \rangle^{\gamma} \dd v	\leq \varepsilon\norm{\big(g\wedge \kappa\big)^{q_0}}_{L^{p_0}_{k_0}(\R^d)} 	+ C(\varepsilon,E_0) \norm{\big(g\wedge \kappa\big)^{q_0-1}}_{L^1_{m_3}}\]
with $m_3 = \gamma -d -1$.
\end{proof}

We now prove another coercivity estimate for $\gamma<0$.
\begin{proposition}[Second coercivity estimate] \label{p:second-coercivity}
Assume that $\gamma \in (-2s, - \frac{2s d}{d+2s})$. Let $g = (f-A)_+$ for $A = A(v) = \ma \langle v \rangle^{-q}$ for some $q \in (0,\qnsl)$. 
There exists positive constants $\Ccorbis$ and $\ma_0$, depending on $s, d, \gamma$ and hydrodynamical bounds from \eqref{e:hydro},  such that for all $\kappa \ge 1$
and $\ma \ge \ma_0$,
\begin{equation}
  \label{eq:second-coercivity}
  \iint_{\R^{2d}} d_{\phizk} (g,g') K_f(v, v') \dv \dv'   \ge \Ccorbis \kappa^{\frac{2s}d q_0} N^{-\frac{2s}d} \int_{\R^d} (g(v)-\kappa)_+ \langle v \rangle^{\gamma} \dv
\end{equation}
with $N = \int_{\R^d} (g \wedge \kappa)^{q_0-1} \langle v \rangle^{-(d-1)}  \dv$. 
\end{proposition}
\begin{proof}
  We  adapt the proof of \cite[Lemma~3.2]{gis}. As we shall see it, the only difference lies in the choice of $R$.

\noindent \textsc{Use of the non-degeneracy set.}
  We start with reducing the domain of integration to velocities $(v,v')$ such that $g' > \kappa > 2g$ and then use the lower bound \eqref{e:c2p} on the  distance $d_{\phizk}$,   
  \[
    \iint_{\R^{2d}} d_{\phizk} (g,g') K_f(v, v') \dv \dv'
     \ge C_{2,q_0} \int_{\{g'> \kappa \}} \varphi_{0,\kappa} (g') \left\{\int_{\{ g < \kappa/2\}} K_f(v, v') \dv \right\} \dv'.
\]
  We use next the non-degeneracy set for fixed $v'$,
  \[
  	S(v') = \left\{v \in \R^d : K(v, v') \geq \mu_0 \langle v'\rangle^{\gamma + 2s+1}\abs{v-v'}^{-(d+2s)}\right\}. 
  \]
  We know that \cite[Lemma~2.6]{gis} proves the existence of $\mu_0,\eps_0 > 0$, depending only on the hydrodynamical quantities and dimension,
  such that for all $R \in (0, \eps_0 \langle v' \rangle)$,
  \[
  	 \abs{S(v') \cap B_R(v')}\geq \mu_0 R^d \langle v' \rangle^{-1}.
  \]
Using this set, we can continue the estimate from below as follows,
  \begin{align}
\nonumber    \iint_{\R^{2d}} d_{\phizk} (g,g') K_f(v, v') \dv \dv'
\nonumber    & \ge C_{2,q_0} \int_{\{g'> \kappa \}} \varphi_{0,\kappa} (g') \left\{\int_{\{ g < \kappa/2\}  \cap S(v') \cap B_R(v')} K_f(v, v') \dv \right\} \dv' \\
\nonumber    & \ge C_{2,q_0} \mu_0 \int_{\{g'> \kappa \}} \varphi_{0,\kappa} (g') \left\{\int_{\{ g < \kappa/2\}  \cap S(v') \cap B_R(v')} \langle v' \rangle^{\gamma+2s+1} R^{-d-2s} \dv \right\} \dv' \\
\label{e:sce-1}    & \ge C_{2,q_0} \mu_0 \int_{\{g'> \kappa \}} \varphi_{0,\kappa} (g') \langle v' \rangle^{\gamma+2s} \left| \{ g < \kappa/2\}  \cap S(v') \cap B_R(v') \right| \frac{\langle v'\rangle}{R^d}
      R^{-2s}  \dv'.
  \end{align}
  
  In our case, we choose $R = R(v')$ such that 
  \[ \frac{R^d}{\langle v' \rangle} = C_0 \frac{N}{\kappa^{q_0-1} \langle v' \rangle^{-(d-1)}}
  \]
for some constant $C_0$ to be chosen later. We have to check that such an $R$ is smaller than $\eps_0 \langle v' \rangle$. 
In order to do so, we first estimate $N$ from above.

\noindent \textsc{Bound on $N$.}  
We continue the proof with bounding $N$: since $g = (f-A)_+$ and $(g \wedge \kappa) = 0$ for $f \leq A = \ma \langle v \rangle^{-q}$, we reduce the integration to $\{f > A = \ma \langle v \rangle^{-q}\}$, but on this set we know $ \langle v \rangle^{-(d-1)} \leq \left(\frac{f}{\ma}\right)^{\frac{d-1}{q}}$. 
  
  If $\frac{d-1}{q} \geq 1$, then $-(d-1) \leq -q$ so $\langle v \rangle^{-(d-1)}  \leq \langle v \rangle^{-q} \leq \frac{f}{\ma}$. Thus
   \begin{align*}
  	N &= \int_{\R^d} (g \wedge \kappa)^{q_0-1} \langle v \rangle^{-(d-1)}  \dv			\leq\frac{\kappa^{q_0-1}}{\ma} \int_{f \geq A} f \dv \leq \frac{\kappa^{q_0-1}}{\ma} M_0.
  \end{align*}
  
  If $\frac{d-1}{q} < 1$, pick $\alpha = 1 - \frac{d-1}{q}$, and notice $0 < \alpha < q_0 -1$ for any $-2s < \gamma < -\frac{2sd}{d+2s}$ and $q \leq \qnsl$.
  Indeed, $\alpha < q_0-1$ is equivalent to $1-\frac{d-1}q < \frac{2s}d$ and since $q \le \qnsl$, we only have to check that $1-\frac{d-1}\qnsl < \frac{2s}d$, or equivalently
  $\qnsl < \frac{d(d-1)}{d-2s}$. The latter condition can be written
  \[ 6sd < 4s^2 + 2sd^2 + 2sd(\gamma +2s) - \gamma d^2\]
  where we recall that $\gamma <0$ and $\gamma+2s>0$. For $d \ge 3$ and $s>0$, we simply have $6 sd \le 2s d^2$. For $d=2$, the assumption of the proposition implies that $-\gamma > \frac{2s}{1+s}$.
  In particular $-\gamma d^2 = -4\gamma > \frac{8s}{1+s} \ge 4s - 4s^2 =6 sd - 4s^2 -2sd^2$.

  Thus
 \begin{align*}
  	N &= \int_{\R^d} (g \wedge \kappa)^{q_0-1} \langle v \rangle^{-(d-1)}  \dv= \int_{\{ f \geq A \}} (g \wedge \kappa)^{q_0-1-\alpha}(g \wedge \kappa)^{\alpha} \langle v \rangle^{-(d-1)}  \dv\\
          &\leq\frac{ \kappa^{q_0-1-\alpha}}{\ma^{\frac{d-1}{q}}} \int_{\{ f \geq A \}} f^{\alpha + \frac{d-1}{q}} \dv  \leq \frac{ \kappa^{q_0-1-\alpha}}{\ma^{\frac{d-1}{q}}} M_0
            \le \frac{ \kappa^{q_0-1}}{\ma^{\frac{d-1}{q}}} M_0 \qquad \text{ for }  \kappa \ge 1.
 \end{align*}
   Thus
  \[
  	N \leq \frac{\kappa^{\frac{2s}{d}}}{\ma_0^{\min\{\frac{d-1}{q}, 1\}}} M_0 .
  \]

\noindent \textsc{Smallness of $R$.}
We now choose $\ma_0$ (as a function of $C_0$, to be chosen later)  such that 
  \[
  	 \frac{C_0 M_0}{\eps_0^d } \leq  \ma_0^{\min\{\frac{d-1}{q}, 1\}}. 
  \] 
Then for each $\kappa \geq 1$, we have $R \le \eps_0 \langle v' \rangle$. Indeed, the upper bound on $N$ that we obtained above yields,
  \[
    R^d  = C_0 \frac{N}{\kappa^{\frac{2s}{d}}}\langle v'\rangle^d \leq \frac{C_0 M_0}{\ma_0^{\min\{\frac{d-1}{q}, 1\}}} \langle v'\rangle^d \leq \eps_0^d \langle v'\rangle^d. \]

\noindent \textsc{Lower bound on the super-level set of $g'$.}  
Using this choice of $R$ allows us to estimate from below the measure $\left| \{ g < \kappa/2\}  \cap S(v') \cap B_R(v') \right|$. Indeed, we first write,
\begin{align*}
    |\{ g \wedge \kappa > \kappa/2 \} \cap S(v') \cap B_R (v') | & \le \frac{2^{q_0-1}}{\kappa^{q_0-1}} \int_{S(v') \cap B_R (v')} (g(v) \wedge \kappa)^{q_0-1} \frac{\langle v \rangle^{d-1}}{\langle v \rangle^{d-1}} \dv \\
                                                                 & \le 2^{q_0+d-2} \frac{N}{\kappa^{\frac{2s}d}} \langle v' \rangle^{d-1} \\
    & = \frac{2^{q_0+d-2}}{C_0} \frac{R^d}{\langle v' \rangle}. 
  \end{align*}
  We used that for $v \in B_R (v')$ and $R$ such that $R \le \eps_0 \langle v' \rangle$, we have $\langle v \rangle \le 2 \langle v' \rangle$ (since $\eps_0 \le 1$). 
  This implies that 
  \begin{align}
    \nonumber
    |\{ g < \kappa/2 \} \cap S(v') \cap B_R (v') |  &\ge \mu_0 \frac{R^d}{\langle v' \rangle} - \frac{2^{q_0+d-2}}{C_0} \frac{R^d}{\langle v' \rangle}\\
    \label{e:sce-2}    & \ge \frac{\mu_0}2 \frac{R^d}{\langle v' \rangle}
  \end{align}
  if
  \[ C_0  = 2^{q_0+d-1} \mu_0^{-1}.\]

\noindent \textsc{Conclusion.}
Keeping in mind that $\varphi_{0,\kappa} (g') \ge q_0 \kappa^{q_0-1} (g'-\kappa)_+$ in $\{g'> \kappa\}$, we combine \eqref{e:sce-1} with \eqref{e:sce-2} and get
\begin{align*}
  \iint_{\R^{2d}} d_{\phizk} (g,g') K_f(v, v') \dv \dv'  & \ge C_{2,q_0} \frac{\mu_0^2}2 \int_{\{g' > \kappa\}} \varphi_{0,\kappa} (g') \langle v' \rangle^{\gamma+2s} R^{-2s} \dv' \\
  &\ge C_{2,q_0} \frac{\mu_0^2}2 q_0 \kappa^{q_0-1}  C_0^{-\frac{2s}d} N^{-\frac{2s}d } \kappa^{(q_0-1) \frac{2s}d} \int_{\R^d} (g'-\kappa)_+ \langle v' \rangle^{\gamma}  \dv'.
\end{align*}
In view of the definition of $q_0$, we obtain the announced lower bound for some constant $\Ccorbis$ equal to  $C_{2,q_0} \frac{\mu_0^2}2 q_0 C_0^{-\frac{2s}d}$. 
\end{proof}

\subsection{The good extra term}

We exhibit a third coercivity term. It controls the weighted $L^{q_0-1}$-norm of $(f-A)_+$.
This lemma is similar to \cite[Lemma~5.3]{gis}, but more general since $A$ depends on $v$.
In contrast with the previous coercivity estimates, it is nonlinear in nature, it cannot
be stated for two general functions $f$ and $g$.
\begin{proposition}[The good extra term]\label{p:coercivity-get}
  Let  $q \in [0,d+1]$ and $A (t,v) = \ma (t) \langle v \rangle^{-q}$. Then there exists a constant $\Cget > 0$ depending on $m_0, M_0, E_0, H_0, d$, such that if $\ma$ satisfies 
\begin{equation}\label{e:cond-a}
  \ma(t) \geq \aget \quad \text{ with } \quad \aget := 2^{d+1} (M_0 + E_0) c_0^{-1}
\end{equation}
with $c_0$ from \eqref{eq:cone-meas}, then for any $f$ satisfying \eqref{e:hydro}, there holds,
\begin{equation}\label{eq:ii-q0}
    \iint_{\R^{2d}} \dphizk(f-A) \left( A' - f'\right)_+ K_f(v, v') \dv \dv'
  \geq \Cget \ma^{1+\frac{2s}{d}} \int_{\R^d}  \Big((f-A)_+ \wedge \kappa\Big)^{q_0-1}  \langle v \rangle^{l_q}   \dv.
\end{equation}
\end{proposition}
\begin{proof}
  Let $R>0$ (to be chosen). We denote the  ``good'' set of velocities
  by $G := \left\{v : \langle v\rangle \geq 2 R\right\}$.
  For $v' \in B_R(v)$ and $v \in G$, we have $\langle v' \rangle \le 2 \langle v \rangle$ and $\langle v \rangle \le 2 \langle v' \rangle$.

  We also recall that $\dphizk (r) = q_0\big(r_+ \wedge \kappa\big)^{q_0-1}$.  
  
In particular, $A(v') \ge  2^{-q} A (v)$ for such $v$'s and $v'$'s. 
   We now use the non-degeneracy cone $\Xi (v)$, see in particular \eqref{eq:coercivity}, in order to write,
\begin{align*}
	 \iint_{\R^{2d}} & \dphizk(f-A)  \left( A' - f'\right)_+ K_f(v, v') \dd v \dd v' \\
	 &\geq \frac{q_0}{2} \int_G \int_{ \left\{ f' < \frac{A'}{2}\right\} \cap \left\{ v + \Xi(v) \cap B_R\right\} } \big((f-A)_+\wedge \kappa \big)^{q_0-1}(v) A'  K_f(v, v') \dd v' \dd v\\
	 &\geq \frac{q_0\lambda_0}{2^{q+1}} \int_G R^{-d-2s} \langle v \rangle^{\gamma + 2s+1} \big((f-A)_+\wedge \kappa\big)^{q_0-1} A \abs{\left\{ f' < \frac{A'}{2}\right\} \cap \left\{ v + \Xi(v) \cap B_R\right\} }  \dd v.
\end{align*}

We then use \eqref{eq:cone-meas} and Chebyshev's inequality to bound the sublevel set from below for $v \in G$ as follows,
\begin{align*}
	\abs{\left\{ f' < \frac{A'}{2}\right\} \cap \left\{ v + \Xi(v) \cap B_R\right\} } &= \abs{\left\{ v + \Xi(v) \cap B_R\right\} }  -\abs{\left\{ f' \ge \frac{A'}{2}\right\} \cap \left\{ v + \Xi(v) \cap B_R\right\} }  \\
                                                                                          &\geq c_0 R^{d} \langle v \rangle^{-1} -  \int_{v + \Xi(v) \cap B_{R}} \frac{2}{A'} f' \dd v'\\                                                                                          &\geq c_0 R^{d} \langle v \rangle^{-1} -  \frac{2}{\ma} \int_{v + \Xi(v) \cap B_{R}} \langle v' \rangle^{q-2} f(v') \langle v' \rangle^{2} \dd v'\\
  \intertext{we use again that $\langle v' \rangle \le 2 \langle v \rangle$ if $q \ge 2$ and $\langle v \rangle \le 2 \langle v' \rangle$ if $q \le 2$ to get,}
	&\geq  c_0 R^{d} \langle v \rangle^{-1} -  \frac{2^{1+|q-2|}}{A(v) \langle v \rangle^{2}} \int_{\R^d} f(v') \langle v'\rangle^{2}  \dd v'\\
	&\geq  c_0 R^{d} \langle v \rangle^{-1} -  \frac{2^{d+1}}{A(v) \langle v \rangle^{2}} (M_0 + E_0).
\end{align*}

We  choose next $R=R(v)$ such that
\[
	R^d  =  \frac{2^{d+2} (M_0 + E_0)}{c_0 A(v) \langle v \rangle}.
\]
Then
\[
  \abs{\left\{ f' < \frac{A'}{2}\right\} \cap \left\{ v + \Xi(v) \cap B_{R}\right\} }\geq  \frac{c_0}2 R^{d} \langle v \rangle^{-1} .
\]
With such an estimate of the sublevel set in hand, we can finish to justify the estimate from below of the good extra term,
\begin{align*}
	\iint_{\R^{2d}} \dphizk(f-A) \left(A' - f'\right)_+ & K_f(v, v') \dd v' \dd v  \\
	&\geq \frac{q_0\lambda_0c_0}{2^{q+2}} \int_G R^{-2s} \langle v \rangle^{\gamma + 2s} \big((f-A)_+\wedge \kappa\big)^{q_0-1}(v)  A(v)  \dd v\\
	&= \Cget \int_G \langle v \rangle^{\gamma + 2s + \frac{2s}{d}}  \big((f-A)_+\wedge \kappa\big)^{q_0-1}(v)  A^{1+\frac{2s}{d}}(v)  \dd v
\end{align*}
with $\Cget = \frac{q_0\lambda_0c_0}{2^{q+2}} \left(2^{d+2} c_0^{-1} (M_0+E_0) \right)^{-\frac{2s}d}$.

Finally, since $A(v) = \ma  \langle v \rangle^{-q}$, we have
\[
	G = \left\{ v: \langle v \rangle \geq 2 R\right\} = \left\{ v: \langle v \rangle^{d+ 1-q} \geq \frac{2^{d+2}(M_0 + E_0)}{c_0 \ma} \right\}.
\]
If $q \leq d+1$ and if  $\ma$ is such that  \( \ma \geq  2^{d+2} (M_0 + E_0) c_0^{-1}\) then $G = \R^d$.
\end{proof}

\section{Error term due to the velocity dependence of the barrier}

This section is devoted to the estimate of the first error term appearing in the right hand side of the Truncated Convex Inequalities,
see Lemma~\ref{l:weak-sols-q0}. We recall that it has the following form,
\[
\mathcal{E}_1 :=  \iint_{\R^{2d}}  \big((f-A)_+\wedge \kappa\big)^{q_0-1}(v) \left(A' -A\right) K_f(v, v') \dd v' \dd v.
\]
In the next paragraph, we use ideas from \cite{MR4033752}.
We recall that 
\[ A (t,v) = \ma (t) \langle v \rangle^{-q}.\]

\begin{lemma}[Error term in the velocity variable]\label{l:etvv}
 We consider a function $f$ satisfying \eqref{e:hydro}.
Let $q \in [0, d+1]$ and  $A(v) := \ma \langle v \rangle^{-q}$ for some $\ma >0$. There exists $\Cun$ depending on $d, m_0, M_0, E_0, H_0$ such that
\begin{equation}\label{eq:iii-q0}
\begin{aligned}
  \iint_{\R^{2d}}   \big((f-A)_+\wedge \kappa\big)^{q_0-1}(v) & \big(A(v') -A(v)\big) K_f(v, v') \dd v' \dd v \\
  & \leq \Cun \ma \int_{\R^d} \big((f-A)_+\wedge \kappa\big)^{q_0-1}(v) \langle v \rangle^{l_q} \dd v,
\end{aligned}
\end{equation}
with $l_q$ given by \eqref{itm:lq}.
\end{lemma}

Recalling that the kernel $K_f$ is given by,
\[
  K_f(v, v+ z) = \frac{1}{\abs{z}^{d+2s}} \int_{v_*' \in v+ z^\perp} f(v_*') \abs{v - v_*'}^{\gamma +2s +1} \tilde b(\cos \theta) \dd v_*',
\]
we aim at estimating 
\[
\mathcal E_1 =  \int_{\R^d} \big((f-A)_+ \wedge \kappa \big)^{q_0-1}(v)  \mathcal I(v) \dd v 
\]
with
\begin{align*}
	\mathcal I(v) &:= \int_{\R^d} \left(A(v') -A(v)\right) K_f(v, v') \dd v' \\
	&= \int_{\R^d}   \frac{\left(A(v+z) -A(v)\right)}{\abs{z}^{d+2s}} \left\{ \int_{v_*' \in v+ z^\perp} f(v_*') \abs{v - v_*'}^{\gamma +2s +1} \tilde b(\cos \theta) \dd v_*' \right\} \dd z\\
	&= \int_{\R^d}   f(v_*') \abs{v - v_*'}^{\gamma +2s} \left\{ \int_{ z \in (v-v_*')^\perp} \frac{\left(A(v+z) -A(v)\right)}{\abs{z}^{d+2s-1}} \tilde b(\cos \theta) \dd z \right\} \dd v_*',
\end{align*}

where we changed the order of integration in the last equality.

We  distinguish the singular part from the non-singular one,
\[ \mathcal{I} (v) = \mathcal{I}_1 (v) + \mathcal{I}_2 (v) \]
with
\begin{equation}\label{eq:split-s-ns}
\begin{cases}
	\mathcal I_1(v) &= \int_{\R^d}   f(v_*') \abs{v - v_*'}^{\gamma +2s} \left\{ \int_{ z \in (v-v_*')^\perp} \chi_{\abs{z} > \frac{1}{2} \langle v \rangle} \frac{\left(A(v+z) -A(v)\right)}{\abs{z}^{d+2s-1}} \tilde b(\cos \theta) \dd z \right\} \dd v_*'  \\
\mathcal I_{2}(v)	 &=  \int_{\R^d}   f(v_*') \abs{v - v_*'}^{\gamma +2s} \left\{ \int_{ z \in (v-v_*')^\perp} \chi_{\abs{z} < \frac{1}{2} \langle v \rangle} \frac{\left(A(v+z) -A(v)\right)}{\abs{z}^{d+2s-1}} \tilde b(\cos \theta) \dd z \right\} \dd v_*'.
\end{cases} 
\end{equation}
We now estimate each part separately. 

\begin{lemma}[Estimate of the non-singular part]\label{lem:ns}
  Recall $A (t,v) = \ma (t) \langle v \rangle^{-q}$.
  \begin{equation}\label{eq:I1-bis}
  \mathcal I_1(v) \leq
  \begin{cases} C E_0 \ma  \langle v\rangle^{\gamma -2 - \min(q, d-1)}, \qquad &\textrm{if } q\neq d-1,\\
    C E_0 \ma  \langle v\rangle^{\gamma  -d },\qquad &\textrm{if } q= d-1.
  \end{cases}
\end{equation}
In particular, if $q \leq d+1$, then
  \begin{equation}\label{eq:I1}
	\mathcal I_1(v) \leq  C\ma  \langle v\rangle^{l_q},
\end{equation}
where $l_q$ is given by \eqref{itm:lq}.
\end{lemma}
\begin{proof}
For the non-singular part we further distinguish two cases: $\abs{v_*'} < \frac{1}{8} \abs{v}$ and $\frac{1}{8}\abs{v} < \abs{v_*'}$,
\[\mathcal{I}_1(v) = \mathcal{I}_{1,1} (v) + \mathcal{I}_{1,2} (v) \]
with
\[\begin{cases}
	\mathcal I_{1,1} (v) = &\int_{\abs{v_*'} < \frac{1}{8}\abs{v}}   f(v_*') \abs{v - v_*'}^{\gamma +2s} \int_{ z \in (v-v_*')^\perp} \chi_{\abs{z} > \frac{1}{2} \langle v \rangle} \frac{\left(A(v+z) -A(v)\right)}{\abs{z}^{d+2s-1}} \tilde b(\cos \theta) \dd z\dd v_*' \\[1ex]
\mathcal I_{1, 2}(v) =	& \int_{\frac{1}{8}\abs{v} < \abs{v_*'}}   f(v_*') \abs{v - v_*'}^{\gamma +2s} \int_{ z \in (v-v_*')^\perp} \chi_{\abs{z} > \frac{1}{2} \langle v \rangle} \frac{\left(A(v+z) -A(v)\right)}{\abs{z}^{d+2s-1}} \tilde b(\cos \theta) \dd z\dd v_*'.
\end{cases}\]
The idea for $\mathcal I_{1,1}$ is to show that $\abs{v + z} > \abs{v}$ by exploiting the orthogonality $v + z \perp v_*'$ and the smallness $\abs{v_*'} < \frac{1}{8} \abs{v}$. In particular, $\mathcal I_{1, 1}$ has a good sign. For the second term $\mathcal I_{1, 2}$, we exploit the $2$-moment of $f$. 
\medskip

We start with $\mathcal I_{1, 1}$. Since and $z \perp (v-v'_*)$ and $\abs{v_*'} < \frac{1}8 \abs{v}$ and $\abs{z} > \frac{1}{2}\langle v \rangle > \frac{1}{2} \abs{v}$ , we find 
\begin{equation}\label{eq:I11}
\begin{aligned}
	\abs{v+z}^2 &= \abs{v}^2 + \abs{z}^2 + 2v \cdot z = \abs{v}^2 + \abs{z}^2 + 2 v_*' \cdot z\\
	&\geq  \abs{v}^2 + \abs{z}\left(\abs{z} - 2 \abs{v_*'}\right) >  \abs{v}^2 + \abs{z}\left(\abs{z} - \frac{1}{4}\abs{v}\right) \\
	&>  \abs{v}^2 + \frac{1}{4}\abs{z}\abs{v} > \left(1 + \frac{1}{8}\right)  \abs{v}^2.
\end{aligned}
\end{equation}
Since $q \geq 0$,
\(
	A(v+z) - A(v) = \ma \left[\langle v +z \rangle^{-q} - \langle v  \rangle^{-q}\right] \leq  0,
\)
we conclude that $\mathcal I_{1, 1}(v) < 0$.
\medskip

For $\mathcal I_{1, 2}$ we find
\begin{align*}
	\mathcal I_{1, 2}(v) &=\int_{ \frac{1}{8}\abs{v} < \abs{v_*'}}   f(v_*') \abs{v - v_*'}^{\gamma +2s} \int_{ z \in (v-v_*')^\perp} \chi_{\abs{z} > \frac{1}{2} \langle v \rangle} \frac{\left(A(v+z) -A(v)\right)}{\abs{z}^{d+2s-1}} \tilde b(\cos \theta) \dd z\dd v_*'\\
\intertext{we use that $\tilde b (\cos \theta)\le \tilde b_+$ and $A(v) \ge 0$,}
                             &\leq \tilde b_+ \ma \int_{ \frac{1}{8}\abs{v} <\abs{v_*'}}   f(v_*') \abs{v - v_*'}^{\gamma +2s} \int_{ z \in (v-v_*')^\perp} \chi_{\abs{z} > \frac{1}{2} \langle v \rangle}  \frac{\langle v+z\rangle^{-q}}{\abs{z}^{d+2s-1}}  \dd z\dd v_*'\\
	&\leq\tilde b_+ \ma \int_{ \frac{1}{8}\abs{v} <\abs{v_*'}}   f(v_*') \frac{(9\abs{v_*'})^{\gamma +2s}\langle v_*'\rangle^{2}}{\langle v_*'\rangle^{2}} \int_{ z \in (v-v_*')^\perp} \chi_{\abs{z} > \frac{1}{2} \langle v \rangle}  \frac{\langle v+z\rangle^{-q}}{\abs{z}^{d+2s-1}}  \dd z\dd v_*'\\
	&\leq \bigg(9^{\gamma +2s}8^{2 - \gamma - 2s}\tilde b_+\bigg) \ma \langle v \rangle^{\gamma +2s - 2}\int_{ \frac{1}{8}\abs{v} <\abs{v_*'}}   f(v_*') \langle v_*'\rangle^{2} \int_{ z \in (v-v_*')^\perp} \chi_{\abs{z} > \frac{1}{2} \langle v \rangle}  \frac{\langle v+z\rangle^{-q}}{\abs{z}^{d+2s-1}}  \dd z\dd v_*'
\end{align*}
since  $2 \geq \gamma + 2s$.

Finally, we compute,
\begin{align*}
	&\int_{ z \in (v-v_*')^\perp} \chi_{\abs{z} > \frac{1}{2} \langle v \rangle}  \frac{\langle v+z\rangle^{-q}}{\abs{z}^{d+2s-1}}  \dd z \\
	&\quad= \int_{ z \in (v-v_*')^\perp} \chi_{\abs{z} > \frac{1}{2} \langle v \rangle} \chi_{\abs{v+z} > \abs{v}}  \frac{\langle v+z\rangle^{-q}}{\abs{z}^{d+2s-1}}  \dd z + \int_{ z \in (v-v_*')^\perp} \chi_{\abs{z} > \frac{1}{2} \langle v \rangle} \chi_{\abs{v+z} < \abs{v}}   \frac{\langle v+z\rangle^{-q}}{\abs{z}^{d+2s-1}}\dd z\\
  \intertext{we let $\omega_{d-2}$ denote the volume of the unit sphere of dimension $d-2$, }
	&\quad\leq \omega_{d-2} \langle v\rangle^{-q} \int_{\frac{1}{2}\langle v\rangle}^\infty r^{-d +1 - 2s +d - 2} \dd r  +  (\langle v \rangle/2)^{-d - 2s + 1} \int_{z \in (v-v'_*)^\perp} \chi_{|v+z|< |v|}\langle v+z\rangle^{-q} \dd z\\
 	&\quad \leq C \langle v \rangle^{-q-2s} + C \langle v \rangle^{-(d-1)-2s} \int_0^{|v|} \frac{r^{d-2}}{(1+r^2)^{q/2}} \dd r 
\end{align*}
for some positive constant $C$ only depending on $d$ and $s$ only.

We now distinguish cases.
If $q > d-1$, then  $\frac{r^{d-2}}{(1+r^2)^{q/2}} \in L^1(\R)$ and 
\[ \int_{ z \in (v-v_*')^\perp} \chi_{\abs{z} > \frac{1}{2} \langle v \rangle}  \frac{\langle v+z\rangle^{-q}}{\abs{z}^{d+2s-1}}  \dd z \le \Cqu \langle v \rangle^{ -(d-1)-2s}\]
with
  \[ \Cqu = C + C\int_0^{+\infty} \frac{r^{d-2}}{(1+r^2)^{q/2}} \dd r  \le C + C(1 + 2^{-q/2} (q- (d-1))^{-1})\]
(split the domain integration between $r \le 1$ and $r \ge 1$).

If  $q < d-1$, then we simply write,
  \[ \int_0^{|v|} \frac{r^{d-2}}{(1+r^2)^{q/2}} \dd r \le \int_0^{|v|} r^{d-2-q} \dd r = \frac{|v|^{d-1-q}}{d-1-q} \le \frac{\langle v \rangle^{d-1-q}}{ d-1-q}. \]
  This implies that
  \[ \int_{ z \in (v-v_*')^\perp} \chi_{\abs{z} > \frac{1}{2} \langle v \rangle}  \frac{\langle v+z\rangle^{-q}}{\abs{z}^{d+2s-1}}  \dd z \le C \langle v \rangle^{ -q-2s}.\]

If $q =d-1$, then
  \[ \int_0^{|v|} \frac{r^{d-2}}{(1+r^2)^{(d-1)/2}} \dd r \le \int_0^{|v|} r^{-1} \dd r = \ln |v| .\]
  This implies in this case that
  \[ \int_{ z \in (v-v_*')^\perp} \chi_{\abs{z} > \frac{1}{2} \langle v \rangle}  \frac{\langle v+z\rangle^{-q}}{\abs{z}^{d+2s-1}}  \dd z \le C \langle v \rangle^{ -(d-1)-2s}  \ln \langle v \rangle.\]

We thus proved that,
\begin{equation}\label{eq:I12}
\begin{aligned}
	\mathcal I_{1, 2}(v) \leq \begin{cases} C E_0 \ma  \langle v\rangle^{\gamma -2 - \min(q, d-1)}, \qquad &\textrm{if } q\neq d-1,\\
	C E_0 \ma  \langle v\rangle^{\gamma  -d},\qquad &\textrm{if } q= d-1.
	\end{cases}
\end{aligned}
\end{equation}	
Using $\mathcal I_{1, 1} < 0$ and \eqref{eq:I12} we obtain \eqref{eq:I1-bis} and \eqref{eq:I1}. 

We can check that
\[
	\gamma - 2 - \min(q, d-1) \leq (\gamma + 2s) + \frac{2s}{d}  - q\frac{d+2s}{d}= l_q,
\]
for $q \neq d-1$, and 
\[
	\gamma  - d \leq (\gamma + 2s) + \frac{2s}{d}  - (d-1)\frac{d+2s}{d} = l_{d-1}.
\]
The proof is now complete.
\end{proof}

\begin{lemma}[Estimate of the singular part]\label{lem:s}
  Recall that $A (t,v) = \ma (t) \langle v \rangle^{-q}$ and $\mathcal{I}_2$ is given by \eqref{eq:split-s-ns}.
  \begin{equation}\label{eq:I2}
    \mathcal{I}_2(v) \le \Cqd \fraka  \langle v \rangle^{\gamma-q}
  \end{equation}
with $\Cqd = \omega_{d-2} \frac{q(q+4)}2 2^{\frac32 (q+2)}$.
  
In particular, for $q \leq d+1$ there holds $ \gamma -q \leq l_q$ where $l_q$ is given in \eqref{itm:lq}. 
\end{lemma}
\begin{proof}
For the singular part,  we do a Taylor expansion: there exists $\theta \in (0, 1)$ such that
\[
A(v +z) - A(v) - \nabla_v A(v + \theta z) \cdot z = \frac12 \text{D}_v^2 A(v+ \theta z) z \cdot z
\]
for some $\theta \in (0,1)$. Moreover, $A  (w) = \fraka \mathcal{A}(|w|)$ with $\mathcal{A} (r) = (1+r^2)^{-q/2}$. In particular,
\[\frac1{\ma} D^2_v A (w) = \mathcal{A}'' (|w|) \frac{w}{|w|} \otimes \frac{w}{|w|} + \frac{\mathcal{A}'(|w|)}{|w|} \left( I - \frac{w}{|w|} \otimes \frac{w}{|w|} \right) \]
where $I$ denotes the identity matrix. As far as derivatives of $\mathcal{A}$ are concerned, we have
\[ \frac{|\mathcal{A}'(r)|}r \le q (1+r^2)^{-\frac{q+2}2} \quad \text{ and } \quad |\mathcal{A}'' (r)| \le q(q+3) (1+r^2)^{-\frac{q+2}2}.\]
This implies that for all $w,\xi \in \R^d$, we have
\[ D^2_v A (w) \xi \cdot \xi \le q(q+4) \ma \langle w \rangle^{-q-2} |\xi|^2.\]
In particular,
\[
|A(v +z) - A(v) - \nabla_v A(v + \theta z) \cdot z | \le  \frac{q(q+4)}2 \ma \langle v+ \theta z \rangle^{-q-2} |z|^2.
\]
We further note for $|z| \leq \frac12 \langle v \rangle$,
\begin{equation}
  \label{e:vplusz}
  \langle v + \theta z \rangle \geq \frac{1}{2\sqrt{2}} \langle v\rangle.
\end{equation}
Indeed,
\[ \langle v + \theta z \rangle \ge \frac1{\sqrt{2}} (1+|v+\theta z|) \ge \frac1{\sqrt{2}} (1+|v| -| z|) \ge \frac1{\sqrt{2}} (1+|v| -\frac12 \langle v \rangle) \ge  \frac{1}{2\sqrt{2}} \langle v\rangle. \]
We thus have,
\[
|A(v +z) - A(v) - \nabla_v A(v + \theta z) \cdot z | \le  \Cqd  \ma \langle v \rangle^{-q-2} |z|^2
\]
with $\Cqd = \frac{q(q+4)}2 2^{\frac32 (q+2)}$. Due to the symmetry \eqref{eq:symmetry} and the fact that $\gamma + 2s \ge 0$, we thus obtain
\begin{equation}\label{eq:inner-int-s}
\begin{aligned}
	\mathcal I_2(v) &=\int_{\R^d}   f(v_*') \abs{v - v_*'}^{\gamma +2s} \int_{ z \in (v-v_*')^\perp} \chi_{\abs{z} < \frac{1}{2} \langle v \rangle} \frac{\left(A(v+z) -A(v)\right)}{\abs{z}^{d+2s-1}} \tilde b(\cos \theta) \dd z\dd v_*'\\
	&\leq \Cqd \tilde b_+ \fraka\langle v \rangle^{-q - 2} \int_{\R^d} f(v_*') \abs{v-v_*'}^{\gamma + 2s} \int_{z \in (v-v_*')^\perp} \chi_{ \abs{z} < \frac{1}{2}\langle v\rangle} \abs{z}^{2-d -2s +1} \dd z \dd v_*'\\
	&\leq \omega_{d-2}\Cqd \tilde b_+   \fraka \langle v \rangle^{-q - 2}  \left\{\int_{\R^d} f(v_*') \left(\abs{v}^{\gamma + 2s} +\abs{v_*'}^{\gamma + 2s}\right) \dd v_*' \right\} \left\{\int_0^{\frac{1}{2}\langle v \rangle} r^{3 - d - 2s +d - 2} \dd r \right\}\\
        &\leq \omega_{d-2} \Cqd \fraka\langle v \rangle^{-q - 2} \langle v \rangle^{\gamma +2s} \langle v \rangle^{-2s + 2}.
\end{aligned}
\end{equation}
This implies \eqref{eq:I2}. 
\end{proof}

\begin{proof}[Proof of Lemma~\ref{l:etvv}]
Therefore, if we now denote by $\mathcal{E}_1$ the left hand side of \eqref{eq:iii-q0}, then \eqref{eq:I1} and \eqref{eq:I2} imply
\begin{align*}
	\mathcal E_1 &=  \int_{\R^d} \big((f-A)_+\wedge \kappa\big)^{q_0-1}(v) \mathcal I(v) \dd v = \int_{\R^d} \big((f-A)_+\wedge \kappa\big)^{q_0-1}(v) \left[\mathcal I_1(v) + \mathcal I_2(v)\right] \dd v\\
	&\leq C \fraka\int_{\R^d} \big((f-A)_+\wedge \kappa\big)^{q_0-1}(v)\langle v \rangle^{l_q} \dd v.
\end{align*}
This concludes the proof.
\end{proof}

\section{Remaining error terms for hard potentials}

In the case of hard potentials, that is to say in the case where $\gamma \ge 0$, we can treat error terms appearing in the right hand side of \eqref{eq:truncated-convex} all together. 
More precisely, we can estimate,
\[
  \int \Phi(f)  \left(f \ast \abs{\cdot}^\gamma\right)  \dd v 
  =   \int  \left(\left(q_0 - 1\right)\Big((f-A)_+\wedge \kappa\Big)^{q_0}(v) +q_0 A\Big((f-A)_+\wedge \kappa\Big)^{q_0 - 1}\right)(v) \left(f \ast \abs{\cdot}^\gamma\right)  \dd v
\]
with  $\Phi$ from \eqref{eq:Phi}. Note that the case $\gamma \ge 0$ is significantly simpler than $\gamma < 0$.

\subsection{Generation of decay}

We start with an estimate that is used to prove the generation of moments in the case of hard potentials. 
\begin{lemma}[Error terms for hard potentials]\label{lem:gamma-pos}
  Let $f: \R^d \to [0, \infty)$ have finite mass $M_0$ and energy $E_0$ and let $q \in [0,  d + 1]$.
  There exists a constant $\Chard >0$ depending on $M_0, E_0, d$ such that for any $\varepsilon \in (0, 1)$ there holds
\begin{multline}\label{eq:iv-q0}
	 \int_{\R^d} \Phi(f)   \left(f \ast \abs{\cdot}^\gamma\right) \dd v \\
	  \leq \varepsilon\norm{ \Big((f-A)_+\wedge \kappa\Big)^{q_0}}_{L^{p_0}_{k_0}(\R^d)}  + \Chard \left(\fraka + C(\varepsilon) \right)\norm{ \Big((f-A)_+\wedge \kappa\Big)^{q_0-1}}_{L^1_{\gamma - q}(\R^d)}.
\end{multline}
\end{lemma}
\begin{proof}
We first remark that
\[
  \int_{\R^d} f(w) \abs{v-w}^\gamma \dd w \le C_\gamma \int_{\R^d} f(w) (\langle v \rangle^\gamma + |w|^2) \dd w  
\]
for some constant $C_\gamma$ only depending on $\gamma$. We used that for $\gamma \in [0,1]$, $(a+b)^\gamma \le C_\gamma ((a^2+1)^{\frac{\gamma}2} + b^2)$. Let $C_{\gamma,0}$ denote $C_\gamma ( M_0 + E_0)$. 
We have,
\begin{align*}
  \nonumber	 \int_{\R^d}  \Phi(f) \left(f \ast \abs{\cdot}^\gamma\right) \dd v
  &\leq C_{\gamma,0} \int_{\R^d} \Phi(f)\langle v \rangle^\gamma \dd v\\
 &\leq q_0 C_{\gamma,0}  \int_{\R^d} \left( \fraka   \big((f-A)_+\wedge \kappa\big)^{q_0 - 1} \langle v \rangle^{\gamma-q} +  \big((f-A)_+\wedge \kappa\big)^{q_0 } \langle v \rangle^{\gamma} \right)  \dd v.
\end{align*}
Using the interpolation from Lemma \ref{lem:interpolation-q0-3} with $k_3 = \gamma$, there holds for 
\(
	m_3 = \gamma - (d +1)  \le \gamma -q,
\)
(recall $q \leq d + 1$) and any $\varepsilon \in (0,1)$
\[
  \int_{\R^d}   \big((f-A)_+\wedge \kappa\big)^{q_0} \langle v \rangle^{\gamma} \dd v \leq \varepsilon \norm{ \big((f-A)_+\wedge \kappa\big)^{q_0}}_{L^{p_0}_{k_0}} + C(\varepsilon, E_0) \left( \int_{\R^d} \big((f-A)_+\wedge \kappa\big)^{q_0 -1}(v)  \langle v \rangle^{\gamma -q}  \dd v \right).
\]
This yields the claim \eqref{eq:iv-q0}. 
\end{proof}

\section{Main error terms  for soft potentials}

We next explain how to estimate the following term in the case of soft potentials, \textit{i.e.} when $\gamma<0$, 
\begin{multline*} \iint_{\R^d \times \Omega} \Phi\big(f\big) \left(f \ast \abs{\cdot}^\gamma\right) \dd v \dd x  \\= 
\iint_{\R^d \times \Omega}  \left(\left(q_0 - 1\right)\Big((f-A)_+\wedge \kappa\Big)^{q_0}(v) +q_0 A\Big((f-A)_+\wedge \kappa\Big)^{q_0 - 1}\right)(v) \left(f \ast \abs{\cdot}^\gamma\right)  \dd v \dx. 
\end{multline*}
We will use the following auxiliary results. 

\subsection{Auxiliaries}

\begin{lemma}[Convolution product]\label{lem:gamma-neg-aux0} 
Let $\gamma <0$ and $f : \R^d \to [0, \infty)$ have finite mass $M_0$. Let $A(t, v) = \fraka \langle v \rangle^{-q}$ for some $q \geq 0$ and suppose, for all $t \geq 0$, $\fraka(t) \geq 2^d\frac{M_0}{\Cqt}$ with $\Cqt =\frac{\omega_{d-1} (2\sqrt{2})^q }{(d+\gamma)}$. Then there exists $\Cconv > 0$ depending on $d, \gamma, q$ and $M_0$, such that
\[
	(f\wedge A) \ast |\cdot|^\gamma \leq \Cconv \fraka^{-\frac{\gamma}{d}} \langle v\rangle^{\gamma \min\left(1, \frac{q}{d}\right)}.
\]
\end{lemma}
\begin{proof}
For $R>0$ such that $R< \langle v \rangle /2$, we write
\begin{align*}
  \int_{\R^d} (f \wedge A)(v-w) \abs{w}^\gamma \dd w & = \int_{B_R} (f \wedge A)(v-w) \abs{w}^\gamma \dd w + \int_{\R^d\setminus B_R} (f \wedge A)(v-w) \abs{w}^\gamma \dd w\\
                                                     &\leq \fraka  \int_{B_R} \langle v -w \rangle^{-q} |w|^\gamma \dd w + R^\gamma M_0 \\
  \intertext{and for $R < \langle v \rangle /2$ and $w \in B_R$, we have $\langle v -w \rangle \ge \frac1{2\sqrt{2}} \langle v \rangle$ (argue as in \eqref{e:vplusz}),}
                                                     & \le \fraka \; \omega_{d-1} (2\sqrt{2})^q \langle v \rangle^{-q} \frac{R^{d+\gamma}}{d+\gamma} + R^\gamma M_0  \le \Cqt a \langle v \rangle^{-q} R^{d+\gamma} + M_0 R^\gamma,
\end{align*}
where
\begin{equation}
  \label{e:cqt}
  \Cqt = \frac{\omega_{d-1} (2\sqrt{2})^q}{(d+\gamma)}.
\end{equation}
Choosing $R = \big(\frac{M_0}{\Cqt \fraka}\big)^{\frac{1}{d}} \langle v\rangle^{\min(1, \frac{q}{d})} \le \frac12 \langle v \rangle$ (since $\fraka \geq 2^d\frac{M_0}{\Cqt}$)  yields
\begin{align*}
  \int_{\R^d} (f \wedge A)(v-w) \abs{w}^\gamma \dd w &\leq (\Cqt)^{-\frac{\gamma}{d}}M_0^{1+\frac{\gamma}{d}} \fraka^{-\frac{\gamma}{d}} \langle v \rangle^{-q +(d+\gamma) \min(1, \frac{q}{d})} + (\Cqt)^{-\frac{\gamma}{d}}M_0^{1+\frac{\gamma}{d}}   \fraka^{-\frac{\gamma}{d}}  \langle v \rangle^{\gamma \min(1, \frac{q}{d})}\\
  &\leq  2(\Cqt)^{-\frac{\gamma}{d}}M_0^{1+\frac{\gamma}{d}}  \fraka^{-\frac{\gamma}{d}}\langle v \rangle^{\gamma \min(1, \frac{q}{d})}. \qedhere
\end{align*}
\end{proof}

We next collect auxiliary computations in the following lemma.
\begin{lemma}[Auxiliary technical result]\label{lem:aux-algebra}
Let $\gamma \in (-2s,0)$ and
\[
	1 \leq q \leq 2 + \frac{d}{2s} k_0,  \qquad l_1 = \min(2, q).
\]
If we pick 
\begin{equation}\label{eq:choice-theta-beta}
	\theta = \frac{d+2s}{d-2s}, \qquad \frac{1}{\beta} = \frac{d^2 + 4s^2}{(d+2s)^2}, \qquad \frac{1}{\beta'} = \frac{4sd}{(d+2s)^2},
\end{equation}
then 
\begin{equation}\label{eq:algebra-eq}
	\frac{1}{\beta} + \frac{1}{\beta'} = 1, \qquad\frac1{\theta}+\frac1{\theta'} = 1, \qquad \gamma_2 \theta = 1, \qquad \gamma_3\theta' =1,
\end{equation}
where 
\begin{equation}\label{eq:gammas}
	\gamma_1 := \frac{8s^2(d+\gamma)}{d(4sd +\gamma(d+2s))}, \qquad \gamma_2 := \frac{1}{\beta}\frac{(d-2s)(d+2s)}{d^2 + 4s^2}, \qquad \gamma_3 := \frac{1}{\beta}\frac{4s(d+2s)}{d^2 + 4s^2}.
\end{equation}
Moreover, 
\begin{equation}\label{eq:algebra-ineq}
	\gamma_1 \theta' < 1 + \frac{2s}{d},
\end{equation}
and there holds
\begin{equation}\label{eq:weight-condition}
	m_1 \leq \frac{2}{\beta'} + \frac{m_3}{\beta},
\end{equation}
and
\begin{equation}\label{eq:weight-condition-IV}
	m_4 \leq \frac{2}{\beta'} + \frac{m_3}{\beta},
\end{equation}
with 
\begin{align*}
	m_1 &=  - \frac{4sd^2}{(d+2s)(4sd + \gamma(d+2s))} l_1 + \frac{\gamma d}{4sd + \gamma(d+2s)} k_0,  \qquad m_3= \frac{4sd}{d^2 +4s^2} l_q + \frac{d(d-2s)}{d^2 + 4s^2}k_0,\\
	m_4 &=  - 2 \frac{4sd}{(d+2s)(\gamma + 2s)}  - \frac{(2s-\gamma) d}{(d+2s)(\gamma +2s)} k_0.
\end{align*}
\end{lemma}
\begin{proof}
The relations in \eqref{eq:algebra-eq} are verified by a straightforward computation. We then rewrite \eqref{eq:algebra-ineq} as
\begin{equation*}
	\frac{2s(d+2s)(d+\gamma)}{d(4sd +\gamma(d+2s))} < 1 + \frac{2s}{d},
\end{equation*}
which holds for $\gamma + 2s > 0$.

It remains to check that \eqref{eq:weight-condition} and \eqref{eq:weight-condition-IV} are satisfied.
We rewrite \eqref{eq:weight-condition} as 
\begin{equation*}
	- \frac{4sd^2}{4sd + \gamma(d+2s)} l_1 + \frac{\gamma d(d+2s)}{4sd + \gamma(d+2s)} k_0 \leq 2 \frac{4sd}{(d+2s)}  + \frac{4sd}{(d+2s)} l_q + \frac{d(d-2s)}{(d+2s)} k_0.
\end{equation*}
Then using that $l_q = k_0 + 2\frac{2s}{d}  - q \frac{d+2s}{d}$ we get
\begin{equation*}
	- \frac{4sd^2}{(4sd + \gamma(d+2s))} l_1 + \frac{\gamma d(d+2s)}{4sd + \gamma(d+2s)} k_0 \leq 8s  - 4s q + dk_0,
\end{equation*}
or equivalently,
\begin{equation*}
	- d^2 l_1 \leq (2 -  q)(4sd + \gamma(d+2s)) + d^2 k_0.
\end{equation*}

If  $2 \geq q$ and $k_0 \geq 0$  then \eqref{eq:weight-condition} is satisfied since $l_1 \ge 0$.

If $2 \leq q$ and $k_0 \geq 0$ then $l_1 =2$, and since $q \leq 2 + \frac{d}{2s}k_0$, \eqref{eq:weight-condition} is satisfied if 
\begin{equation*}
	- 2 d^2  \leq -\frac{d}{2s} k_0 (4sd + \gamma(d+2s)) + d^2 k_0 = -d^2 k_0 - \frac{d}{2s}\gamma(d+2s) k_0, 
\end{equation*}
which holds if
\begin{equation*}
	\left(\gamma + 2s - \frac{2s}{d}\right) = k_0 \leq 2,
\end{equation*}
which is true since $\gamma +2s \leq 2$.

If $k_0 < 0$ then $q  \leq 2 + \frac{d}{2s} k_0< 2$ and in particular $l_1 =q$, so that \eqref{eq:weight-condition} is satisfied if
\begin{equation*}
	0 \leq (2 -  q)(4sd + \gamma(d+2s)) + d^2 k_0 + d^2 q.
      \end{equation*}
We remark that $4sd + \gamma (d+2s) = (\gamma+2s)d + 2s (\gamma +d) \ge 0$. Since $2-q \ge  \frac{d}{2s} (-k_0) \ge 0$ and $q \ge 1$, we have to check that 
\begin{equation*}
	\left(\frac{\gamma}{d} + \frac{\gamma+2s}{2s}\right)k_0 \leq   1,
\end{equation*}
is satisfied.
If $\frac{\gamma}{d} + \frac{\gamma+2s}{2s} \geq 0$ then \eqref{eq:weight-condition} is satisfied.  If $\frac{\gamma}{d} + \frac{\gamma+2s}{2s} \leq 0$, then we see that \eqref{eq:weight-condition} holds as soon as 
\[
	\gamma k_0 =  \gamma \left(\gamma + 2s -\frac{2s}{d}\right) \leq d ,
\]
which is satisfied. Indeed,  $0 < -\gamma < d$ and $s <1$ ensures that $-\gamma \frac{2s}d \le 2$. We conclude by recalling that $d \ge 2$. 
Thus in any case we see that \eqref{eq:weight-condition} is satisfied for $1 \leq q \leq 2 + \frac{d}{2s} k_0$.
\medskip

Finally, we rewrite \eqref{eq:weight-condition-IV} as
\begin{eqnarray*}
  -2 \frac{4sd}{(d+2s)(\gamma + 2s)}  - \frac{(2s-\gamma) d}{(d+2s)(\gamma +2s)} k_0 \leq 2 \frac{4sd}{(d+2s)^2}  + \frac{d^2 + 4s^2}{(d+2s)^2}\left( \frac{4sd}{d^2 +4s^2} l_q + \frac{d(d-2s)}{d^2+4s^2} k_0\right),\\
\Leftrightarrow	-2 \frac{4sd}{(\gamma + 2s)}  - \frac{(2s-\gamma) d}{(\gamma +2s)} k_0 \leq 2 \frac{4sd}{(d+2s)}  + \frac{1}{(d+2s)}\left( 4sd \left[k_0+ 2 \frac{2s}{d}  - q\frac{d+2s}{d}\right]+ d(d-2s)k_0\right), \\
\Leftrightarrow	 4s q  \leq 8s +2 \frac{4sd}{(\gamma + 2s)}  + dk_0+ \frac{(2s-\gamma) d}{(\gamma +2s)} k_0.
\end{eqnarray*}
We thus see, that for $q \leq 2 + \frac{d}{2s} k_0$, \eqref{eq:weight-condition-IV} is satisfied if
\begin{equation*}
	\gamma  k_0 = \gamma\left(\gamma + 2s - \frac{2s}{d}\right) \leq 4s.
\end{equation*}	
This is satisfied  since $-\gamma \leq d$.
\end{proof}

\subsection{Splitting}

We split the error term involving $\Phi (f)$ by writing $f =  (f \wedge A) + (f-A)_+ \wedge \kappa + (f-A - \kappa)_+$,
\[
  \int_{\R^d} \Phi\big(f\big) \left(f \ast \abs{\cdot}^\gamma\right) \dd v    = \mathcal{E}_2 + \Etrois  + \Equatre+\mathcal{E}_5
\]
with
\begin{equation}\label{eq:iv-split}
\begin{cases}
  \mathcal{E}_2 &=q_0 c_b\int  A\Big((f-A)_+\wedge \kappa\Big)^{q_0 - 1} \big((f\wedge A) \ast \abs{\cdot}^\gamma\big) \dd v\\
  \Equatre &=q_0  c_b \int  A\Big((f-A)_+\wedge \kappa\Big)^{q_0 - 1} \big((f-A)_+\wedge \kappa\ast \abs{\cdot}^\gamma\big) \dd v\\
  \Etrois &=(q_0-1)c_b \int \Big((f-A)_+\wedge \kappa\Big)^{q_0}\left(f \ast \abs{\cdot}^\gamma\right) \dd v\\
  \mathcal{E}_5&=q_0  c_b \int  A\Big((f-A)_+\wedge \kappa\Big)^{q_0 - 1} \big((f-A - \kappa)_+\ast \abs{\cdot}^\gamma\big) \dd v.
\end{cases}
\end{equation}

\subsection{The term $\mathcal{E}_2$}
\begin{lemma}[Estimate of $\mathcal{E}_2$]\label{lem:iv.i}
  Let $\gamma \in (-d, 0]$ and $\gamma + 2s > 0$ and $q \in [0, d+1]$. Then
\begin{equation}\label{eq:iv.i-gamma-neg}
	\mathcal{E}_2 \leq  \frac12 \Cget \fraka^{1 + \frac{2s}{d}} \norm{\big((f-A)_+\wedge \kappa\big)^{q_0-1}}_{L^1_{l_q}},
\end{equation}
as soon as $\fraka$ satisfies for all $t \geq0$
\begin{equation}
  \label{e:cond-a-bis}
\fraka(t) \ge \underline{a}^{(2)} = \max \left(\frac{2^dM_0 }{C_{q, 3}}, \left(\frac{2}{\Cget}\right)^{\frac{d}{\gamma+2s}}, a_{get} \right)
\end{equation}
(with  $a_{get}, \Cget$ from Proposition~\ref{p:coercivity-get}).
\end{lemma}
\begin{proof}
By Lemma \ref{lem:gamma-neg-aux0} (see the assumption on $\ma$) there holds
\begin{equation*}
  \mathcal{E}_2 \leq \Cconv q_0 c_b \fraka^{-\frac{\gamma}{d}} \int_{\R^d}   A\Big((f-A)_+\wedge \kappa\Big)^{q_0 - 1} \langle v \rangle^{\gamma \min(1, \frac{q}{d})} \dd v.
\end{equation*}
It suffices to check that $1 - \frac{\gamma}{d} < 1 +\frac{2s}{d}$, which is true since $\gamma+2s > 0$, and that $\gamma \min(1, \frac{q}{d}) -q \leq l_q$, which again is true since $q \leq d + 1$. Thus \begin{equation*}
	\fraka^{1-\frac{\gamma}{d}} \int \Big((f-A)_+\wedge \kappa\Big)^{q_0 - 1} \langle v \rangle^{\gamma\min(1, \frac{q}{d})-q} \dd v \leq \frac12 \Cget \fraka^{1 + \frac{2s}{d}} \norm{\Big((f-A)_+\wedge \kappa\Big)^{q_0-1}}_{L^1_{l_q}}
\end{equation*}
where $\Cget$ is the constant from Proposition~\ref{p:coercivity-get}. We used the assumption on $\ma$ from the statement.
\end{proof}

\subsection{The term $\Equatre$}
	
We consider next $\Equatre$ defined in \eqref{eq:iv-split}.
\begin{lemma}[Estimate of $\Equatre$]\label{lem:iv.ii}
Let $\gamma <0$  and $q \in [0, d+1 +\frac{d}{2s} \gamma]$. Then  there exists a constant $\underline{a}^{(3)}>0$, only depending on the parameters $d,\gamma,s, 2 \wedge q$ and the constants $c_b$, $a_{get}$, $\Ccor$ and $\Cget$, such that
\begin{equation}\label{eq:iv.ii-gamma-neg}
  \Equatre \leq \frac12 \Ccor  \norm{\Big((f-A)_+\wedge \kappa\Big)^{q_0}}_{L^{p_0}_{k_0}} + \frac14 \Cget \fraka^{1 + \frac{2s}{d}} \norm{\Big((f-A)_+\wedge \kappa\Big)^{q_0-1}}_{L^1_{l_q}}
\end{equation}
(with $\Ccor$, $a_{get}$ and $\Cget$ from Propositions~\ref{p:cor-gamma-neg} and \ref{p:coercivity-get}) as soon as $a  \ge \underline{a}^{(3)}$.
\end{lemma}
To prove this lemma, we show that we can split $\Equatre$ as follows.
\begin{lemma}[Splitting $\Equatre$]\label{lem:split-plus-wedge}
We have  \(\Equatre \le \Equatre^{(i)} + \Equatre^{(ii)} \) with
\begin{align*}
 & \Equatre^{(i)}  =   C_3^{(i)} \norm{(f \wedge A)\langle \cdot \rangle^{l_1} }_{L^{p_1}} \norm{ \big((f-A)_+\wedge \kappa\big)^{q_0-1}\langle \cdot \rangle^{-\frac{2sl_1}{d+2s}}}_{L^{p_2}} \; \norm{\big((f-A)_+\wedge \kappa\big)\langle \cdot\rangle^{-\frac{d l_1}{d+2s}} }_{L^{q_3}}  , \\
 &\Equatre^{(ii)} = q_0 c_b M_0 \fraka \norm{\big((f-A)_+\wedge \kappa\big)^{q_0-1}}_{L^1_{-q}}
\end{align*}
  where $l_1 = \min (q,2)$,
  and $p_1, p_2, p_3, q_3 \in [1,+\infty]$ satisfy
\begin{equation}\label{eq:holder-exp}
	\frac{1}{p_1} + \frac{1}{p_2} + \frac{1}{p_3} = 1, \qquad 1 + \frac{1}{p_3} = \frac{1}{q_3} +\frac{-\gamma}{d}.
      \end{equation}
      The constant $C_3^{(i)}$ is given by \(2^{\frac{d l_1}{2(d+2s)}}q_0 c_b \norm{| \cdot |^\gamma}_{L^{\frac{-\gamma}d,\infty}}\).
\end{lemma}
\begin{proof}
We use that for $w \in B_1$ and $v \in \R^d$, we have,
\[
	\langle v-w\rangle^2 \leq 1 + \abs{v}^2 + \abs{w}^2 \leq 2 + \abs{v}^2 \leq 2 \langle v \rangle^2.
\]
We now distinguish two domains of integration: $|w| \le 1$ and $|w|\ge 1$. We use the fact that $\gamma<0$ to get $|w|^\gamma \le 1$ for $|w|\ge 1$ and we write for $l_1 = \min(q, 2)$,
\begin{align*}
\Equatre  &= q_0 c_b\int_{\R^d} (f \wedge A) \big((f-A)_+\wedge \kappa\big)^{q_0 -1} \left(\int_{B_1}  \big((f-A)_+\wedge \kappa\big)(v-w) \abs{ w}^{\gamma} \dd w \right) \dd v\\
  &\quad + q_0c_b\int_{\R^d} (f \wedge A)  \big((f-A)_+\wedge \kappa\big)^{q_0 -1} \left(\int_{\R^d \setminus B_1} \big((f-A)_+\wedge \kappa\big)(v-w) \abs{ w}^{\gamma} \dd w \right) \dd v\\
  \intertext{we use $\gamma < 0$ and $|w|\ge 1$ so that we can bound the inner integrand by the mass due to \eqref{e:hydro}} 
  &\leq q_0c_b \int_{\R^d} (f \wedge A)\langle v \rangle^{l_1} \langle v \rangle^{-l_1}  \big((f-A)_+\wedge \kappa\big)^{q_0 -1} \left(\int_{B_1}  \big((f-A)_+\wedge \kappa\big)(v-w) \abs{ w}^{\gamma} \dd w \right) \dd v\\
  &\quad + q_0 c_bM_0 \fraka \int_{\R^d} \big((f-A)_+\wedge \kappa\big)^{q_0 -1} \langle v \rangle^{-q} \dv\\
  \intertext{we use $\langle v-w\rangle^2  \leq 2 \langle v \rangle^2$}
          &\leq 2^{\frac{d l_1}{2(d+2s)}}q_0 c_b \int_{\R^d} (f \wedge A)\langle v \rangle^{l_1}   \big((f-A)_+\wedge \kappa\big)^{q_0 -1}  \langle v \rangle^{-\frac{2sl_1}{d+2s}}\\
  &\qquad \qquad \qquad \times \left(\int_{B_1}\langle v-w \rangle^{-\frac{d l_1}{d+2s}}  \big((f-A)_+\wedge \kappa\big)(v-w) \abs{ w}^{\gamma} \dd w \right) \dd v\\
  &\quad + q_0 c_bM_0 \fraka \int_{\R^d}  \big((f-A)_+\wedge \kappa\big)^{q_0 -1} \langle v \rangle^{-q} \dv\\
\intertext{we use weak Young's inequality,}
          & \leq 2^{\frac{d l_1}{2(d+2s)}}q_0 c_b \norm{| \cdot |^\gamma}_{L^{\frac{-\gamma}d,\infty}}
            \norm{(f \wedge A)\langle \cdot \rangle^{l_1} }_{L^{p_1}} \norm{  \big((f-A)_+\wedge \kappa\big)^{q_0-1}\langle \cdot \rangle^{-\frac{2sl_1}{d+2s}}}_{L^{p_2}} \\
  & \qquad \times \norm{ \big((f-A)_+\wedge \kappa\big)\langle \cdot\rangle^{-\frac{d l_1}{d+2s}} }_{L^{q_3}} \\
  &\quad +  q_0 c_bM_0 \fraka \int_{\R^d}  \big((f-A)_+\wedge \kappa\big)^{q_0 -1} \langle v \rangle^{-q} \dv. \qedhere
\end{align*}
\end{proof}
Then we estimate each piece separately.
\begin{lemma}[Estimate of $\mathcal{E}^{(ii)}_3$]\label{lem:plus-wedge-b}
  \begin{equation}\label{eq:E22+a}
	 \Equatre^{(ii)}  \leq \frac18 \Cget \fraka^{1+\frac{2s}{d}} \norm{ \big((f-A)_+\wedge \kappa\big)^{q_0-1}}_{L^1_{l_q}}
       \end{equation}
       as soon as $\fraka \ge \underline{a}^{(3,ii)}:=\left( \frac{8 q_0 c_bM_0}{\Cget}\right)^{\frac{d}{2s}}$. 
\end{lemma}
\begin{proof}
We note that \(-q \leq l_q \) for $q \leq 2 + \frac{d}{2s} k_0$, and $1 < 1 + \frac{2s}{d}$, so that the conclusion holds true.
\end{proof}
\begin{lemma}[Estimate of $\Equatre^{(i)}$]\label{lem:plus-wedge-a}
  There exists $\underline{a}^{(3,i)}$, that only depends on $d,\gamma,s,E_0,l_1$, $c_b$, $\Ccor$ and $\Cget$
  (see Propositions~\ref{p:coercivity-1} and \ref{p:coercivity-get}), such that
  \[
	\Equatre^{(i)}\leq \frac12 \Ccor \norm{\big((f-A)_+\wedge \kappa\big)^{q_0}}_{L^{p_0}_{k_0}}  + \frac18 \Cget \fraka^{1+\frac{2s}{d}}  \norm{\big((f-A)_+\wedge \kappa\big)^{q_0-1}}_{L^1_{l_q}}
      \]
      as soon as $\fraka \ge   \underline{a}^{(3,i)}$. 
\end{lemma}
\begin{proof}
For $\Equatre^{(i)}$,  we first notice that,
\begin{equation*}
	\norm{(f \wedge A) \langle \cdot \rangle^{l_1} }_{L^{p_1}} \leq \norm{(f \wedge A)\langle \cdot \rangle^{l_1}}_{L^1}^{\frac{1}{p_1}} \norm{(f\wedge A)\langle \cdot \rangle^{l_1} }_{L^\infty}^{1- \frac{1}{p_1}} \leq  E_0^{\frac{1}{p_1}} \fraka^{1- \frac{1}{p_1}}. 
\end{equation*}

Next we choose $p_3 = \infty$ and $p_2 = \frac{q_3}{q_0-1} \ge 1$ in \eqref{eq:holder-exp}, which implies
\[
	\frac{1}{q_3} = \frac{d+\gamma}{d}, \qquad \frac{1}{p_2} = \frac{2s(d+\gamma)}{d^2}, \qquad 1-\frac{1}{p_1} = \frac{2s(d+\gamma)}{d^2}.
\]
Thus 
\begin{equation*}
\begin{aligned}
	\Equatre^{(i)}&= C_3^{(i)} \norm{(f \wedge A)\langle \cdot \rangle^{l_1} }_{L^{p_1}} \norm{  \big((f-A)_+\wedge \kappa\big)^{q_0-1}\langle \cdot \rangle^{-\frac{2sl_1}{d+2s}}}_{L^{p_2}} \norm{  \big((f-A)_+\wedge \kappa\big)\langle \cdot\rangle^{-\frac{d l_1}{d+2s}}}_{L^{q_3}} \\
	&\leq C_3^{(i)} E_0^{\frac{1}{p_1}} \fraka^{\frac{2s(d+\gamma)}{d^2}} \left(\int  \big((f-A)_+\wedge \kappa\big)^{\frac{d}{d+\gamma}} \langle v\rangle^{-\frac{d^2 l_1}{(d+2s)(d+\gamma)}} \dd v\right)^{\frac{(d+\gamma)(d+2s)}{d^2}}.
\end{aligned}
\end{equation*}

\noindent \textit{Step 1: First interpolation.}
Then we interpolate with Lemma \ref{lem:interpolation-q0-1} for $k_1 = -\frac{d^2 l_1}{(d+2s)(d+\gamma)}$
\begin{equation*}
	\int_{\R^d} \big((f-A)_+\wedge \kappa\big)^{\frac{d}{d+\gamma}} \langle v \rangle^{k_1}\dd v \leq \norm{\big((f-A)_+\wedge \kappa\big)^{q_0}}_{L^{p_0}_{k_0}}^{\alpha_1}\norm{(f-A)_+\wedge \kappa}_{L^1_{m_1}}^{\alpha_2},
\end{equation*}
where
\begin{align*}
	&\alpha_1 = \frac{-\gamma d}{4s(d+\gamma)}, \qquad \alpha_2 = \frac{4sd + \gamma(d+2s)}{4s(d+\gamma)}, \\
	&m_1 = - \frac{4sd^2}{(d+2s)(4sd + \gamma(d+2s))} l_1 + \frac{\gamma d}{4sd + \gamma(d+2s)} k_0.
\end{align*}
We can check that $\frac{-\gamma(d+2s)}{4s d} < 1$ 
since $-\gamma \in [0,2s]$ and $d \geq 2$. Together with Young's inequality, this implies that there exists $C = C(\Ccor, c_b,E_0,l_1,\gamma,d,s)$ such that, 
\begin{align*}
	\Equatre^{(i)}  
	&\leq C_3^{(i)} E_0^{\frac{1}{p_1}}  \fraka^{\frac{2s(d+\gamma)}{d^2}} \norm{\big((f-A)_+\wedge \kappa\big)^{q_0}}_{L^{p_0}_{k_0}}^{\frac{-\gamma(d+2s)}{4s d}}\norm{(f-A)_+\wedge \kappa}_{L^1_{m_1}}^{\frac{(d+2s)(4sd + \gamma(d+2s))}{4sd^2}} \\
	&\leq \frac14 \Ccor \norm{\big((f-A)_+\wedge \kappa\big)^{q_0}}_{L^{p_0}_{k_0}} + C \fraka^{\gamma_1} \norm{(f-A)_+\wedge \kappa}_{L^1_{m_1}}^{\frac{d+2s}{d}}
\end{align*}
where $\gamma_1$ is given in Lemma~\ref{lem:aux-algebra}, see \eqref{eq:gammas}.
\medskip

\noindent \textit{Step 2: Second interpolation.}
Now we pick $\beta, \beta' \geq 1$ as in \eqref{eq:choice-theta-beta}, so that \eqref{eq:weight-condition} implies together with Hölder's inequality
\begin{equation*}
\begin{aligned}
	\norm{(f-A)_+\wedge \kappa}_{L^1_{m_1}}&\leq \norm{(f-A)_+\wedge \kappa}_{L^1_{\frac{2}{\beta_1'} + \frac{m_3}{\beta_1}}}\leq  \norm{(f-A)_+\wedge \kappa}_{L^1_{2}}^{\frac{1}{\beta_1'}} \norm{(f-A)_+\wedge \kappa}_{L^1_{m_3}}^{\frac{1}{\beta_1}}\\
	&\leq E_0^{\frac{1}{\beta_1'} }\norm{(f-A)_+\wedge \kappa}_{L^1_{m_3}}^{\frac{1}{\beta_1}}.
\end{aligned}
\end{equation*}	
To bound the right hand side of the previous inequality, we interpolate again, now with Lemma \ref{lem:interpolation-q0-2} and $k_2 = m_3$,
\begin{equation*}
	\norm{(f-A)_+\wedge \kappa}_{L^1_{m_3}} \leq \norm{\big((f-A)_+\wedge \kappa\big)^{q_0}}_{L^{p_0}_{k_0}}^{\alpha_3} \norm{\big((f-A)_+\wedge \kappa\big)^{q_0-1}}_{L^1_{m_2}}^{\alpha_4},
\end{equation*}	
where 
\begin{equation*}
	\alpha_3 = \frac{d(d-2s)}{d^2 + 4s^2}, \qquad \alpha_4 = \frac{4sd}{d^2 +4s^2}, \qquad m_2 = \frac{4s^2 + d^2}{4sd}\left[m_3 - \frac{d(d-2s)}{d^2 + 4s^2}k_0\right].
\end{equation*}	
Note that the choice of $m_3$ in Lemma \ref{lem:aux-algebra} is such that $m_2 = l_q$, that is $m_3 = \frac{4sd}{d^2 +4s^2} l_q + \frac{d(d-2s)}{d^2 + 4s^2}k_0$.
In particular, we find
\begin{align*}
	\norm{(f-A)_+\wedge \kappa}_{L^1_{m_1}}^{\frac{d+2s}{d}} &\leq E_0^{\frac{1}{\beta'} \frac{d+2s}{d}}\norm{(f-A)_+\wedge \kappa}_{L^1_{m_3}}^{\frac{1}{\beta}\frac{d+2s}{d}}\\
	&\leq E_0^{\frac{1}{\beta'} \frac{d+2s}{d}}\norm{\big((f-A)_+\wedge \kappa\big)^{q_0}}_{L^{p_0}_{k_0}}^{\gamma_2} \norm{\big((f-A)_+\wedge \kappa\big)^{q_0-1}}_{L^1_{l_q}}^{\gamma_3}
\end{align*}
where $\gamma_2$ and $\gamma_3$ also come from Lemma~\ref{lem:aux-algebra}, see \eqref{eq:gammas}.
\medskip

\noindent \textit{Step 3: Young's inequality.}
We combine Step 1 and Step 2 with Young's inequality, so that for $\theta \geq 1$ and any $\varepsilon \in (0, 1)$
\begin{align*}
	\Equatre^{(i)} 
	&\leq \frac14 \Ccor \norm{\big((f-A)_+\wedge \kappa\big)^{q_0}}_{L^{p_0}_{k_0}} + C \fraka^{\gamma_1} \norm{\big((f-A)_+\wedge \kappa\big)^{q_0}}_{L^{p_0}_{k_0}}^{\gamma_2} \norm{\big((f-A)_+\wedge \kappa\big)^{q_0-1}}_{L^1_{l_q}}^{\gamma_3}\\
	&\leq \frac14 \Ccor \norm{\big((f-A)_+\wedge \kappa\big)^{q_0}}_{L^{p_0}_{k_0}} +  \frac14 \Ccor  \norm{\big((f-A)_+\wedge \kappa \big)^{q_0}}_{L^{p_0}_{k_0}} + \bar C \fraka^{\gamma_1 \theta'}\norm{\big((f-A)_+\wedge \kappa\big)^{q_0-1}}_{L^1_{l_q}}^{\gamma_3 \theta'}
\end{align*}
with $\bar C = \bar C(\Ccor, c_b,E_0,l_1,\gamma,d,s)$ and  $\theta'= \frac{\theta}{1-\theta}$, see Lemma~\ref{lem:aux-algebra} again. We used that $\gamma_2 \theta =1$, see \eqref{eq:algebra-eq}. To complete the proof of Lemma \ref{lem:plus-wedge-a}, we  use  \eqref{eq:algebra-ineq} from Lemma \ref{lem:aux-algebra}. 
\end{proof}
\begin{proof}[Proof of Lemma \ref{lem:iv.ii}]
Combine Lemmas \ref{lem:split-plus-wedge}, \ref{lem:plus-wedge-b} and \ref{lem:plus-wedge-a}.
\end{proof}

\subsection{The term $\Etrois$}
The estimate of $\Etrois$ defined in \eqref{eq:iv-split} follows with similar methods.
\begin{lemma}[Estimate of $\Etrois$]\label{lem:iv.iii}
  Let $f: [0, \infty) \to \R$ have finite mass $M_0$ and finite $2$-moment $E_0$, and let $q \in [0,\frac{d}{2s} k_0]$.
  
Then  there exists a constant $\underline{a}^{(4)}>0$, only depending on the parameters $d,\gamma,s, \min (2, q)$ and the constants $c_b$, $\Ccor$ and $\Cget$ (see Propositions~\ref{p:coercivity-1} and \ref{p:coercivity-get}), such that 
\begin{equation}\label{eq:iv.ii.b}
\Etrois \leq   \frac14 \Ccor  \norm{\Big((f-A)_+\wedge \kappa\Big)^{q_0}}_{L^{p_0}_{k_0}} + \frac18 \Cget \fraka^{1 + \frac{2s}{d}} \norm{\Big((f-A)_+\wedge \kappa\Big)^{q_0-1}}_{L^1_{l_q}}
\end{equation}
(with $\Ccor$ and $\Cget$ from Propositions~\ref{p:cor-gamma-neg} and \ref{p:coercivity-get}) as soon as $a \ge \underline{a}^{(4)}$. 
\end{lemma}
To prove this lemma, we split $\Etrois$ as follows.
\begin{lemma}[Splitting $\Etrois$]\label{lem:split-Phi}
Let $p_4, p_5, q_5 \in [1,+\infty]$ satisfy
\begin{equation}\label{eq:holder-exp-Phi}
	\frac{1}{p_4} + \frac{1}{p_5} = 1, \qquad 1 +\frac{1}{p_5} = \frac{1}{q_5} + \frac{-\gamma}{d}.
\end{equation}
We have  \(\Etrois \le  \Etrois^{(i)} +  \Etrois^{(ii)}\) with
\begin{align*}
 \Etrois^{(i)} & =  C_4^{(i)}    \norm{\big((f-A)_+\wedge \kappa\big)^{q_0}\langle \cdot \rangle^{-2} }_{L^{p_4}} \norm{ (f-A)_+\langle \cdot \rangle^{2}}_{L^{q_5}}  , \\
 \Etrois^{(ii)} & = M_0(q_0-1)c_b  \norm{(f-A)_+\wedge \kappa}_{L^{q_0}}^{q_0},
\end{align*}
with $C_4^{(i)} = 4 (q_0-1)c_b\norm{| \cdot |^\gamma}_{L^{\frac{-\gamma}d,\infty}}$.
\end{lemma}
\begin{proof}
We use for $w \in B_1$ and $v \in \R^d$ 
\begin{equation*}
	\abs{v} \leq \abs{v-w} + \abs{w} \leq \abs{v-w} + 1,
\end{equation*}
which implies $\langle v \rangle \leq 2 \langle v-w \rangle$,
so that distinguishing two domains of integration, $|w| \le 1$ and $|w|\ge 1$, we find
\begin{align*}
	\Etrois 
	&= (q_0-1)c_b \int_{\R^d}  \big((f-A)_+\wedge \kappa\big)^{q_0} \langle v \rangle^{-2} \langle v \rangle^{2} \left( \int_{B_1} f(v-w)  \abs{w}^\gamma \dd w \right) \dd v \\
	&\quad + (q_0-1)c_b \int_{\R^d} \big((f-A)_+\wedge \kappa\big)^{q_0}\left( \int_{\R^d \setminus B_1}f(v-w)  \abs{w}^\gamma \dd w \right) \dd v \\
	&\leq 4 (q_0-1)c_b \int_{\R^d} \big((f-A)_+\wedge \kappa\big)^{q_0} \langle v \rangle^{-2}  \left( \int_{B_1} \langle v-w \rangle^{2}f(v-w)  \abs{w}^\gamma \dd w \right) \dd v \\
	&\quad + M_0 (q_0-1)c_b \int_{\R^d}\big((f-A)_+\wedge \kappa\big)^{q_0}\dd v \\
	&\leq 4 (q_0-1)c_b \norm{| \cdot |^\gamma}_{L^{\frac{-\gamma}d,\infty}}  \norm{\big((f-A)_+\wedge \kappa\big)^{q_0}\langle \cdot \rangle^{-2} }_{L^{p_4}} \norm{ f\langle \cdot \rangle^{2}}_{L^{q_5}} \\
  & \quad + M_0(q_0-1)c_b  \norm{(f-A)_+\wedge \kappa}_{L^{q_0}}^{q_0}. \qedhere
\end{align*}
\end{proof}
Again we estimate each piece separately.
\begin{lemma}[Estimate of $\Etrois^{(ii)}$]\label{lem:Phi-b}
There holds
\[
  \Etrois^{(ii)}  \leq \frac18 \Ccor \norm{\big((f-A)_+\wedge \kappa\big)^{q_0}}_{L^{p_0}_{k_0}} + \frac1{16}\Cget \fraka^{1+\frac{2s}{d}}\norm{\big((f-A)_+\wedge \kappa\big)^{q_0-1}}_{L^1_{l_q}}
\]
(with $\Ccor$ and $\Cget$ from Propositions~\ref{p:coercivity-1} and \ref{p:coercivity-get}) as soon as $\fraka \ge \underline{a}^{(4,ii)}$ with $\underline{a}^{(4,ii)}>0$ only depending on $\Ccor,c_b,q_0,\Cget,d,$ $\gamma,$ $s,M_0$. 
\end{lemma}
\begin{proof}
We use Lemma \ref{lem:interpolation-q0-3} with $k_3 = 0$ to get for some constant $C = C(\Ccor, q_0, M_0,c_b,d,\gamma,s)$,
\[
	\Etrois^{(ii)}  \leq \frac18 \Ccor \norm{\big((f-A)_+\wedge \kappa\big)^{q_0}}_{L^{p_0}_{k_0}} + C \norm{\big((f-A)_+\wedge \kappa\big)^{q_0-1}}_{L^1_{m_3}},
\]
with $m_3 = -\frac{d}{2s} k_0 - 2$. For $q \leq 2 + \frac{d}{2s} k_0$, the exponent $m_3$ satisfies 
\(
	m_3 \leq l_q .
\)
\end{proof}
\begin{lemma}[Estimate of $\Etrois^{(i)}$]\label{lem:Phi-a}
There holds
\[
	\Etrois^{(i)} \leq \frac18 \Ccor \norm{\big((f-A)_+\wedge \kappa\big)^{q_0}}_{L^{p_0}_{k_0}}  + \frac1{16} \Cget \fraka^{1+\frac{2s}{d}}  \norm{\big((f-A)_+\wedge \kappa\big)^{q_0-1}}_{L^1_{l_q}}
\]
as soon as $\fraka \ge \underline{a}^{(4,i)}$ with $\underline{a}^{(4,i)}>0$ only depending on $\Ccor,c_b,\Cget, d,\gamma,s,M_0$ (with $\Ccor$ and $\Cget$ from Propositions~\ref{p:coercivity-1} and \ref{p:coercivity-get}). 
\end{lemma}
\begin{proof}
We choose $q_5 = 1$ in \eqref{eq:holder-exp-Phi}. Then 
\begin{equation*}
	\frac{1}{p_5} = \frac{-\gamma}{d} \in (0, 1), \qquad \frac{1}{p_4} = 1 + \frac{\gamma}{d} = \frac{d +\gamma}{d} \in (0, 1),
\end{equation*}
and thus
\begin{align*}
	\Etrois^{(i)}& \le 4 (q_0-1)c_b \norm{\big((f-A)_+\wedge \kappa\big)^{q_0}\langle \cdot \rangle^{-2} }_{L^{p_4}} \norm{ f\langle \cdot \rangle^{2}}_{L^{q_5}}\\
	& \le 4 (q_0-1)c_b   E_0 \left(\int \big((f-A)_+\wedge \kappa\big)^{\frac{d+2s}{d+\gamma}} \langle v \rangle^{-2 \frac{ d}{d+\gamma}} \dd v\right)^{\frac{d+\gamma}{d}}.
\end{align*}

\noindent  \textit{Step 1: First interpolation.}
We then interpolate thanks to Lemma \ref{lem:interpolation-q0-1} with $p_1 = \frac{d+2s}{d+\gamma} \in [1,p_0 q_0]$,
\begin{align*}
	\Bigg(\int \big((f-A)_+\wedge \kappa\big)^{\frac{d+2s}{d + \gamma}}&\langle v \rangle^{-2\frac{ d}{d+\gamma}} \dd v\Bigg)^{\frac{d+\gamma}{d}} \\
	&\leq  \left(\int  \big((f-A)_+\wedge \kappa\big)^{q_0p_0} \langle v \rangle^{k_0p_0} \dd v \right)^{\frac{\alpha_5}{p_0}}\left( \int  \big((f-A)_+\wedge \kappa\big) \langle v \rangle^{m_4} \dd v \right)^{\alpha_6},
\end{align*}
where $\alpha_5, \alpha_6, m_4$ are given by
\begin{equation*}
	\alpha_5 = \frac{2s - \gamma}{4s}, \qquad \alpha_6 = \frac{(d+2s)(\gamma + 2s)}{4sd}, \qquad m_4 = - 2\frac{4sd}{(d+2s)(\gamma + 2s)}  - \frac{(2s-\gamma) d}{(d+2s)(\gamma +2s)} k_0.
\end{equation*}      
We remark that $\alpha_5 \in (0,1)$. Then Young's inequality yields for any $\varepsilon_1 \in (0,1)$
\[	 \Etrois^{(i)} \leq \frac1{16}\Ccor \norm{ \big((f-A)_+\wedge \kappa\big)^{q_0}}_{L^{p_0}_{k_0}} + C(\Ccor, M_0, E_0, c_b)\norm{(f-A)_+\wedge \kappa}_{L^1_{m_4}}^{\frac{d+2s}{d}}.
\]

\noindent \textit{Step 2: Second interpolation.}
We pick $\beta, \beta' \geq 1$ as in \eqref{eq:choice-theta-beta}. In particular, \eqref{eq:weight-condition-IV} allows us to use Hölder's inequality in order to get,
\[
	\norm{(f-A)_+\wedge \kappa}_{L^1_{m_4}}\leq  \norm{(f-A)_+\wedge \kappa}_{L^1_{2}}^{\frac{1}{ \beta'}}\norm{(f-A)_+\wedge \kappa}_{L^1_{m_3}}^{\frac{1}{\beta}} \leq E_0^{\frac{1}{ \beta'}} \norm{(f-A)_+\wedge \kappa}_{L^1_{m_3}}^{\frac{1}{\beta}}.
\]
If we now apply the interpolation Lemma \ref{lem:interpolation-q0-2} for $k_2 = m_3$, we find
\[
	\int_{\R^d}\big((f-A)_+\wedge \kappa\big) \langle v \rangle^{m_3} \dd v \leq \left(\int_{\R^d} \big((f-A)_+\wedge \kappa\big)^{q_0p_0} \langle v \rangle^{k_0p_0} \dd v \right)^{\frac{\alpha_3}{p_0}}\left( \int_{\R^d}   \big((f-A)_+\wedge \kappa\big)^{q_0-1}\langle v \rangle^{m_2} \dd v \right)^{\alpha_4},
\]
where 
\[
	\alpha_3 = \frac{d(d - 2s)}{4s^s + d^2}, \qquad \alpha_4 = \frac{4sd}{4s^2 + d^2},\qquad m_2 = \frac{4s^2 + d^2}{4sd} \left[m_3- \frac{d(d-2s)}{4s^2 + d^2} k_0\right].
\]
Again the choice of $m_3$ in Lemma \ref{lem:aux-algebra} is such that $m_2 = l_q$.
In particular, there holds
\[
	\norm{(f-A)_+\wedge \kappa}_{L^1_{m_4}}^{\frac{d+2s}{d}} \leq E_0^{\frac{1}{ \beta'}\frac{d+2s}{d}} \norm{\big((f-A)_+\wedge \kappa\big)^{q_0}}_{L^{p_0}_{k_0}}^{\gamma_2} \norm{(f-A)_+\wedge \kappa}_{L^1_{l_q}}^{\gamma_3}
\]
with $\gamma_2, \gamma_3$ given by \eqref{eq:gammas}.

\medskip

\noindent \textit{Step 3: Young's inequality.}
We combine Step 1 and Step 2 with Young's inequality, so that for $\theta \geq 1$,
\begin{equation*}
\begin{aligned}
	\Etrois^{(i)}  &\leq \frac1{16} \Ccor \norm{\big((f-A)_+\wedge \kappa\big)^{q_0}}_{L^{p_0}_{k_0}} + C(\Ccor) \norm{\big((f-A)_+\wedge \kappa\big)^{q_0}}_{L^{p_0}_{k_0}}^{\gamma_2} \norm{\big((f-A)_+\wedge \kappa\big)^{q_0-1}}_{L^1_{l_q}}^{\gamma_3}\\
	&\leq \frac1{16} \Ccor  \norm{\big((f-A)_+\wedge \kappa\big)^{q_0}}_{L^{p_0}_{k_0}} +  \frac1{16} \Ccor  \norm{(f-A)_+^{q_0}}_{L^{p_0}_{k_0}}^{\gamma_2 \theta} + \bar C(\Ccor) \norm{\big((f-A)_+\wedge \kappa\big)^{q_0-1}}_{L^1_{l_q}}^{\gamma_3 \theta'}.
\end{aligned}
\end{equation*}
If we pick $\theta$ as in \eqref{eq:choice-theta-beta}, then  Lemma \ref{lem:aux-algebra} implies \eqref{eq:algebra-eq} and we conclude by choosing $\underline{a}^{(4,i)}=\underline{a}$ such that
$\frac1{16} \Cget \underline{a}^{1+\frac{2s}d} \ge \bar C(\Ccor)$.
\end{proof}
\begin{proof}[Proof of Lemma \ref{lem:iv.iii}]
The result follows as a consequence of Lemmas \ref{lem:split-Phi}, \ref{lem:Phi-b} and \ref{lem:Phi-a}.
\end{proof}

\subsection{The error term due to truncation for soft potentials}

Finally it remains to estimate $\mathcal  E_5$ in \eqref{eq:iv-split}.
\begin{lemma}[Estimate of  $\mathcal E_5$]\label{lem:iv.iv}
Let $\gamma \in (- 2s,0)$ and  $q \le \qnsl$ where $\qnsl$ is given in \eqref{itm:qnsl}. There exists $\underline{a}^{(5)}>0$, only depending on $d,\gamma,s$ and hydrodynamical bounds from \eqref{e:hydro} such that, if $\ma \ge \underline{a}^{(5)}$, then
\begin{align*}
  \mathcal{E}_5\leq &\left( \frac{1}{16} \Cget \, \fraka^{\frac{d+2s}{d}}  + \tilde C_5 \mathfrak{a} \right) \int_{\R^d}\Big((f-A)_+\wedge \kappa\Big)^{q_0-1} \langle v \rangle^{l_q}\dd v\\
  & +\frac{1}{2} \Ccorbis \kappa^{\frac{2s}{d} q_0 }{N}^{-\frac{2s}{d}} \int_{\R^d} (f-A-\kappa)_+ \langle v \rangle^\gamma \dv +  C_5 \Bigg(\int_{\R^d}(f-A-\kappa)_+^{p_0} \langle v \rangle^{k_0p_0}  \dd v \Bigg)^{\frac{1}{p_0}},
\end{align*}
with $$N = \int_{\R^d} ((f-A)_+ \wedge \kappa)^{q_0-1} \langle v \rangle^{-(d-1)}  \dv.$$
The constant $\tilde C_5$ depends only on $q_0$, $c_b$, and $M_0$. 
If $\gamma \in (- \frac{2s}{d+2s}d,0)$, the positive constant $C_5$ only depends on $d,\gamma,s,c_b, E_0, M_0, \Cget$. If $\gamma \in (-2s,- \frac{2s}{d+2s}d)$, $C_5$ additionally depends on $\Ccorbis$. We refer to Proposition~\ref{p:coercivity-get} for $\Cget$ and to Proposition ~\ref{p:second-coercivity} for $\Ccorbis$.
\end{lemma}
\begin{proof}
To estimate $\mathcal{E}_5$, we first recall its definition,
\[
	\mathcal{E}_5= q_0  c_b \int_{\R^d}  A(v)\Big((f-A)_+\wedge \kappa\Big)^{q_0 - 1} \big((f-A - \kappa)_+\ast \abs{\cdot}^\gamma\big) \dd v.
\]

\noindent \textit{Step 1: Splitting.}
We split the integration domain as before and we find for a constant $m$ to be chosen later, 
\begin{align*}
	\mathcal{E}_5&= q_0  c_b \int_{\R^d}A(v)\Big((f-A)_+\wedge \kappa\Big)^{q_0 - 1} \langle v\rangle^{m}\langle v\rangle^{-m} \int_{B_1} (f-A - \kappa)_+(v-w) \abs{w}^\gamma \dd w \dd v\\
	& +q_0  c_b \int_{\R^d} A(v)\Big((f-A)_+\wedge \kappa\Big)^{q_0 - 1}  \int_{\R^d \setminus B_1} (f-A - \kappa)_+(v-w) \abs{w}^\gamma \dd w \dd v\\
	&\leq 2^{\frac{m}{2}} q_0  c_b \fraka \int_{\R^d} \langle v\rangle^{-q-m}\Big((f-A)_+\wedge \kappa\Big)^{q_0 - 1} \int_{B_1}  \langle v-w\rangle^{m}(f-A - \kappa)_+(v-w) \abs{w}^\gamma \dd w \dd v\\
	& +q_0  c_b M_0 \fraka \int_{\R^d} \Big((f-A)_+\wedge \kappa\Big)^{q_0 - 1}  \langle v\rangle^{-q} \dd v\\
	&=: \mathcal{E}_5^{(i)} + \mathcal{E}_5^{(ii)}.
\end{align*}
Note that in case that $m \geq 0$, we used $\langle v \rangle \leq \langle v-w\rangle$ for $w \in B_1$ and $v \in \R^d$. In case that $m < 0$, we used $\langle v-w\rangle \leq \sqrt{2} \langle v \rangle$ for $w \in B_1, v \in \R^d$.

The term $\mathcal{E}_5^{(ii)}$ is treated by checking that $-q \leq l_q$ holds true for any $q \leq d+1 + \frac{d\gamma}{2s}$.
\begin{equation} \label{e:e5ii}
\mathcal{E}_5^{(ii)} \le q_0  c_b M_0 \fraka \int_{\R^d} \Big((f-A)_+\wedge \kappa\Big)^{q_0 - 1}  \langle v\rangle^{l_q} \dd v.
\end{equation}
\medskip

\noindent \textit{Step 2: estimate of $\mathcal{E}_5^{(i)}$.}
We use again H\"older's and Young's inequalities to write,
\begin{align*}
  \mathcal{E}_5^{(i)} & \leq  2^{\frac{m}{2}} q_0  c_b \fraka \norm{\Big((f-A)_+\wedge \kappa\Big)^{q_0 - 1}\langle \cdot\rangle^{-(q+m)}}_{L^{\frac{d}{d+\gamma}}} \norm{ (f-A - \kappa)_+ }_{L^1_{m}} \\
  & \leq 2^{\frac{m}{2}} q_0  c_b   \fraka \kappa^{-\gamma \frac{2s}{d^2}} \left( \int_{\R^d} ((f-A)_+ \wedge \kappa)^{q_0-1} \langle v \rangle^{-(q+m)\frac{d}{d+\gamma}} \dv \right)^{\frac{d+\gamma}d} \norm{ (f-A - \kappa)_+ }_{L^1_{m}}.
\end{align*}
We used
\begin{align*}
	\left((f-A)\wedge \kappa\right)^{(q_0-1)\frac{d}{d+\gamma}} &=\left((f-A)\wedge \kappa\right)^{(q_0-1)}\left((f-A)\wedge \kappa\right)^{(q_0-1)\frac{-\gamma}{d+\gamma}}\\
	&\leq\left((f-A)\wedge \kappa\right)^{(q_0-1)} \kappa^{(q_0-1)\frac{-\gamma}{d+\gamma}}.
\end{align*}
It is now convenient to introduce some notation. We consider for any $\ell \in \R$,
\[ N_\ell = \int_{\R^d} ((f-A)_+ \wedge \kappa)^{q_0-1} \langle v \rangle^{\ell} \dv \quad \text{ and } \quad M_\ell = \int_{\R^d} (f-A - \kappa)_+  \langle v \rangle^\ell \dv.\]
We also let $k_m$ denote $-(q+m)\frac{d}{d+\gamma}$ and we write the previous inequality under the following compact form,
\begin{equation}\label{eq:start-E5}
  \mathcal{E}_5^{(i)}  \leq 2^{\frac{m}{2}} q_0  c_b   \fraka \kappa^{-\gamma \frac{2s}{d^2}} N_{k_m}^{\frac{d+\gamma}d} M_m.
\end{equation}
We distinguish two cases: $\frac{-2sd}{d+2s} \leq \gamma < 0$ and $-2s < \gamma < \frac{-2sd}{d+2s}$.

\medskip

\noindent \textit{Step 3: estimate of $\mathcal{E}_5^{(i)}$ for not too negative $\gamma$'s.}

We start by assuming $\frac{-2sd}{d+2s}\leq \gamma < 0$. We use Young's inequality, so that
\[
	 \mathcal{E}_5^{(i)}\leq \frac{1}{16} \Cget \fraka^{\frac{d}{d+\gamma}}  N_{k_m} +\left(\frac{1}{16} \Cget\right)^{\frac{d+\gamma}{\gamma}}\left(2^{\frac{m}{2}} q_0  c_b\right)^{-\frac{d}{\gamma}}\kappa^{\frac{2s}{d}}M_m^{\frac{d}{-\gamma}}.
\]
We pick $m = k_0 = \gamma + 2s - \frac{2s}{d}$.
We then observe that for any $\gamma \geq - \frac{2sd}{d+2s}$, there holds $\frac{d}{d+\gamma} \leq 1 + \frac{2s}{d}$, and moreover $k_m \leq l_q$, provided that $q \leq \qnsl =d+1 + \frac{\gamma d}{2s}$, $\gamma \geq - \frac{2sd}{d+2s}$ and $d \geq 2$, see Lemma \ref{lem:tedious-2}. This shows 
\[
	 \frac{1}{16} \Cget \fraka^{\frac{d}{d+\gamma}}   N_{k_m}   \leq \frac{1}{16} \Cget \fraka^{1+\frac{2s}{d}}\int_{\R^d}\Big((f-A)_+\wedge \kappa\Big)^{q_0-1} \langle v \rangle^{l_q}\dd v.
\]
Moreover, we use Hölder's  and Chebyshev's inequalities to bound
\begin{equation*}
\begin{aligned}
	\kappa^{\frac{2s}{d}}M_{k_0}^{\frac{d}{-\gamma}} &= \kappa^{\frac{2s}{d}}M_{k_0}^{\frac{d}{-\gamma}-1}M_{k_0}\leq \kappa^{\frac{2s}{d}}M_{k_0}^{\frac{d+\gamma}{-\gamma}}\Bigg(\int_{\R^d}  (f-A-\kappa)_+^{p_0}\langle v \rangle^{k_0p_0} \dd v \Bigg)^{\frac{1}{p_0}}\abs{\{f-A >\kappa\}}^{\frac{2s}{d}}\\
	&\leq \kappa^{\frac{2s}{d}}\kappa^{-\frac{2s}{d}}E_0^{\frac{d+\gamma}{-\gamma}}\Bigg(\int_{\R^d} (f-A-\kappa)_+^{p_0}\langle v \rangle^{k_0p_0} \dd v \Bigg)^{\frac{1}{p_0}}\Bigg(\int_{\R^d} (f-A)_+ \dd v\Bigg)^{\frac{2s}{d}}\\
	&= E_0^{\frac{d+\gamma}{-\gamma}}M_0^{\frac{2s}{d}}\Bigg(\int_{\R^d}  (f-A-\kappa)_+^{p_0}\langle v \rangle^{k_0p_0} \dd v \Bigg)^{\frac{1}{p_0}}.
\end{aligned}
\end{equation*}
We used that $k_0 < 2$ to bound $M_{k_0}$ from above by $E_0$.

Gathering the estimates, we get,
\begin{equation}\label{eq:E-a-plus-large-gamma}
\begin{aligned}
  \mathcal{E}_5^{(i)}  \leq& \frac{1}{16} \Cget \fraka^{1+\frac{2s}{d}}\int_{\R^d}\Big((f-A)_+\wedge \kappa\Big)^{q_0-1} \langle v \rangle^{l_q}\dd v \\
	&+\left(\frac{1}{16} \Cget\right)^{\frac{d+\gamma}{\gamma}}\left(2^{\frac{m}{2}} q_0  c_b\right)^{-\frac{d}{\gamma}}E_0^{\frac{d+\gamma}{-\gamma}}M_0^{\frac{2s}{d}}\Bigg(\int_{\R^d} (f-A-\kappa)_+^{p_0} \langle v \rangle^{k_0p_0} \dd v \Bigg)^{\frac{1}{p_0}}.
\end{aligned}
\end{equation}
Combining \eqref{e:e5ii} with \eqref{eq:E-a-plus-large-gamma} yields the announced estimate and concludes the proof of the lemma in case that $\gamma \in \big[-\frac{2sd}{d+2s}, 0\big)$.

\noindent \textit{Step 4: estimate of $\mathcal{E}_5^{(i)}$ for very negative $\gamma$'s.}
We now assume $-2s < \gamma < \frac{-2sd}{d+2s}$. It remains to bound $\mathcal{E}_5^{(i)}$ for this range of $\gamma$. We start from \eqref{eq:start-E5} and we recall $k_m = -(q+m)\frac{d}{d+\gamma}$ for some $m \in \R$ to be determined, 
\[	  \mathcal{E}_5^{(i)}  \leq 2^{\frac{m}{2}} q_0  c_b \fraka  {N}_{k_m}^{\frac{d+\gamma}{d}}M_m= 2^{\frac{m}{2}} q_0  c_b  \kappa^{\frac{-\gamma 2s}{d^2}} \fraka {N}_{k_m}^{\frac{d}{d+2s}}{N}_{k_m}^{\frac{d+\gamma}{d}-\frac{d}{d+2s}}M_m.
\]
Note that 
\[
	\frac{d+\gamma}{d}-\frac{d}{d+2s} < 0
\]
for $\gamma < - \frac{2sd}{d+2s}$.

We use Young's inequality
\[
	  \mathcal{E}_5^{(i)}\leq \frac{1}{16} \Cget \fraka^{\frac{d+2s}{d}} {N}_{k_m} + C_5^{(i)} \kappa^{-\frac{\gamma(d+2s)}{d^2}}{N}_{k_m}^{\left(\frac{d+\gamma}{d}-\frac{d}{d+2s}\right)\frac{d+2s}{2s}}M_m^{\frac{d+2s}{2s}},
\]
where $C_5^{(i)} = \left(\frac{1}{16} \Cget\right)^{-\frac{d}{2s}}\left(2^{\frac{m}{2}} q_0  c_b\right)^{\frac{d+2s}{2s}}$. We further bound for any $\alpha_0 \in [0, 1)$
\[
	M_m  \leq E_0^{\alpha_0}M_{m_{\alpha_0}}^{1-\alpha_0} \quad \text{ with } \quad m_{\alpha_0} :={\frac{m-2 \alpha_0}{1-\alpha_0}}
      \]
thanks to Lemma \ref{lem:interpolation-q0-0}. We found,
\begin{equation}\label{eq:aux-E5i}
	  \mathcal{E}_5^{(i)}\leq \frac{1}{16} \Cget \fraka^{\frac{d+2s}{d}}N_{k_m} + C_5^{(i)} E_0^{\alpha_0\frac{d+2s}{2s}} \kappa^{-\frac{\gamma(d+2s)}{d^2}}{N}_{k_m}^{\left(\frac{d+\gamma}{d}-\frac{d}{d+2s}\right)\frac{d+2s}{2s}}M_{m_{\alpha_0}}^{(1-\alpha_0)\frac{d+2s}{2s}}.
\end{equation}

Then, in order to be able later on to absorb the first term in the good extra term (Proposition~\ref{p:coercivity-get}), we need
\[
	k_m:= -(q+m)\frac{d}{d+\gamma} \leq l_q = k_0 + 2 \frac{2s}{d}  - q \frac{d+2s}{d}.
\]
For the second term, we will use later on  the second coercivity estimate from Proposition \ref{p:second-coercivity}. This is the reason why we require 
\[
	k_m:= -(q+m)\frac{d}{d+\gamma} \geq -(d-1),
\]
and 
\[
	m_{\alpha_0}:= \frac{m-2 \alpha_0}{1-\alpha_0} \leq k_0 +\frac{2s}{d} k_m.
\]
In order to satisfy the first two constraints, $q$ has to satisfy,
\[
	-(d-1) \leq l_q,
\]
which holds true for $q \le \qnsl$.

We pick $m$ such that
\[
	k_m= -(q+m)\frac{d}{d+\gamma} = -(d-1),
\]
and, in order to be able later on to we use Proposition \ref{p:second-coercivity},  we pick $k_m = -(d-1)$ for which we have
\[
	N :=  \int_{\R^d} \Big((f-A)_+\wedge \kappa\Big)^{q_0-1}\langle v \rangle^{-(d-1)}\dd v = {N}_{-(d-1)}.
\]
Thus we consider 
\[ \bar \theta = -\frac{\left(\frac{d+\gamma}{d}-\frac{d}{d+2s}\right)\frac{d+2s}{2s}}{\frac{2s}d} = \frac{-\gamma(d+2s) -2sd}{4s^2} \] 
and we find
\begin{equation*}
\begin{aligned}
	\kappa^{-\frac{\gamma(d+2s)}{d^2}}{N}_{k_m}^{\left(\frac{d+\gamma}{d}-\frac{d}{d+2s}\right)\frac{d+2s}{2s}}M_{m_{\alpha_0}}^{\frac{d+2s}{2s}(1-\alpha_0)}&= \kappa^{\frac{2s}{d} q_0\bar \theta }{N}^{-\frac{2s}{d}\bar \theta} M_{m_{\alpha_0}}^{\bar \theta}{\kappa}^{-\frac{\gamma(d+2s)}{d^2} -\frac{2s}{d} q_0 \bar\theta} M_{m_{\alpha_0}}^{\frac{(d+2s)(1-\alpha_0)}{2s}-\bar\theta}\\
	&= \kappa^{\frac{2s}{d} q_0\bar\theta }{N}^{-\frac{2s}{d}\bar \theta} M_{m_{\alpha_0}}^{\bar \theta} \kappa^{\frac{(\gamma+2s)(d+2s)}{2sd}}M_{m_{\alpha_0}}^{\frac{(d+2s)(\gamma+2s)}{4s^2}+\frac{d-\alpha_0(d+2s)}{2s}}.
\end{aligned}
\end{equation*}
We plug this into \eqref{eq:aux-E5i} and use Young's inequality, to get
\begin{equation}\label{eq:aux-E5i-2}
\begin{aligned}
	 \mathcal{E}_5^{(i)} &\leq  \frac{1}{16} \Cget \fraka^{\frac{d+2s}{d}} {N}_{k_m} +\frac{1}{8} C_{cor'} \kappa^{\frac{2s}{d} q_0 }{N}^{-\frac{2s}{d}} M_{m_{\alpha_0}} + C_5\kappa^{\frac{2s}{d}}M_{m_{\alpha_0}}^{1+\frac{2sd}{(\gamma+2s)(d+2s)} - \alpha_0 \frac{2s}{\gamma+2s}}.
\end{aligned}
\end{equation}
where $C_5 := \left(\frac{1}{8}C_{cor'}\right)^{\frac{\gamma(d+2s) + 2sd}{(\gamma+2s)(d+2s)}}  \left(C_5^{(i)} E_0^{\alpha_0\frac{d+2s}{2s}}\right)^{\frac{4s^2}{(\gamma+2s)(d+2s)}}$.
We used that
\[ \frac1{1-\theta} = \frac{4s^2}{(\gamma+2s)(d+2s)}.\]
Note that the estimate \eqref{eq:aux-E5i-2} holds for any $0 \leq \alpha_0 < 1$, but we want to pick $\alpha_0$, such that 
\begin{equation}\label{eq:alpha0}
	1+\frac{2sd}{(\gamma+2s)(d+2s)} - \alpha_0 \frac{2s}{\gamma+2s} = 1,
\end{equation}
but at the same time we need to make sure that 
\begin{equation}\label{eq:constraint-alpha0}
	m_{\alpha_0} = \frac{m-2 \alpha_0}{1-\alpha_0} \leq k_0 +\frac{2s}{d} k_m = \gamma,
\end{equation}
where we recall that we picked $m$ and $k_m$ as
\[
	m = -q +\frac{(d-1)(d+\gamma)}{d}, \qquad k_m = -(d-1),
\]
We verify \eqref{eq:constraint-alpha0} in Lemma \ref{l:tedious} below.

Finally, we conclude the proof just as in the case for larger $\gamma$ above: we use that $\frac{m-2 \alpha_0}{1-\alpha_0} \leq k_0$ and find with Chebyshev's inequality
\begin{equation*}
\begin{aligned}
	\kappa^{\frac{2s}{d}}M_{m_{\alpha_0}}   &\leq \kappa^{\frac{2s}{d}}\Bigg(\int_{\R^d}  (f-A-\kappa)_+^{p_0}\langle v \rangle^{k_0p_0} \dd v \Bigg)^{\frac{1}{p_0}}\abs{\{f-A >\kappa\}}^{\frac{2s}{d}}\\
	&\leq \kappa^{\frac{2s}{d}}\kappa^{-\frac{2s}{d}}\Bigg(\int_{\R^d}(f-A-\kappa)_+^{p_0} \langle v \rangle^{k_0p_0}\dd v \Bigg)^{\frac{1}{p_0}}\Bigg(\int_{\R^d} (f-A)_+ \dd v\Bigg)^{\frac{2s}{d}}\\
	&\leq M_0^{\frac{2s}{d}}\Bigg(\int_{\R^d}(f-A-\kappa)_+^{p_0} \langle v \rangle^{k_0p_0}  \dd v \Bigg)^{\frac{1}{p_0}}.
\end{aligned}
\end{equation*}
We plug this into \eqref{eq:aux-E5i-2} and use \eqref{eq:alpha0} to find
\begin{equation}\label{eq:E5-concl-small-gamma}
\begin{aligned}
	 \mathcal{E}_5^{(i)} &\leq\frac{1}{16} \Cget \fraka^{\frac{d+2s}{d}} {N} +\frac{1}{8} C_{cor'} \kappa^{\frac{2s}{d} q_0 }{N}^{-\frac{2s}{d}} M_{\gamma} + C_5\kappa^{\frac{2s}{d}}M_{m_{\alpha_0}}\\
	 &\leq\frac{1}{16} \Cget \fraka^{\frac{d+2s}{d}} {N} +\frac{1}{8} C_{cor'} \kappa^{\frac{2s}{d} q_0 }{N}^{-\frac{2s}{d}} M_{\gamma} + C_5M_0^{\frac{2s}{d}}\Bigg(\int_{\R^d}(f-A-\kappa)_+^{p_0} \langle v \rangle^{k_0p_0}  \dd v \Bigg)^{\frac{1}{p_0}}.
\end{aligned}
\end{equation}
Estimate \eqref{eq:E5-concl-small-gamma} concludes the proof of the statement for $\gamma\in \big(-2s, -\frac{2sd}{d+2s}\big)$.
\end{proof}
\begin{lemma} \label{lem:tedious-2}
 Let $d \geq 2$,  $s \in (0, 1)$,  $\gamma < 0$ and  $q \leq d+1 +\frac{d\gamma}{2s}$.
  Then 
  \begin{equation}\label{eq:constraint-tbc}
  -(q+k_0) \frac{d}{d+\gamma} \leq l_q = k_0 + 2 \frac{2s}{d}  - q \frac{d+2s}{d}.
  \end{equation} 
\end{lemma}
\begin{proof}
We rewrite \eqref{eq:constraint-tbc} as
\[
	\left(\frac{d+2s}{d} - \frac{d}{d+\gamma} \right) q \leq \left(1 + \frac{d}{d+\gamma}\right)k_0 + 2 \frac{2s}{d}.
\]
or also
\[
	\left(\gamma(d+2s) + 2sd\right) q \leq d\left(2d + \gamma \right)k_0 + 4s(d+\gamma) .
\]
If we now use $q \leq \frac{d}{2s} k_0 + 2$, we find
\[
	\left(\gamma(d+2s) + 2sd\right) \left(\frac{d}{2s} k_0 + 2\right) \leq d\left(2d + \gamma \right)k_0 + 4s(d+\gamma).
\]
or equivalently, 
\[
	\left(\gamma \frac{d^2}{2s} + d^2\right)k_0  \leq 2d^2 k_0 - 2\gamma d,
\]
which is satisfied if
\[
	0\leq \left(d-\gamma \frac{d}{2s}\right) k_0 - 2 \gamma .
\]
Now we use that  $k_0 = \gamma +2s - \frac{2s}{d}$. We find
\[
	0\leq \left(d-\gamma \frac{d}{2s}\right)\left(\gamma + 2s - \frac{2s}{d}\right) - 2\gamma = d\left( \frac{2s-\gamma}{2s}\right)\left(\gamma + 2s\right)  - (2s+ \gamma).
\]
We divide by $\gamma + 2s > 0$ to find
\[
	0 \leq \frac{d}{2s}(2s - \gamma) - 1,
\]	
which is true for any $d \geq 1$ and $\gamma < 0$. 
\end{proof}

\begin{lemma} \label{l:tedious}
 Let $d \geq 2$,    $s \in (0, 1)$,  $-2s < \gamma < -\frac{2sd}{d+2s}$ and $q \geq d-1-2s\frac{d-2}{d+2s}$.
  There exists $\alpha_0 \in (0, 1)$ such that \eqref{eq:alpha0} and\eqref{eq:constraint-alpha0} hold true.
\end{lemma}
\begin{proof}
In view of \eqref{eq:alpha0}, we pick
\[
	\alpha_0 = \frac{d}{d+2s} \in (0, 1).
\]
  We rewrite the constraint \eqref{eq:constraint-alpha0} as
\[
	\frac{d^2 - d - \gamma -d q}{d(2-\gamma)} \leq  \alpha_0
\]
and check that \eqref{eq:constraint-alpha0} is satisfied.
It is equivalent to,
\[ (d+2s) (d^2 - \gamma) - d (d+2s)(q+1) \le d^2 (2 - \gamma).\]
We now remark that for $q \geq d-1-2s\frac{d-2}{d+2s}$, we have
\[ (d+2s) (q+1) \ge d^2 + 4s.\]
We are thus left with checking that
\[ (d+2s) (d^2 - \gamma) - d (d^2 +4s) \le d^2 (2 - \gamma).\]
We rearrange terms and get,
\[ (-\gamma) (d+2s -d^2) - 4s d - (2-2s) d^2 \le 0. \]
We conclude by remarking that for $d \ge 2$ and $s \in (0,1)$, we have $d+2s -d^2 \le 0$. 
\end{proof}

\subsection{Final Estimate}

Combining Lemmas~\ref{lem:iv.i}, \ref{lem:iv.ii}, \ref{lem:iv.iii} and \ref{lem:iv.iv} implies the following statement.
\begin{lemma}[Error terms for soft potentials]\label{lem:iv}
  Let $\gamma <0$. Let $q \in [0, \qnsl]$ where $\qnsl$ is given by \eqref{itm:qnsl}.
  There exists a positive constant $\underline{a}$ only depending on $d,\gamma,s$, $M_0$, $E_0$, $c_b$, $\Cget$, $\Ccor$, $\Ccorbis$  (see Propositions~\ref{p:coercivity-1}, \ref{p:second-coercivity} and \ref{p:coercivity-get}) and $\Cqt$ (see formula~\eqref{e:cqt})  such that
\begin{multline}\label{eq:iv-gamma-neg-q0}
  \iint_{\R^d \times \Omega} \Phi\big(f\big)  \left(f \ast \abs{\cdot}^\gamma\right) \dd v \dd x \\
  \leq  \frac34 \Ccor  \int_{ \Omega} \norm{\big((f-A)_+\wedge \kappa\big)^{q_0}}_{L^{p_0}_{k_0}} \dd x+C_5 \fraka \int_{\Omega} \norm{ (f-A-\kappa)_+}_{L^{p_0}_{k_0}} \dx \\
  + \left( \frac{15}{16} \Cget \fraka^{1 + \frac{2s}{d}}  + \tilde C_5 \mathfrak{a}\right) \int_{ \Omega}\norm{\big((f-A)_+\wedge \kappa\big)^{q_0-1}}_{L^1_{l_q}} \dd x\\
  +\frac{1}{2} \Ccorbis \kappa^{\frac{2s}{d} q_0 }\int_{\Omega}\left[\norm{\big((f-A)_+\wedge \kappa\big)^{q_0-1}}_{L^1_{-(d-1)}}^{-\frac{2s}{d}} \int_{\R^d}  (f-A-\kappa)_+ \langle v \rangle^{\gamma} \dd v \right]\dd x, 
\end{multline}
as soon as $\fraka  \ge \underline{a}$. 
\end{lemma}
\begin{remark}
The constants $\Ccor$, $\Ccorbis$, $\Cget$ and $C_5, \tilde C_5$   appear in Propositions~\ref{p:cor-gamma-neg}, \ref{p:second-coercivity}, \ref{p:coercivity-get} and  Lemma~\ref{lem:iv.iv} respectively.
\end{remark}
\begin{proof}
Estimate~\eqref{eq:iv-gamma-neg-q0} is a consequence of the splitting of the error term as described in \eqref{eq:iv-split} and Lemmas~\ref{lem:iv.i}, \ref{lem:iv.ii}, \ref{lem:iv.iii} and \ref{lem:iv.iv}.
\end{proof}

\section{Proof of the decay estimate}

\subsection{Monotonicity for generation}

The next lemma is obtained by combining Lemma~\ref{l:weak-sols-q0} with (coercivity and error) estimates from previous sections. 
\begin{lemma}[Monotonicity]\label{l:decroissance}
Let $f$ be a suitable weak sub-solution of the Boltzmann equation in the sense of Definition~\ref{defi:suitable} with either in-flow, bounce-back, specular / diffuse / Maxwell reflection boundary condition.  Let
\[
	A(t,v) := \fraka(t) \langle v \rangle^{-\qnsl} \quad \text{ with } \quad \fraka(t):= \aast\left(1+ t^{-\frac{d}{2s}}\right) 
\]
with $\qnsl$ from \eqref{itm:qnsl} and  for some constant $\aast > 0$ large enough depending on $d,\gamma,s, c_b, M_0, E_0$ and $\Cb$. Then 
\begin{equation}\label{e:monotone}
	\frac{\dd}{\dd t} \iint_{\R^d \times \Omega} \phizk(f-A)\dv \dx \leq \mathcal{I}_\kappa(t) \quad \text{ in } \mathcal{D}'((0,T)),
\end{equation}
with $\mathcal{I}_\kappa(t)$ such that for any $T>0$, we have
\[
	\int_0^{T}  \mathcal{I}_\kappa(t) \dd t \xrightarrow[\kappa \to \infty]{}0.
\]
\end{lemma}
\begin{remark}
  The constant $C_b$ comes from the assumption on the boundary data $f_b$ in
  the statement of the main result, see Theorem~\ref{thm:decroissance}.
  The constant $c_b$ comes from the cancellation lemma, see \eqref{eq:cancellation}.
\end{remark}
\begin{proof}
We write $q$ instead of $\qnsl$ for clarity.   We assume that $\aast$ satisfies 
\begin{equation}\label{eq:condition-a-1}
\aast \ge \aget \quad \text{ and } \quad \aast \ge \Cb \quad \text{(and \quad $\aast \ge \underline{a}$ \quad if $\gamma<0$)}
\end{equation}
with $\aget$ given in \eqref{e:cond-a} and $\underline{a}$ comes from Lemma~\ref{lem:iv}.
The first condition ensures that we can use Proposition~\ref{p:coercivity-get} about the good extra term  while the second one ensures that boundary terms will not appear in the Truncated Convex Inequalities, see  Lemma~\ref{l:weak-sols-q0}. 

We distinguish the cases of hard and soft potentials. \bigskip

\paragraph{\sc The case of hard potentials.}
In this case,  we apply Lemma~\ref{l:weak-sols-q0} and
 collect estimates \eqref{eq:i-q0} with $g = (f-A)_+$, \eqref{eq:ii-q0}, \eqref{eq:iii-q0}, \eqref{eq:iv-q0} to get that for any $\varepsilon \in (0, 1)$,
\begin{align*}
	\frac{\dd}{\dd t} &\iint_{\R^d \times \Omega} \phizk(f-A)\dv \dx \\
	 &\leq - \Ccor \int_{\Omega} \norm{\big((f-A)_+\wedge\kappa\big)^{q_0}}_{L^{p_0}_{k_0}(\R^d)} \dd x+  \Ccor \int_\Omega \norm{\big((f-A)_+\wedge\kappa\big)^{q_0-1}}_{L^1_{\gamma-d-1}(\R^d)} \dd x \\
                          &\quad- \Cget \fraka^{1+\frac{2s}{d}} (t) \int_\Omega \norm{ \big((f-A)_+\wedge\kappa\big)^{q_0-1}}_{L^1_{l_q}(\R^d)} \dd x  \\
  & \quad +(C_1+\Chard) \fraka(t)\int_\Omega \norm{ \big((f-A)_+\wedge\kappa\big)^{q_0-1}}_{L^1_{l_q}(\R^d)} \dd x \\
	&\quad+ \varepsilon\int_\Omega \norm{ \big((f-A)_+\wedge\kappa\big)^{q_0}}_{L^p_{k_0}(\R^d)} \dd x+ \Chard C(\varepsilon) \int_\Omega \norm{ \big((f-A)_+\wedge\kappa\big)^{q_0-1}}_{L^1_{l_q}(\R^d)}\dd x\\
	&\quad -  q_0 \dot \fraka(t)\int_\Omega \norm{\big((f-A)_+\wedge\kappa\big)^{q_0-1}}_{L^1_{-q}(\R^d)} \dd x\\	
\intertext{we then pick $\eps = \Ccor$ and get,}
	&\leq  \bar C \left(-\bar c \fraka^{1+\frac{2s}{d}}(t) + 1 +\fraka(t) - \dot \fraka(t) \right)\int_\Omega \norm{ \big((f-A)_+\wedge\kappa\big)^{q_0-1}}_{L^1_{l_q}(\R^d)} \dd x
\end{align*}
for some positive constants $\bar C$ and $\bar c$ depending on $\Ccor$, $\Cget$, $\Chard$, $C(\Ccor)$, $C_1$, $d$ and $s$.       
The last inequality follows from the fact that $\gamma-d-1 \le l_q$ and $- q\leq l_q$ (recall that $\dot \ma<0$ and $l_q$ is given in \eqref{itm:lq}) for $q \leq d + 1$ and $\gamma \geq 0$.  Finally, we choose $\aast$ large enough (depending
on $\bar c$, $d$ and $s$) such that
\begin{equation}\label{eq:condition-a-2}
	- \bar c \, \fraka^{1+\frac{2s}{d}}(t) + 1 +\fraka(t) - \dot \fraka(t) \leq 0.
\end{equation}
In particular, this yields,
\(
	 \frac{\dd}{\dd t} \iint_{\R^d \times \Omega} \phizk\big((f-A)_+\big)\dv \dx  \leq 0
\)
and concludes the proof  for hard potentials ($\gamma \geq 0$).
\bigskip

\paragraph{\sc The case of soft potentials.}
In this case, we apply Lemma~\ref{l:weak-sols-q0} and collect estimates \eqref{eq:i-q0} with $g = (f-A)_+$, \eqref{eq:second-coercivity}, \eqref{eq:ii-q0}, \eqref{eq:iii-q0}, \eqref{eq:iv-gamma-neg-q0} and use Lemma~\ref{l:weak-sols-q0} to deduce,
\begin{align*}
  &\frac{\dd}{\dd t} \iint_{\R^d \times \Omega}  \phizk(f-A)\dv \dx \\
	&\leq - \Ccor \int_{\Omega}\norm{\big((f-A)_+\wedge\kappa\big)^{q_0}}_{L^{p_0}_{k_0}(\R^d)} \dd x+  \Ccor \int_{\Omega} \norm{\big((f-A)_+\wedge\kappa\big)^{q_0-1}}_{L^1_{l_q}(\R^d)} \dd x \\
	&\quad- \Cget \fraka^{1+\frac{2s}{d}}(t) \int_{\Omega}\norm{ \big((f-A)_+\wedge\kappa\big)^{q_0-1}}_{L^1_{l_q}(\R^d)} \dd x  +C_1 \fraka(t)\int_{\Omega}\norm{ \big((f-A)_+\wedge\kappa\big)^{q_0-1}}_{L^1_{l_q}(\R^d)} \dd x \\
	&\quad+ \frac34 \Ccor \int_\Omega \norm{ \big((f-A)_+\wedge\kappa\big)^{q_0}}_{L^{p_0}_{k_0}(\R^d)} \dd x+ \left( \frac{15}{16} \Cget \fraka^{1+\frac{2s}{d}} +\tilde C_5 \fraka \right) \int_\Omega \norm{ \big((f-A)_+\wedge\kappa\big)^{q_0-1}}_{L^1_{l_q}(\R^d)}\dd x\\
	&\quad+ C_5(\fraka)\int_{\Omega} \norm{(f-A-\kappa)_+}_{L^{p_0}_{k_0}(\R^d)}\dd x  -\dot \fraka(t)\int_{\Omega}\norm{\big((f-A)_+\wedge\kappa\big)^{q_0-1}}_{L^1_{-q}(\R^d)} \dd x\\
	&\quad	-\frac{1}{2} \Ccorbis \kappa^{\frac{2s}{d} q_0 }\int_{\Omega}\left[\norm{\big((f-A)_+\wedge \kappa\big)^{q_0-1}}_{L^1_{-(d-1)}}^{-\frac{2s}{d}} \int_{\R^d}  (f-A-\kappa)_+ \langle v \rangle^{\gamma} \dd v \right]\dd x, 
\end{align*}
Using that $-q  \leq l_q$ for $q \leq d+ 1 + \frac{\gamma d}{2s}$ (and in particular for $q = \qnsl$), we get,
\begin{align*}
  \frac{\dd}{\dd t} \iint_{\R^d \times \Omega}  \phizk(f-A)\dv \dx  
                    \leq & \bar C \left(- \bar c \fraka^{1+\frac{2s}{d}}(t) +\fraka(t) - \dot \fraka(t) \right)\int_{\Omega}\norm{ \big((f-A)_+\wedge\kappa\big)^{q_0-1}}_{L^1_{l_q}(\R^d)} \dd x\\
	 & +C_5(\fraka)\int_{\Omega} \norm{(f-A-\kappa)_+}_{L^{p_0}_{k_0}(\R^d)}\dd x.
       \end{align*}

Finally, we choose $\aast$ such that both \eqref{eq:condition-a-1} and \eqref{eq:condition-a-2} (for a new $\bar c$)  hold true and we get
\[
	\frac{\dd}{\dd t} \iint_{\R^d \times \Omega}  \phizk\big((f-A)_+\big)\dv \dd x  \leq \mathcal{I}_\kappa(t),
\]
with
\[
\mathcal{I}_\kappa(t) = C_5(\fraka)\int_{\Omega} \norm{(f-A-\kappa)_+}_{L^{p_0}_{k_0}(\R^d)}\dd x.
\]
Finally we notice that $f \in L^1([0, T] \times \Omega; L^{p_0}_{k_0}(\R^d))$ ensures that,
\[
	\int_0^T \mathcal{I}_\kappa(t) \dd t \to 0 \qquad \textrm{as } \kappa \to \infty.
\]
This concludes the proof in the case of soft potentials.
\end{proof}

\subsection{Proof of decay generation}
\label{sub:lebesgue-shift}

\begin{lemma}[Monotonicity implies decrease] \label{l:mono-decrease}
Let $f$ be a suitable weak sub-solution of the Boltzmann equation with either in-flow, bounce-back, specular / diffuse / Maxwell reflection boundary condition.  Let
\[
	A(t,v) := \fraka(t) \langle v \rangle^{-q} 
\]
for some smooth bounded function $\fraka \colon (0,T) \to (0,+\infty)$ and for some $q \geq 0$. 
Assume that 
\begin{equation}\label{e:monotone-gen}
	\frac{\dd}{\dd t} \iint_{\R^d \times \Omega} \phizk(f-A)\dv \dx \leq \mathcal{I}_\kappa(t) \quad \text{ in } \mathcal{D}'((0,T)),
\end{equation}
for some $\mathcal{I}_\kappa(t)$ such that for any $T>0$, we have
\[
	\int_0^{T}  \mathcal{I}_\kappa(t) \dd t \xrightarrow[\kappa \to \infty]{}0.
      \]
Then $f (t,x,v) \le A(t,v)$ almost everywhere in $(0,T) \times \Omega \times \R^d$.       
\end{lemma}
\begin{proof}
  For all $t \in (0,T)$, we define
\begin{equation}\label{eq:m-def}
	m_\kappa(t) := \iint_{\R^d \times \Omega} \phizk (f-A)(t,x,v) \dv \dx.
\end{equation}
Since $\phizk(f-A) \in L^1((0, T) \times \R^d \times \R^d)$, there holds $m_\kappa \in L^1((0, T))$.
\medskip

\paragraph{\sc Step 1. Lebesgue points.}
We now take two Lebesgue points $t_1, t_2 \in (0, T)$ such that $t_1 < t_2$, and we consider for any $\varepsilon > 0$ the following cutoff $\eta_\varepsilon\in C^\infty_c(\R)$ in time given by
\begin{equation*}
\begin{aligned}
	\eta_\varepsilon(t) = \begin{cases} \frac{1}{\varepsilon} (t_2 - t), \qquad &\text{ if } t \in [t_2 - \varepsilon, t_2), \\
	1, \qquad &\text{ if } t \in (t_1 +\varepsilon, t_2- \varepsilon), \\
	\frac{1}{\varepsilon} (t - t_1), \qquad &\text{ if } t \in (t_1, t_1+ \varepsilon], \\
	0, \qquad &\text{ else}.
	\end{cases}
\end{aligned}
\end{equation*}
Then we test \eqref{e:monotone} with $\eta_\varepsilon$ and integrate over $(0, T)$, so that
\[
  \int_{0}^T m_\kappa'(t) \eta_\varepsilon(t) \dd t \le \int_0^T \mathcal{I}_\kappa (t) \eta_\eps (t)\dt.
\]
The left hand side yields
\begin{equation*}
	\int_{0}^T m_\kappa'(t) \eta_\varepsilon(t) \dd t = - \int_{0}^T m(t) \eta_\varepsilon'(t) \dd t =\frac{1}{\varepsilon} \int_{t_2 - \varepsilon}^{t_2} m(t) \dd t - \frac{1}{\varepsilon}\int_{t_1}^{t_1+ \varepsilon} m(t) \dd t,
\end{equation*}
so that we deduce after taking $\varepsilon \to 0$, we deduce that for almost every $t_1, t_2 \in (0,T)$ with $t_1 <t_2$,
\begin{equation}\label{e:decroissance}
m_\kappa(t_2) - m_\kappa(t_1) \leq \int_{t_1}^{t_2} \mathcal{I}_\kappa(t) \dd t.
\end{equation}
Note that the right hand side converges to $0$ as $\kappa \to \infty$. It also converges to $0$ as $t_2 \to t_1$ or as $t_1 \to t_2$, so that $m_\kappa(t)$ is (coincides a.e. with a) c\`adl\`ag (function)
and the values of $m_\kappa(t)$ are well-defined for all $t \in (0, T)$, see \cite[Corollary~4.9.1]{ouyang-silvestre}. 
We deduce from \eqref{e:decroissance} that $m_\infty$ is non-increasing in $(0,T)$. 

\medskip

\paragraph{\sc Step 2. Shifting the barrier in time.}
To finish the proof of Theorem \ref{thm:decroissance}, it is sufficient to show that $m_\infty(t) = 0$ for a.e. $t \in (0,T)$. Indeed this in turn implies $f(t, x, v) \leq \fraka(t) \langle v \rangle^{-q}$ for almost every $(t, x, v) \in [0, T] \times \R^d \times \R^d$.

Even if $m_\infty$ is monotonically decreasing, the difficulty is that $m_\infty(t)$ is not defined at $t=0$. 
The following argument is taken from \cite[Proof of Theorem 1.1]{ouyang-silvestre}. We define a shifted version of the function $m$ defined in \eqref{eq:m-def},
\[
	m_{\kappa,t_1}(t) :=  \iint_{\R^d \times \Omega}\phizk\big(f(t, x, v) - A(t-t_1, v) \big) \dv \dx.
\]
After updating $\fraka$,  we know from Step~1 that $m_{\infty, t_1}(t)$ is monotonically decreasing for $t \in (t_1, T)$. Since $\fraka(t) \to \infty$ as $t \to 0$, we note
\[
	\lim_{t_1 \to t_2} m_{\infty,t_1}(t_2) = 0. 
\]
Thus for any $\varepsilon \in (0, 1)$ there exists $t_1 \in (0, t_2)$ such that $m_{\infty,t_1}(t_2) < \varepsilon$. 

Moreover, there holds for a.e. $t \in (t_2, T)$
\begin{equation*}
	m_{\infty,t_2}(t) \leq m_{\infty,t_1}(t) \leq m_{\infty,t_1}(t_2)  < \varepsilon.
\end{equation*}
As $\varepsilon > 0$ was arbitrary, we deduce $m_{\infty,t_2}(t) = 0$ for a.e. $t \in (t_2, T)$. This implies that for a.e. $t_2 \in (0,T)$ and $t \in (t_2, T)$, we have  $(x,v) \in \Omega \times \R^d$,
$f(t,x,v) \le A (t-t_2,v) $. We can now fix $t_0 >0$ and consider $t_2 < t_0<t$ and outside the sets of null measure.
By letting $t_2 \to 0$, this implies $f(t,x,v) \le A (t,v)$ for $t > t_0$. Since $t_0$ is arbitrary, the proof is complete.  
\end{proof}

\begin{proof}[Proof of Theorem~\ref{thm:decroissance}]
It is enough to deal with the case $q = \qnsl$. In this case, the conclusion of the theorem is a consequence of Lemmas~\ref{l:decroissance} and \ref{l:mono-decrease}.
\end{proof}


\section{Counterexample to arbitrary polynomial decay for soft potentials}

This section is devoted to the proof of Theorem \ref{thm:no-gen}. We show that in the case of soft potentials, we cannot expect to generate decay at a rate higher than $d+2$.

\begin{proof}[Proof of Theorem~\ref{thm:no-gen}]
We consider a non-negative classical solution  $f$ to \eqref{eq:boltzmann} in $[0, T] \times \Omega \times \R^d$ for some $T > 0$, with initial datum $\fin$ satisfying \eqref{eq:f0-lb} and with boundary data $f_b$ satisfying \eqref{eq:fb-lb}. 

We are going to construct a barrier from below for $f$. We claim that we can find a $\beta > 0$ and $a_1 > 0$ for which
\[
\forall (t, x, v)\in [0, T] \times \R^d \times \R^d, \qquad	f(t, x, v) \geq a_1 e^{-\beta t} \langle v \rangle^{- \hat q}.
\]
In order to prove that such a lower bound holds, we first establish the following lemma. 
\begin{lemma}[Construction of the barrier] \label{l:barrier}
  There exists a constant $\beta = \beta(d,\gamma,s,E_0,M_0,\tilde b_+)>0$ such that the function $\Psi$ defined by
  \[
	\Psi(t, x, v) := e^{-\beta t}\langle v \rangle^{-\hat{q}}.
      \]
      satisfies
\[
	\left(\partial_t + v \cdot \nabla_x\right) \Psi < 	\int_{\R^d} \left(\Psi(v') - \Psi(v)\right)K_f(v, v') \dd v'. 
\]
\end{lemma}
\begin{proof}
We first remark that for all $(t, x, v) \in [0, T] \times \R^d \times \R^d$,
\[
	\left(\partial_t + v \cdot \nabla_x\right) \Psi = -\beta \Psi.
\]
We are thus left with showing the following upper bound, 
\[
\int_{\R^d} \left(\Psi(v) - \Psi(v')\right)K_f(v, v') \dd v' \le \beta \Psi. 
\]
We compute
\begin{align*}
	\int_{\R^d}& \left(\Psi(v) - \Psi(v')\right)K_f(v, v') \dd v' \\
	&= \int_{\R^d}   \frac{\left(\Psi(v) - \Psi(v')\right)}{\abs{z}^{d+2s}} \left\{ \int_{v_*' \in v+ z^\perp} f(v_*') \abs{v - v_*'}^{\gamma +2s +1} \tilde b(\cos \theta) \dd v_*' \right\} \dd z\\
	&= \int_{\R^d}   f(v_*') \abs{v - v_*'}^{\gamma +2s} \left\{ \int_{ z \in (v-v_*')^\perp} \frac{\left(\Psi(v) - \Psi(v')\right)}{\abs{z}^{d+2s-1}} \tilde b(\cos \theta) \dd z \right\} \dd v_*'.
\end{align*}

Once again, we distinguish the singular from the non-singular part,
\begin{align*}
	\mathcal I_{s} := \int_{\R^d}   f(v_*') \abs{v - v_*'}^{\gamma +2s} \left\{ \int_{ z \in (v-v_*')^\perp} \chi_{\abs{z} \leq \frac12 \langle v\rangle}\frac{\left(\Psi(v) - \Psi(v')\right)}{\abs{z}^{d+2s-1}} \tilde b(\cos \theta) \dd z \right\} \dd v_*',
\end{align*}
and
\begin{align*}
	\mathcal I_{ns} := \int_{\R^d}   f(v_*') \abs{v - v_*'}^{\gamma +2s} \left\{ \int_{ z \in (v-v_*')^\perp} \chi_{\abs{z} > \frac12 \langle v\rangle}\frac{\left(\Psi(v) - \Psi(v')\right)}{\abs{z}^{d+2s-1}} \tilde b(\cos \theta) \dd z \right\} \dd v_*'.
\end{align*}
We showed in Lemma \ref{lem:s} the following upper bound for $\mathcal{I}_s$,
\[
	\mathcal I_{s}\leq  \Cqd \langle v \rangle^{\gamma} \Psi(v).
\]
To bound $\mathcal I_{ns}$, we find
\begin{align*}
	\mathcal I_{ns} &= \int_{\R^d}   f(v_*') \abs{v - v_*'}^{\gamma +2s} \left\{ \int_{ z \in (v-v_*')^\perp} \chi_{\abs{z} > \frac12 \langle v\rangle}\frac{\left(\Psi(v) - \Psi(v')\right)}{\abs{z}^{d+2s-1}} \tilde b(\cos \theta) \dd z \right\} \dd v_*'\\
	&\leq \int_{\R^d}   f(v_*') \abs{v - v_*'}^{\gamma +2s} \left\{ \int_{ z \in (v-v_*')^\perp} \chi_{\abs{z} > \frac12 \langle v\rangle}\frac{\Psi(v)}{\abs{z}^{d+2s-1}} \tilde b(\cos \theta) \dd z \right\} \dd v_*'\\
	&\leq \tilde b_+ e^{-\beta t} \langle v \rangle^{-\hat q} \int_{\R^d}   f(v_*') \abs{v - v_*'}^{\gamma +2s} \left\{ \int_{ z \in (v-v_*')^\perp} \chi_{\abs{z} > \frac12 \langle v\rangle}\frac{1}{\abs{z}^{d+2s-1}}\dd z \right\} \dd v_*'\\
	&\leq 2\tilde b_+\omega_{d-2} e^{-\beta t} \langle v \rangle^{-\hat q} \int_{\R^d}   f(v_*') \left(\abs{v}^{\gamma +2s} +\abs{v_*'}^{\gamma +2s}\right) \left\{\int_{\frac12 \langle v\rangle}^\infty r^{-1-2s} \dd r \right\} \dd v_*'\\
	&\leq 2\tilde b_+\omega_{d-2} \frac{2^{2s}}{2s} e^{-\beta t} \langle v \rangle^{-\hat q-2s} \int_{\R^d}   f(v_*') \left(\abs{v}^{\gamma +2s} +\abs{v_*'}^{\gamma +2s}\right)  \dd v_*'\\
  \intertext{using that $\gamma + 2s \in [0,  2]$ and \eqref{e:hydro}} 
                        & \leq \beta e^{-\beta t} \langle v \rangle^{-\hat q-2s} \langle v\rangle^{\gamma+2s}\\
                        & \le  \beta  \Psi(v)
\end{align*}
for some constant $\beta = \beta(d,\gamma,s,E_0,M_0,\tilde b_+)$. We used $\gamma<0$ to get the last line. 
\end{proof}

In order to prove the lower bound, we remark that $f > a_1 \Psi$ at initial time provided that $a_1 > a_0$. Similarly, $f > a_1 \Psi$ on $(0,T) \times \partial \Omega \times \R^d$
provided that $a_1 > a_b$. Assume that $\min_{[0,T] \times \Omega \times \R^d}  (f-a_1 \Psi) <0$. Since $f$ and $\Psi$ are continuous, there exists $(t_*,x_*,v_*) \in [0,T] \times \Omega \times \R^d$ where this minimum is reached. We know that $t_* >0$ and $x_* \in \Omega$ since $f > a_1 \Psi$ initially and on $\partial\Omega$. In particular,
\[ (\partial_t + v \cdot \nabla_x) f (t_*,x_*,v_*) \le  a_1 (\partial_t + v \cdot \nabla_x) \Psi (t_*,x_*,v_*)  \]
(with equality unless $t_* = T$) and
\begin{align*}
  Q (f,f) (t_*,x_*,v_*)  & =  \int_{\R^d} (f(t_*,x_*,v') - f(t_*,x_*,v_*)) K_f (v_*,v') \dv' + c_b (f \ast |\cdot|^\gamma) f (t_*,x_*,v_*)\\
                         & \ge \int_{\R^d} (a_1 \Psi (t_*,x_*,v') - a_1 \Psi(t_*,x_*,v_*)) K_f (v_*,v') \dv' \\
  & > a_1 (\partial_t + v \cdot \nabla_x) \Psi (t_*,x_*,v_*) .
\end{align*}
Since $f$ is a solution of the Boltzmann equation, we get a contradiction. 
\end{proof}


\end{document}